\documentclass[a4paper,12pt]{article}

\usepackage{amsmath}
\usepackage{amsthm}
\usepackage{dsfont}
\usepackage{suetterl}
\usepackage{adforn}
\usepackage{calligra}
\usepackage{pbsi}
\usepackage{aurical}
\usepackage[T1]{fontenc}
\usepackage{textgreek}

\usepackage[utf8]{inputenc}
\usepackage{makeidx}
\usepackage{enumerate}
\usepackage{longtable}
\usepackage{pb-diagram}
\usepackage{multirow}

\usepackage{tikz}
\usetikzlibrary{arrows,snakes,backgrounds}
\usetikzlibrary{matrix}

\usepackage{calc}
\usepackage{ifthen}
\usepackage{latexsym}
\usepackage{amssymb}
\usepackage{amsfonts}
\usepackage[mathscr]{eucal}
\usepackage{amscd}
\usepackage{array}
\usepackage{amsrefs}

\DeclareMathOperator{\ch}{char} \DeclareMathOperator{\End}{End}

\theoremstyle{plain}
 \newtheorem{theorem}{Theorem}
 \newtheorem{lemma}[theorem]{Lemma}
 \newtheorem{proposition}[theorem]{Proposition}
 \newtheorem{corollary}[theorem]{Corollary}

\theoremstyle{remark}
 \newtheorem{remark}[theorem]{Remark}

\theoremstyle{definition}

 \newtheorem{definition}[theorem]{Definition}

\tikzstyle{point}=[circle,draw,inner sep=0pt,minimum size=4.5mm]
\tikzstyle{ordinario}=[circle,draw,inner sep=0pt,minimum size=2mm]
\tikzstyle{grande}=[circle,draw,inner sep=0pt,minimum size=4.5mm]

\makeindex

\begin{document}

\title{Representations and corepresentations of $p$-equipped posets}

\author{Raymundo Bautista \\ \textit{\small{raymundo@matmor.unam.mx}} \\ 
\text{\small{Centro de Ciencias Matem\'aticas, Morelia}}\\ \text{\small{Universidad Nacional Aut\'onoma de M\'exico}}\\ \\
\text{Ivon Dorado} \\ \textit{\small{iadoradoc@unal.edu.co}} 
\\ \text{\small{Departamento de Matem\'aticas, Bogot\'a}}\\ \text{\small{Universidad Nacional de Colombia}} }
\date{}
\maketitle


\begin{abstract}
\noindent {\footnotesize For $p$ a prime number and $\mathscr{P}$ a $p$-equipped finite partially ordered set we construct two different right-peak algebras (in the sense of \cite{KS}) $\Lambda^{(r)}$ and $\Lambda^{(c)}$. We consider the category $\mathcal{U}^{(r)}$ $\left(\mathcal{U}^{(c)}\right)$ consisting of the finitely generated right $\Lambda^{(r)}$-modules ($\Lambda^{(c)}$-modules) which are socle-projective. The categories $\mathcal{U}^{(r)}$ and $\mathcal{U}^{(c)}$ have almost split sequences. We describe he Auslander-Reiten components $\mathcal{C}_{\mathcal{U}}^{(r)}$ and $\mathcal{C}_{\mathcal{U}}^{(c)}$ of the corresponding simple projective modules in $\mathcal{U}^{(r)}$ and $\mathcal{U}^{(c)}$. Then we prove that there is a bijective correspondence between $\mathcal{C}_{\mathcal{U}}^{(r)}$ and $\mathcal{C}_{\mathcal{U}}^{(c)}$, although the corresponding almost split sequences have different shapes.}
\end{abstract}

 {\small {\it Mathematics Subject Classification}: 06A11, 16G20, 16G30.\\

 {\small {\it Keywords}: Equipped poset, Algebraically equipped poset,  Representation, Corepresentation, Morphism category, Pseudo hereditary projective module, Auslander-Reiten component.


\section{Introduction}\label{intro}

In the following $F$ is an arbitrary field. For certain pairs of hereditary $F$-algebras $C$ and $D$, one has that
 there are bijective correspondences between the representations of the preprojective component of $C\mathrm{mod}$ and the
 preprojective component of $D\mathrm{mod}$ and the representations of the preinjective components of $C\mathrm{mod}$ and the preinjective component of $D\mathrm{mod}$.  
 
 For example, consider the quiver
 $$
  \begin{tikzpicture}
   \node (s) at (0,0) [ordinario] {};
      \node (w) at (1,0) [ordinario] {};
   \draw [<-] (w)to(s);
  \end{tikzpicture}$$

For a prime number $p$, there are two algebras given in terms of a normal extension $G$ of $F$ of degree $p$.

 $$
  \begin{tikzpicture}
   \node (wR) at (0,0) {$G$};
      \node (sR) at (2,0) {$F$};
      \node (a) at (1,0.3) {$G$};
   \draw [->] (wR)to(sR);
      \node (wC) at (4,0) {$F$};
      \node (sC) at (6,0) {$G$};
      \node (a) at (5,0.3) {$G$};
   \draw [->] (wC)to(sC);
  \end{tikzpicture}$$

In general, we have pairs of algebras determined by the field extension $G/F$. They are in fact matrix algebras in which appear only the bimodules $_{F}F_{F} $, $_{F}G_{G}$, $_{G}G_{F}$, $_{G}G_{G}$, and one algebra of the pair is obtained by the other simply exchanging $F$ and $G$. For the previous example we have the pair
 $$C=\begin{bmatrix}
 G & G \\
 0 & F
 \end{bmatrix}
 \hspace{2cm}
D= \begin{bmatrix}
 F & G \\
 0 & G
 \end{bmatrix}.$$
 
These algebras are of finite representation type if and only if $p\in \{ 2,3\}$ (see \cite{KS}). In this case, the whole categories $C\mathrm{mod}$ and  $D\mathrm{mod}$ are in a bijective correspondence. The dimension of their modules over $F$, are vectors in $\mathbb{Z}^2$. The rows of each matrix correspond to the projective modules and the lower row is in the radical of the other. By using the dimension vectors and the Coxeter matrix, one can draw the respective Auslander-Reiten quivers.

For $p=2$
 $$
  
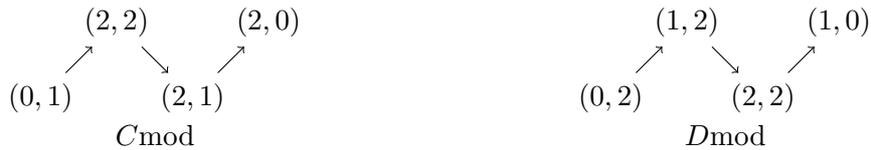
\begin{tikzpicture}
  \node (P1) at (0,0) {$(0,1)$};
  \node (P2) at (1,1) {$(2,2)$};
  \draw [->] (P1)to(P2);
  \node (TP1) at (2,0) {$(2,1)$};
  \node (TP2) at (3,1) {$(2,0)$};
  \draw [->] (TP1)to(TP2);
  \draw [->] (P2)to(TP1);
  \node (name) at (1.5,-0.5) {$C\mathrm{mod}$};
  \node (P1) at (7.5,0) {$(0,2)$};
  \node (P2) at (8.5,1) {$(1,2)$};
  \draw [->] (P1)to(P2);
  \node (TP1) at (9.5,0) {$(2,2)$};
  \node (TP2) at (10.5,1) {$(1,0)$};
  \draw [->] (TP1)to(TP2);
  \draw [->] (P2)to(TP1);
  \node (name) at (9,-0.5) {$D\mathrm{mod}$};
  \end{tikzpicture}$$
For $p=3$ 
 
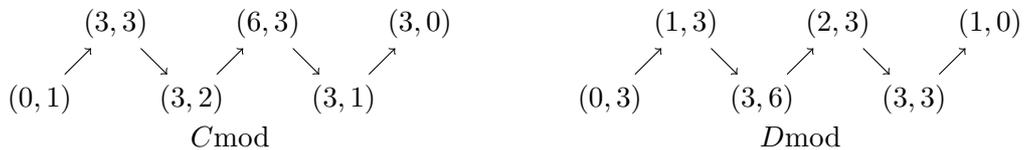
$$
  \begin{tikzpicture}
  \node (P1) at (0,0) {$(0,1)$};
  \node (P2) at (1,1) {$(3,3)$};
  \draw [->] (P1)to(P2);
  \node (TP1) at (2,0) {$(3,2)$};
  \node (TP2) at (3,1) {$(6,3)$};
  \draw [->] (TP1)to(TP2);
  \draw [->] (P2)to(TP1);
  \node (T2P1) at (4,0) {$(3,1)$};
  \node (T2P2) at (5,1) {$(3,0)$};
  \draw [->] (T2P1)to(T2P2);
  \draw [->] (TP2)to(T2P1);
  \node (name) at (2.5,-0.5) {$C\mathrm{mod}$};
  \node (P1) at (7.5,0) {$(0,3)$};
  \node (P2) at (8.5,1) {$(1,3)$};
  \draw [->] (P1)to(P2);
  \node (TP1) at (9.5,0) {$(3,6)$};
  \node (TP2) at (10.5,1) {$(2,3)$};
  \draw [->] (TP1)to(TP2);
  \draw [->] (P2)to(TP1);
  \node (T2P1) at (11.5,0) {$(3,3)$};
  \node (T2P2) at (12.5,1) {$(1,0)$};
  \draw [->] (T2P1)to(T2P2);
  \draw [->] (TP2)to(T2P1);
  \node (name) at (10,-0.5) {$D\mathrm{mod}$};
  \end{tikzpicture}$$

Consider the pair
 $$
  \begin{tikzpicture}
   \node (w1R) at (0,0) {$G$};
      \node (w2R) at (2,0) {$G$};
      \node (a) at (1,0.3) {$G$};
   \draw [->] (w1R)to(w2R);
      \node (sR) at (4,0) {$F$};
      \node (a) at (3,0.3) {$G$};
   \draw [->] (w2R)to(sR);
      \node (w1C) at (6,0) {$F$};
      \node (w2C) at (8,0) {$F$};
      \node (a) at (7,0.3) {$F$};
   \draw [->] (w1C)to(w2C);
   \node (sC) at (10,0) {$G$};
      \node (b) at (9,0.3) {$G$};
   \draw [->] (w2C)to(sC);
  \end{tikzpicture}$$

 $$C=\begin{bmatrix}
 G & G & G \\
 0 & G & G \\
 0 & 0 & F
 \end{bmatrix}
 \hspace{3cm}
D= \begin{bmatrix}
 F & F & G \\
 0 & F & G \\
 0 & 0 & G
 \end{bmatrix}.$$
 
 If $p=3$, these algebras are of tame representation type. The preprojective components of their Auslander-Reiten quivers have the form
 $$
  \begin{tikzpicture}
  \node (P1) at (0,0) {$(0,0,1)$};
  \node (P2) at (1,1) {$(0,3,3)$};
  \node (P3) at (2,2) {$(3,3,3)$};
  \draw [->] (P1)to(P2);
  \draw [->] (P2)to(P3);
  \node (tP1) at (2,0) {$(0,3,2)$};
  \node (tP2) at (3,1) {$(3,9,6)$};
  \node (tP3) at (4,2) {$(0,6,3)$};
  \draw [->] (tP1)to(tP2);
  \draw [->] (tP2)to(tP3);
  \draw [->] (P2)to(tP1);
  \draw [->] (P3)to(tP2);
  \node (t2P1) at (4,0) {$(3,6,4)$};
  \node (t2P2) at (5,1) {$(6,15,9)$};
  \node (t2P3) at (6,2) {$(6,9,6)$};
  \draw [->] (t2P1)to(t2P2);
  \draw [->] (t2P2)to(t2P3);
  \draw [->] (tP2)to(t2P1);
  \draw [->] (tP3)to(t2P2);
  \node (t3P1) at (6,0) {$(3,9,5)$};
  \node (t3P2) at (7,1) {$(9,21,12)$};
  \node (t3P3) at (8,2) {$(3,12,6)$};
  \draw [->] (t3P1)to(t3P2);
  \draw [->] (t3P2)to(t3P3);
  \draw [->] (t2P2)to(t3P1);
  \draw [->] (t2P3)to(t3P2);
  \node (dots1) at (7.5,0) {$\cdots$};
  \node (dots2) at (8.5,1) {$\cdots$};
  \node (dots3) at (9.5,2) {$\cdots$};  
  \end{tikzpicture}$$

 $$
  \begin{tikzpicture}
  \node (P1) at (0,0) {$(0,0,3)$};
  \node (P2) at (1,1) {$(0,1,3)$};
  \node (P3) at (2,2) {$(1,1,3)$};
  \draw [->] (P1)to(P2);
  \draw [->] (P2)to(P3);
  \node (tP1) at (2,0) {$(0,3,6)$};
  \node (tP2) at (3,1) {$(1,3,6)$};
  \node (tP3) at (4,2) {$(0,2,3)$};
  \draw [->] (tP1)to(tP2);
  \draw [->] (tP2)to(tP3);
  \draw [->] (P2)to(tP1);
  \draw [->] (P3)to(tP2);
  \node (t2P1) at (4,0) {$(3,6,12)$};
  \node (t2P2) at (5,1) {$(2,5,9)$};
  \node (t2P3) at (6,2) {$(2,3,6)$};
  \draw [->] (t2P1)to(t2P2);
  \draw [->] (t2P2)to(t2P3);
  \draw [->] (tP2)to(t2P1);
  \draw [->] (tP3)to(t2P2);
  \node (t3P1) at (6,0) {$(3,9,15)$};
  \node (t3P2) at (7,1) {$(3,7,12)$};
  \node (t3P3) at (8,2) {$(1,4,6)$};
  \draw [->] (t3P1)to(t3P2);
  \draw [->] (t3P2)to(t3P3);
  \draw [->] (t2P2)to(t3P1);
  \draw [->] (t2P3)to(t3P2);
  \node (dots1) at (7.5,0) {$\cdots$};
  \node (dots2) at (8.5,1) {$\cdots$};
  \node (dots3) at (9.5,2) {$\cdots$};  
  \end{tikzpicture}$$

Changing the orientation, we have
 $$
  \begin{tikzpicture}
   \node (w1R) at (0,0) {$G$};
      \node (w2R) at (2,0) {$F$};
      \node (a) at (1,0.3) {$G$};
   \draw [->] (w1R)to(w2R);
      \node (sR) at (4,0) {$F$};
      \node (a) at (3,0.3) {$F$};
   \draw [<-] (w2R)to(sR);
      \node (w1C) at (6,0) {$F$};
      \node (w2C) at (8,0) {$G$};
      \node (a) at (7,0.3) {$G$};
   \draw [->] (w1C)to(w2C);
   \node (sC) at (10,0) {$G$};
      \node (b) at (9,0.3) {$G$};
   \draw [<-] (w2C)to(sC);
  \end{tikzpicture}$$

 $$C=\begin{bmatrix}
 G & 0 & G \\
 0 & F & F \\
 0 & 0 & F
 \end{bmatrix}
 \hspace{3cm}
D= \begin{bmatrix}
 F & 0 & G \\
 0 & G & G \\
 0 & 0 & G
 \end{bmatrix};$$
 with the following preprojective components for $p=3$
  $$
  \begin{tikzpicture}
  \node (P2) at (1,1) {$(0,0,1)$};
  \node (P3) at (2,2) {$(3,0,3)$};
  \draw [->] (P2)to(P3);
  \node (tP1) at (2,0) {$(0,1,1)$};
  \node (tP2) at (3,1) {$(3,1,3)$};
  \node (tP3) at (4,2) {$(6,3,6)$};
  \draw [->] (tP1)to(tP2);
  \draw [->] (tP2)to(tP3);
  \draw [->] (P2)to(tP1);
  \draw [->] (P3)to(tP2);
  \node (t2P1) at (4,0) {$(3,0,2)$};
  \node (t2P2) at (5,1) {$(6,2,5)$};
  \node (t2P3) at (6,2) {$(12,3,9)$};
  \draw [->] (t2P1)to(t2P2);
  \draw [->] (t2P2)to(t2P3);
  \draw [->] (tP2)to(t2P1);
  \draw [->] (tP3)to(t2P2);
  \node (t3P1) at (6,0) {$(3,2,3)$};
  \node (t3P2) at (7,1) {$(9,3,7)$};
  \node (t3P3) at (8,2) {$(15,6,12)$};
  \draw [->] (t3P1)to(t3P2);
  \draw [->] (t3P2)to(t3P3);
  \draw [->] (t2P2)to(t3P1);
  \draw [->] (t2P3)to(t3P2);
  \node (P1) at (8,0) {$(6,1,4)$};
  \draw [->] (t3P2)to(P1);  
  \node (dots1) at (9.5,0) {$\cdots$};
  \node (dots2) at (8.5,1) {$\cdots$};
  \node (dots3) at (9.5,2) {$\cdots$};  
  \end{tikzpicture}$$
 
 $$
  \begin{tikzpicture}
  \node (P2) at (1,1) {$(0,0,3)$};
  \node (P3) at (2,2) {$(1,0,3)$};
  \draw [->] (P2)to(P3);
  \node (tP1) at (2,0) {$(0,3,3)$};
  \node (tP2) at (3,1) {$(3,3,9)$};
  \node (tP3) at (4,2) {$(2,3,6)$};
  \draw [->] (tP1)to(tP2);
  \draw [->] (tP2)to(tP3);
  \draw [->] (P2)to(tP1);
  \draw [->] (P3)to(tP2);
  \node (t2P1) at (4,0) {$(3,0,6)$};
  \node (t2P2) at (5,1) {$(6,6,15)$};
  \node (t2P3) at (6,2) {$(4,3,9)$};
  \draw [->] (t2P1)to(t2P2);
  \draw [->] (t2P2)to(t2P3);
  \draw [->] (tP2)to(t2P1);
  \draw [->] (tP3)to(t2P2);
  \node (t3P1) at (6,0) {$(3,6,9)$};
  \node (t3P2) at (7,1) {$(9,9,21)$};
  \node (t3P3) at (8,2) {$(5,6,12)$};
  \draw [->] (t3P1)to(t3P2);
  \draw [->] (t3P2)to(t3P3);
  \draw [->] (t2P2)to(t3P1);
  \draw [->] (t2P3)to(t3P2);
  \node (P1) at (8,0) {$(6,3,12)$};
  \draw [->] (t3P2)to(P1);  
  \node (dots1) at (9.5,0) {$\cdots$};
  \node (dots2) at (8.5,1) {$\cdots$};
  \node (dots3) at (9.5,2) {$\cdots$};  
  \end{tikzpicture}$$

We have a similar correspondence between the preprojective components of the representations of the following wild algebras

 $$
  \begin{tikzpicture}
   \node (w1R) at (0,0) {$G$};
      \node (w2R) at (2,0) {$F$};
      \node (a) at (1,0.3) {$G$};
   \draw [->] (w1R)to(w2R);
      \node (sR) at (4,0) {$G$};
      \node (a) at (3,0.3) {$G$};
   \draw [<-] (w2R)to(sR);
      \node (w1C) at (6,0) {$F$};
      \node (w2C) at (8,0) {$G$};
      \node (a) at (7,0.3) {$G$};
   \draw [->] (w1C)to(w2C);
   \node (sC) at (10,0) {$F$};
      \node (b) at (9,0.3) {$G$};
   \draw [<-] (w2C)to(sC);
  \end{tikzpicture}$$

 $$C=\begin{bmatrix}
 G & 0 & G \\
 0 & G & G \\
 0 & 0 & F
 \end{bmatrix}
 \hspace{3cm}
D= \begin{bmatrix}
 F & 0 & G \\
 0 & F & G \\
 0 & 0 & G
 \end{bmatrix};$$
 $$
  \begin{tikzpicture}
  \node (P2) at (1,1) {$(0,0,1)$};
  \node (P3) at (2,2) {$(3,0,3)$};
  \draw [->] (P2)to(P3);
  \node (tP1) at (2,0) {$(0,3,3)$};
  \node (tP2) at (3,1) {$(3,3,5)$};
  \node (tP3) at (4,2) {$(6,9,12)$};
  \draw [->] (tP1)to(tP2);
  \draw [->] (tP2)to(tP3);
  \draw [->] (P2)to(tP1);
  \draw [->] (P3)to(tP2);
  \node (t2P1) at (4,0) {$(9,6,12)$};
  \node (t2P2) at (5,1) {$(12,12,19)$};
  \node (t2P3) at (6,2) {$(30,27,45)$};
  \draw [->] (t2P1)to(t2P2);
  \draw [->] (t2P2)to(t2P3);
  \draw [->] (tP2)to(t2P1);
  \draw [->] (tP3)to(t2P2);
  \node (t3P1) at (6,0) {$(27,30,45)$};
  \node (t3P2) at (7,1) {$(45,45,71)$};
  \node (t3P3) at (8.3,2) {$(105,108,198)$};
  \draw [->] (t3P1)to(t3P2);
  \draw [->] (t3P2)to(t3P3);
  \draw [->] (t2P2)to(t3P1);
  \draw [->] (t2P3)to(t3P2);
  \node (P1) at (8.3,0) {$(108,105,198)$};
  \draw [->] (t3P2)to(P1);  
  \node (dots1) at (10,0) {$\cdots$};
  \node (dots2) at (8.5,1) {$\cdots$};
  \node (dots3) at (10,2) {$\cdots$};  
  \end{tikzpicture}$$

 $$
  \begin{tikzpicture}
  \node (P2) at (1,1) {$(0,0,3)$};
  \node (P3) at (2,2) {$(1,0,3)$};
  \draw [->] (P2)to(P3);
  \node (tP1) at (2,0) {$(0,1,3)$};
  \node (tP2) at (3,1) {$(3,3,15)$};
  \node (tP3) at (4,2) {$(2,3,12)$};
  \draw [->] (tP1)to(tP2);
  \draw [->] (tP2)to(tP3);
  \draw [->] (P2)to(tP1);
  \draw [->] (P3)to(tP2);
  \node (t2P1) at (4,0) {$(3,2,12)$};
  \node (t2P2) at (5,1) {$(12,12,57)$};
  \node (t2P3) at (6,2) {$(10,9,45)$};
  \draw [->] (t2P1)to(t2P2);
  \draw [->] (t2P2)to(t2P3);
  \draw [->] (tP2)to(t2P1);
  \draw [->] (tP3)to(t2P2);
  \node (t3P1) at (6,0) {$(9,10,45)$};
  \node (t3P2) at (7,1) {$(45,45,213)$};
  \node (t3P3) at (8,2) {$(35,36,198)$};
  \draw [->] (t3P1)to(t3P2);
  \draw [->] (t3P2)to(t3P3);
  \draw [->] (t2P2)to(t3P1);
  \draw [->] (t2P3)to(t3P2);
  \node (P1) at (8,0) {$(36,35,198)$};
  \draw [->] (t3P2)to(P1);  
  \node (dots1) at (9.5,0) {$\cdots$};
  \node (dots2) at (8.5,1) {$\cdots$};
  \node (dots3) at (9.5,2) {$\cdots$};  
  \end{tikzpicture}$$
  
 The main purpose of this paper is to prove a similar relation between two  different algebras $\Lambda ^{(r)}$
 and $\Lambda ^{(c)}$ associated to the same $p$-equipped poset $\mathscr{P}$ (see definition bellow) and to a
 normal extension of fields $G/F$, with $[G:F]=p$. Both algebras are $1$-Gorenstein and right and left pick algebras in the sense of \cite{KS}.
 Denote by $\mathcal{U}^{(r)}$ and $\mathcal{U}^{(c)}$, the corresponding full subacategories of the category  of right modules whose objects are the finitely generated right modules which are socle-projective. It is known from \cite{BM}, that the categories $\mathcal{U}^{(r)}$ and $\mathcal{U}^{(c)}$ have almost split sequences. Denote by $\mathcal{C}_{\mathcal{U}}^{(r)}$, and $\mathcal{C}_{\mathcal{U}}^{(c)}$ the corresponding components of the simple projective module in the
 Auslander-Reiten quiver of $\mathcal{U}^{(r)}$ and $\mathcal{U}^{(c)}$ respectively. We will see that there is a bijective correspondence $\alpha:\mathcal{C}_{\mathcal{U}}^{(r)}\rightarrow \mathcal{C}_{\mathcal{U}}^{(c)}$, such that if
$X,Y\in \mathcal{C}_{\mathcal{U}}^{(r)}$ then there is an irreducible morphism $X\rightarrow Y$ if and only if there is
an irreducible morphism $\alpha(X)\rightarrow \alpha(Y)$.  Moreover $X$ is a projective (injective) object in 
$\mathcal{C}_{\mathcal{U}}^{(r)}$ if and only if $\alpha(X)$ is a projective (injective) object in the category $\mathcal{C}_{\mathcal{U}}^{(c)}$. However if $X$ is not injective, the almost split sequence starting in $X$ has a different shape
than the almost split sequence starting at $\alpha(X)$.


\section{$p$-equipped posets and its representations}\label{SectionPequipped}


In order to establish the objects we want to deal with, let us recall some definitions and results from \cite{BD2017}.

We introduced $p$-equipped posets, for a prime number $p$, as follows:

A finite poset $(\mathscr{P},\leq )$ is called \textit{p-equipped} if to every pair $(x,y)\in\  \leq$, it is assigned a value $\ell\in \{1,2,\dots,p\}$, with the notation $x\leq ^\ell y$, and the following condition holds:

\begin{equation}\label{definition}
\text{If }x\leq ^\ell y\leq ^mz\ \text{ and  }x\leq ^nz\text{, then }n\geq \min \{ \ell+m-1, p \}.
\end{equation}

A relation $x\leq y$ is called \textit{weak} or \textit{strong}, if $x\leq ^1 y$ or $x\leq ^p y$, respectively. It follows that the composition of a strong relation with any other relation is strong. 

For each point $x\in \mathscr{P}$, we have that $x\leq^\ell x\leq^\ell x$ implies $\ell \geq \min \{2\ell -1,p\}$. This is $x\leq ^1 x$, or $x\leq ^p x$. In the first case, we call $x$ a \textit{weak} point, and in the second one \textit{strong}. A relation between an arbitrary point and a strong point is always strong.

We write $x<^\ell y$, if $x\leq^\ell y$ and $x\neq y$. In
particular, if $\ell\in \{2,3,\dots,p-1\}$ we have $x<^\ell y$.

If a $p$-equipped poset $\mathscr{P}$ is \textit{trivially equipped}, i. e. it contains only strong points, then it is an \textit{ordinary} poset.

The next definition lead us to "copy" the structure of a $p$-equipped poset into a collection of subspaces of an algebra. 

\begin{definition}\label{sisAdm}
Let $\textsf{K}$ be a field, $A$ be a $\textsf{K}$-algebra and $\mathscr{P}$ be a partially ordered set with a maximal element {\large \Fontlukas m}. An \textit{admissible system} $\mathcal{R}=\{\mathcal{R}_{i,j}\}_{i\leq j}$, is a collection 
$$\mathcal{R}=\{\mathcal{R}_{i,j}\subseteq A \ | \ i,j\in \mathscr{P} \ \text{ such that } i\leq j,\} , $$
of $\textsf{K}$-vector subspaces $\mathcal{R}_{i,j}$ of $A$, that satisfy the following conditions:
\begin{enumerate}[\text{A}.1]
\item For every $i,j,l\in \mathscr{P}$ with $i\leq j\leq l$, we have $\mathcal{R}_{i,j}\mathcal{R}_{j,l}\subseteq \mathcal{R}_{i,l}$.
\item For every $i\in \mathscr{P}$, the space $\mathcal{R}_i=\mathcal{R}_{i,i}$ is a division $\textsf{K}$-ring, with unit $1_i$, such that for all $r\in \mathcal{R}_{i,j}$ it holds $1_ir=r=r1_j$, for each $j\in \mathscr{P}$, $i\leq j$.
\item If $j< \text{\large \Fontlukas m}$ and $x\in \mathcal{R}_{i,j}$ is different from 0, then there exists $y\in \mathcal{R}_{j,l}$, with $l\neq j$, such that $xy\neq 0$.
\end{enumerate}
\end{definition}

From a poset $\mathscr{P}$ with $m$ elements (including its maximal point), an \textsf{F}-algebra $A$, and an admissible system $\mathcal{R}$, we construct a matrix algebra $\Lambda(\mathcal{R})$.

Given the algebra $\mathscr{M}_m(A)$, of matrices of size $m$ over $A$, its standard basis elements are the matrices $e_{i,j}$, with 1 at the place $(i,j)$ and 0 otherwise. 

Notice that $\mathcal{R}_{i,j}$ is a $\mathcal{R}_i$-$\mathcal{R}_j$-bimodule, for all $i,j\in \mathscr{P}$.

We define 
\[\Lambda=\Lambda(\mathcal{R})=\bigoplus_{\substack{i\leq j \\ i,j\in \mathscr{P}}}e_{i,j}\mathcal{R}_{i,j};\]
which is a subset of the $\textsf{F}$-algebra $\bigoplus_{\substack{i\leq j \\ i,j\in \mathscr{P}}}e_{i,j}A\subset \mathscr{M}_m(A)$\index{$\Lambda=\Lambda(\mathcal{R})$}. 

With the usual matrix operations, $\Lambda$ has a $\textsf{F}$-algebra structure, but it is not always a subalgebra of $\mathscr{M}_m(A)$, because their units are not necessarily the same.

For an ordinary poset $\mathscr{P}$, and $\mathcal{R}_{i,j}=\textsf{F}$, for every $i\leq j$ in $\mathscr{P}$, we recall that $\Lambda$ is the \textit{incidence algebra} of $\mathscr{P}$, see for instance \cite{ASS}.

We denote $e_i=e_{i,i}1_i$, for all $i\in \mathscr{P}$. If 1 is the unit of $\Lambda$, then $1=\sum_{i\in \mathscr{P}}e_i$ is a decomposition in a sum of ortogonal idempotents.

The subalgebra $S=\sum_{i\in \mathscr{P}}e_{i,i}\mathcal{R}_i\subset \Lambda$, determines a decomposition of $\Lambda$ as a direct sum of $S$-$S$-bimodules:
\begin{equation}\label{semisimpleRadical}
\Lambda=S\oplus J, \ \ \text{ where } \ \ J=\bigoplus_{\substack{i< j \\ i,j\in \mathscr{P}}}e_{i,j}\mathcal{R}_{i,j}.
\end{equation}

In \cite{BD2017}, we have introduced representations and corepresentations of $p$-equipped posets. These can be studied through posets with a maximal and a minimal strong points. Then from now on, $\mathscr{P}$ will denote a $p$-equipped poset with a maximal strong element {\large \Fontlukas m} and a minimal strong point $0$.

We associate to $\mathscr{P}$ two different admissible systems determined by a pair of fields $(\textsf{F},\textsf{G})$, where \textsf{G} is a normal extension of degree $p$ over \textsf{F}. We know that $\textsf{G}=\textsf{F}(\xi)$, for some primitive element $\xi$, such that $\xi^p=q\in \textsf{F}$.

Depending on the characteristic of \textsf{F}, the extension $\textsf{G}/\textsf{F}$ may be cyclic or purely inseparable. 

When $\ch \textsf{F}=p$, the extension $\textsf{G}/\textsf{F}$ is purely inseparable. In this case, there is a natural derivation $\delta$ of \textsf{G}, which is an endomorphism of \textsf{G}, satisfying 
\[\delta(a)=0, \text{\ for all\ } a\in \textsf{F},\ \ \delta(\xi^i)=i\xi^{i-1}, \ \  \text{\ for each \ } i<p;\] 
and the Leibniz rule, for every $a,b\in \textsf{G}$,
$$\delta(ab)=a\delta(b)+\delta(a)b.$$

If $\ch \textsf{F}\neq p$, let $\sigma : \textsf{G}\rightarrow \textsf{G}$ be a generator of the Galois group. 

Consider the \textsf{F}-algebra
$$A=(\End_{\textsf{F}}\textsf{G})^{op}.$$

The multiplication in $A$ is denoted by $h*h'=h'h$, for all $h,h' \in \End_{\textsf{F}}\textsf{G}$.

From every $g\in \textsf{G}$, we obtain an \textsf{F}-endomorphism $\mu_g:\textsf{G}\rightarrow \textsf{G}$, given by, 
$$\mu_g(a)=ga\text{, for all }a\in \textsf{G}.$$

We denote 
\begin{equation}\label{defvartheta}
\vartheta= \begin{cases} 
\sigma, & \text{ if } \ch \textsf{F}\neq p,\\
\delta, & \text{ if } \ch \textsf{F}= p,
\end{cases}
\end{equation}
to define the colection 
$$\mathcal{T}=\{\mathcal{T}_{x,y} = A_\ell \ | \ x\leq^\ell y\}_{x,y\in \mathscr{P}},$$  
where, to every $x\leq^\ell y$ in $\mathscr{P}$, it is assigned the following \textsf{F}-subspace of $A$,
\[A_\ell =\textsf{F} \left\langle \sum_{i=0}^{\ell-1}  \vartheta^i * \mu_{g_i}\  |\ \  g_i\in \textsf{G} \right\rangle = \textsf{F} \left\langle \sum_{i=0}^{\ell-1} \mu_{g_i} \vartheta^i \  |\ \  g_i\in \textsf{G} \right\rangle.\]

Notice that
\[\mathcal{T}_x \cong \textsf{G},  \text{ if } x \text{ is weak, and}\]
\[\mathcal{T}_x = A,  \text{ if } x \text{ is strong}.\]

By the decomposition 
$$\textsf{G}=\textsf{F}\oplus \textsf{F}\xi \oplus \cdots \oplus \textsf{F}\xi^{p-1};$$
we have the idempotent 
\begin{equation}\label{varepsilon}
\varepsilon=i\pi; 
\end{equation}
where $i:\textsf{F}\longrightarrow \textsf{G}$ and $\pi:\textsf{G}\longrightarrow \textsf{F}$ are the canonical inclusion and projection, respectively. For a point $x\in \mathscr{P}$, we choose the idempotent
$$\varepsilon_x= \begin{cases} 
\varepsilon, & \text{ if }  x \text{ is strong},\\
\mu_1, & \text{ if }  x \text{ is weak}.
\end{cases}$$

Now we define
$$\mathcal{R}^{(r)}_{x,y}=\varepsilon_x \mathcal{T}_{x,y}\varepsilon_y.$$

The collection of \textsf{F}-subspaces of $A$ 
$$\mathcal{R}^{(r)}=\left\{ \mathcal{R}^{(r)}_{x,y}\right\}_{\substack{x\leq y \\ x,y\in \mathscr{P}}};$$ 
is an admissible system such that $\mathcal{R}^{(r)}_x\cong \textsf{F}$ if $x$ is strong and $\mathcal{R}^{(r)}_x\cong \textsf{G}$ if $x$ is weak.

The field $\textsf{G}$, is an \textsf{F}-algebra, so we can define another admissible system 
$$\mathcal{R}^{(c)}=\left\{\mathcal{R}^{(c)}_{x,y}\subseteq \textsf{G} \ | \ x,y\in \mathscr{P} \ \text{ such that } x\leq y,\right\};$$ 
by assigning to  each $x\leq^\ell y$ in $\mathscr{P}$, the \textsf{F}-space
\[\mathcal{R}^{(c)}_{x,y}=\textsf{F} \langle 1, \xi, \xi^2, \ldots, \xi^{\ell-1} \rangle.\]

In this case $\mathcal{R}^{(c)}_x = \textsf{G}$ if $x$ is strong, and $\mathcal{R}^{(c)}_x = \textsf{F}$ if $x$ is weak.

For a $p$-equipped poset $\mathscr{P}$, we have defined in \cite{BD2017} the category of representations and the category of co-representations. Then we defined two equivalences of categories; one between the representations of $\mathscr{P}$ and the socle-projective $\Lambda \left(\mathcal{R}^{(r)}\right)$-modules, and the other one between the corepresentations of $\mathscr{P}$ and the socle-projective $\Lambda \left(\mathcal{R}^{(c)}\right)$-modules.

Let us denote by $\Lambda^{(r)}$ and $\Lambda^{(c)}$ the algebras $\Lambda \left(\mathcal{R}^{(r)}\right)$ and $\Lambda \left(\mathcal{R}^{(c)}\right)$, and by $\mathcal{U}^{(r)}$ and $\mathcal{U}^{(c)}$ the categories of finitely generated socle-projective modules over $\Lambda^{(r)}$ and $\Lambda^{(c)}$ which do not contain $e_0\Lambda^{(r)}$ and $e_0\Lambda^{(c)}$ as a direct summand, respectively.


\section{Auslander-Reiten sequences}

In this section we will prove some properties for $\Lambda^{(r)}$ and $\Lambda^{(c)}$ using the same methods. Then we will write just $\Lambda$ to denote $\Lambda^{(r)}$ as well as $\Lambda^{(c)}$, and $\mathcal{U}$ for $\mathcal{U}^{(r)}$ and $\mathcal{U}^{(c)}$. Let us use the same conventions for all the notation we will introduce through this paper. 

The category $\mathcal{U}$ is a full subcategory of $ \mathrm{mod}\, \Lambda $ closed under extensions, so it has an
exact structure given by those sequences in $\mathcal{U}$ which are exact in $\mathrm{mod}\, \Lambda $. 

The object $e_0\Lambda$ is the only projective-injective in $ \mathrm{mod}\, \Lambda $ and $e_\text{\large \Fontlukas m}\Lambda$ is the only simple projective in $ \mathrm{mod}\, \Lambda $. We have
$$\text{soc}(e_0\Lambda)=(e_\text{\large \Fontlukas m}\Lambda)^t;$$
$$\text{top}(e_0\Lambda)=D(e_\text{\large \Fontlukas m}\Lambda)^{t^\prime};$$
for some natural numbers $t$ and $t^\prime$.

Denote by $\mathcal{V}$ the category of finitely generated top-injective $\Lambda $-modules which do not contain $e_{0}\Lambda $ as a direct summand. It is a a full subcategory of $ \mathrm{mod}\, \Lambda $ closed under extensions, so as in the case of $\mathcal{U}$, the category $\mathcal{V}$ has an exact structure given by those sequences of morphisms in $\mathcal{V}$ which are
exact in $ \mathrm{mod}\, \Lambda $. There exists an exact functor $F:\mathcal{U}\rightarrow \mathcal{V}$ which is
an equivalence of categories (see \cite{BD2017}, Proposition 31).

\begin{theorem}\label{objetosARseq} 
If $M$ is an indecomposable non-projective object in $\mathcal{U}$ then there is an almost split sequence in $\mathcal{U}$:
$$0\rightarrow N\rightarrow E\rightarrow M\rightarrow 0.$$

If $M$ is an indecomposable non-injective object in $\mathcal{U}$, then there exists an almost split sequence in
$\mathcal{U}$:
$$0\rightarrow M\rightarrow F\rightarrow L\rightarrow 0.$$
\end{theorem}

\begin{proof}
Follows from Proposition 3.3 and $3.7$ of \cite{BM}.
\end{proof}

\begin{definition}
(See Definition 2.6 of \cite{SL}) Let $\mathcal{A}$ be a Krull-Schmidt category. This category is called
\textit{Auslander-Reiten} if each indecomposable $X\in \mathcal{A}$ has a sink and a source morphism in $X$.
Moreover if the sink morphism is not a monomorphism then there exists an almost split sequence ending in $X$ and if
the source morphism in $X$ is not an epimorphism, then there is an almost split sequence starting in $X$.
\end{definition}

Throughout this work if $\mathcal{A}$ is a Krull-Schmidt category and $M$ is an indecomposable in $\mathcal{A}$, such that there is an almost split sequence $x$ ending in $M$, we denote by $\tau M$ the left end of $x$. Similarly if
$y$ is an almost split sequence starting in $M$, we denote by $\tau ^{-1}M$ the right end of $y$.

\begin{proposition} 
The categories $\mathcal{U}$ and $\mathcal{V}$ are Auslander-Reiten.
\end{proposition} 

\begin{proof}
Take $M$ an indecomposable in $\mathcal{U}$. If $M$ is not projective in $\mathcal{U}$, by the first part of Theorem \ref{objetosARseq}, there is an almost split sequence in $\mathcal{U}$, ending in $M$, so there is a sink morphism in $M$.
If $M$ is projective in $\mathcal{U}$, $M$ is a projective right $\Lambda $-module, so the inclusion
$i_{M}:\mathrm{rad}M\rightarrow M$ is a sink morphism in $\mathrm{mod}\, \Lambda $. Here $\mathrm{rad}M$ is an object in $\mathcal{U}$, then $i_{M}$ is a sink morphism in $\mathcal{U}$ ending in $M$. 

Using the equivalence $F:\mathcal{U}\rightarrow \mathcal{V}$, we can see that if $N$ is any indecomposable in $\mathcal{V}$ there is a sink  morphism in $N$. Now if $N$ is not injective we know that the second part of Theorem \ref{objetosARseq} holds for $\mathcal{V}$, therefore there is
an almost split sequence in $\mathcal{V}$ starting in $N$, so there is a source morphism in $N$. In case $N$ is injective in $\mathcal{V}$, it is an injective right $\Lambda $-module, so there is an epimorphism
$j_{N}:N\rightarrow N/\mathrm{soc}N$ which is a source morphism in $\mathrm{mod}\, \Lambda $. In this case 
$N/\mathrm{soc}N\in \mathcal{V}$, so $j_{N}$ is a source morphism in $\mathcal{V}$. Therefore for any indecomposable $N$ in $\mathcal{V}$ there is a source and sink morphism in $N$. Using the equivalence $F$ one conclude that the same property holds in $\mathcal{U}$. 

Suppose $f:Z\rightarrow M$ is a sink morphism. If $f$ is not a monomorphism, then $M$ can not be projective, otherwise $f$ is isomorphic to the inclusion $i_{M}:\mathrm{rad}M\rightarrow M$, and there can not exists an almost split sequence in $\mathcal{U}$ ending in $M$, which contradicts $M$ is projective, by Theorem \ref{objetosARseq}.

Take now $g:M\rightarrow Y$ a source morphism in $\mathcal{U}$ which is not an epimorphism. This means that there is a non zero morphism $s:Y\rightarrow W$ in $\mathcal{U}$ such that $sg=0$. Clearly $F(g):F(M)\rightarrow F(Y)$ is a source morphism which is not an epimorphism in $\mathcal{V}$. Then $F(M)$ can not be injective in $\mathcal{V}$, otherwise $F(g)$ is isomorphic to the epimorphism $F(M)\rightarrow F(M)/\mathrm{soc}F(M)$.

Since $F$ is an exact equivalence, $M$ is not injective in $\mathcal{U}$, which implies, by Theorem \ref{objetosARseq}, that there exists an almost split sequence in $\mathcal{U}$ starting in $M$. 

This proves that $\mathcal{U}$ and therefore $\mathcal{V}$ are Auslander-Reiten categories.
\end{proof}

We recall that if $\mathcal{A}$ is an Auslander-Reiten category then a subset $\mathcal{S}$ of non-isomorphic indecomposable objects of $\mathcal{A}$ is a section if for any $X\in \mathcal{S}$ and an irreducible morphism
$X\rightarrow Y$ either $Y\in \mathcal{S}$ or $\tau Y\in \mathcal{S}$ but not both $Y$ and $\tau Y$ are in $\mathcal{S}$.

By Theorem 2.8 of \cite{SL}, if $\mathcal{S}$ is a section in $\mathcal{U}$ or in $\mathcal{V}$, then $\mathcal{S}$ has not oriented cycles.

Let $\underline{\mathcal{U}}$ be the category $\mathcal{U}$ modulo the ideal generated by the morphisms in $\mathcal{U}$ which
factorizes through projectives, and $\overline{\mathcal{U}}$ the category $\mathcal{U}$ modulo the ideal generated by those morphisms which are factorized through injectives.  Then by Proposition 3.7 of \cite{Bautista04} there is an equivalence of categories
$\Phi :\underline{\mathcal{U}}\rightarrow \overline{\mathcal{U}}$ such that if $M$ is any indecomposable non projective in
$\mathcal{U}$, then $\Phi (M)\cong \tau (M).$ Similarly we denote by $\underline{\mathcal{V}}$ the category $\mathcal{V}$
modulo the ideal generated by the maps which factorizes through projectives and $\overline{\mathcal{V}}$ the category 
$\mathcal{V}$ modulo the ideal generated by the morphisms which factorizes through injectives. Then using
the equivalence $F:\mathcal{U}\rightarrow \mathcal{V}$ we obtain an equivalence of categories
$\Psi :\underline{\mathcal{V}}\rightarrow \overline{\mathcal{V}}$ such that if $N$ is an indecomposable non projective in $\mathcal{V}$ then
$\Psi (N)\cong \tau (N)$.

\begin{proposition}\label{EndXisomEndtauX}
Let $M$ and $N$ be indecomposables non projectives in $\mathcal{U}$ and $\mathcal{V}$, respectively, then there are isomorphisms of $\textsf{F}$-algebras:
$$\mathrm{End}_{\mathcal{U}}(M)/\mathrm{rad}\mathrm{End}_{\mathcal{U}}(M)\cong 
\mathrm{End}_{\mathcal{U}}(\tau M)/\mathrm{rad}\mathrm{End}_{\mathcal{U}}(\tau M),$$
$$\mathrm{End}_{\mathcal{V}}(N)/\mathrm{rad}\mathrm{End}_{\mathcal{V}}(N)\cong 
\mathrm{End}_{\mathcal{V}}(\tau N)/\mathrm{rad}\mathrm{End}_{\mathcal{V}}(\tau N).$$
\end{proposition}

\begin{proof}
The equivalence of categories $\Phi :\underline{\mathcal{U}}\rightarrow \overline{\mathcal{U}}$ induces the first isomorphism and
the equivalence of categories $\Psi :\underline{\mathcal{V}}\rightarrow \overline{\mathcal{V}}$ induces the second one.
\end{proof}

Let $\mathcal{P}(\Lambda )$ be the category whose objects are morphisms $X:X^{1}\rightarrow X^{2}$, where
$X^{1}$ and $X^{2}$ are finitely generated projectives right $\Lambda $-modules. 
If $X_{1}:X_{1}^{1}\rightarrow X_{1}^{2}$ and $X_{2}:X_{2}^{1}\rightarrow X_{2}^{2}$ are objects in $\mathcal{P}(\Lambda )$, a morphism from $X_{1}$ to $X_{2}$ is given by a pair of morphisms $(u_{1},u_{2})$ 
$$u_{1}:X_{1}^{1}\rightarrow X_{2}^{1};$$ 
$$u_{2}:X_{1}^{2}\rightarrow X_{2}^{2};$$ 
such that $u_{2}X_{1}=X_{2}u_{1}$.

The category $\mathcal{P}(\Lambda)$ has an exact structure, where the exact sequences are pairs of morphisms
$X_{1}\stackrel{u}{\longrightarrow }X_{2}\stackrel{v}{\longrightarrow }X_{3}$ such that the following sequences of
right $\Lambda$-modules are exact:
$$0\longrightarrow X_{1}^{1}\stackrel{u_{1}}{\longrightarrow }X_{2}^{1}\stackrel{v_{1}}{\longrightarrow }X_{3}^{1}\longrightarrow 0,
\quad 0\longrightarrow X_{1}^{2}\stackrel{u_{2}}{\longrightarrow }X_{2}^{2}\stackrel{v_{2}}{\longrightarrow }X_{3}^{2}\longrightarrow 0.$$
This category has enough projectives and enough injectives. The injectives are the objects of the form
$T(P)=P\rightarrow 0$ and $I(P)=P\stackrel{id_{P}}{\longrightarrow }P$, the projectives have the form $I(P)$ and $S(P)=0\rightarrow P$, where $P$ is a finitely generated projective right $\Lambda $-module. Moreover the category
$\mathcal{P}(\Lambda )$ has almost split sequences (see Theorem 5.1 of \cite{Bautista04}).

\begin{definition}\label{catProblema}
We denote by $\mathcal{M}$ the full subcategory of $\mathcal{P}(\Lambda)$ whose objects are morphisms of the form $P\stackrel{\phi }{\longrightarrow }(e_{0}\Lambda)^{\nu }$, with $\nu$ a set, and $P$ a projective $\Lambda$-module such that $Pe_{0}=0$.  Observe that $\mathcal{M}$ is in fact a full subcategory
of $\mathcal{P}^{1}(\Lambda )$, which is the full subcategory of $\mathcal{P}(\Lambda)$ whose objects are morphisms
$X:X^{1}\rightarrow X^{2}$ such that $\mathrm{Im}X\subset \mathrm{rad}X^{2}$.
\end{definition}

\begin{proposition} 
The category $\mathcal{M}$ has almost split sequences.
\end{proposition}

\begin{proof}
By Lemma 3.4 of \cite{BM}, the category $\mathcal{M}$ is equivalent as an exact category to the category
$\mathrm{mod}\, \mathcal{D}^{e}$, where $\mathcal{D}^{e}$ is obtained by deletion of some idempotent $e^{\prime}=1-e$ of the Drozd ditalgebra $\mathcal{D}$ of $\Lambda $. Therefore $\mathcal{D}^{e}$ is a finite-dimensional
Roiter ditalgebra with semisimple layer. By Theorem 7.18 of \cite{BSZ}, the category $\mathrm{mod}\, \mathcal{D}^{e}$ has almost split sequences, so $\mathcal{M}$ has almost split sequences.
\end{proof}

In the following, if $M$ is a right $\Lambda$-module, we denote by $X_{M}\in \mathcal{P}(\Lambda)$ a minimal projective presentation of $M$.

\begin{proposition}\label{ASsequence}
For $i\in \mathscr{P}$, $i\neq 0$, there is an almost split sequence in $\mathcal{M}$:
$$X_{D(\Lambda e_{i})}\rightarrow X_{D(\Lambda e_{i})/\mathrm{soc}D(\Lambda e_{i})}\oplus T(Z)\rightarrow T(e_{i}\Lambda )$$
where $Z$ is some finitely generated projective right $\Lambda $-module.
\end{proposition}

\begin{proof}
Since $i\neq 0$, so $D(\Lambda e_{i})$ is a non simple injective, thus by Proposition $5.9$ of \cite{Bautista04}, there
is an almost split sequence in $\mathcal{P}(\Lambda )$:
$$(a)\quad  X_{D(\Lambda e_{i})}\rightarrow X_{D(\Lambda e_{i})/\mathrm{soc}D(\Lambda e_{i})}\oplus T(Z)\rightarrow T(e_{i}\Lambda )$$
where $Z$ is a finitely generated projective right $\Lambda $-module.

Here $D(\Lambda e_{i})$, $D(\Lambda e_{i})/\mathrm{soc}D(\Lambda e_{i})$ are in $\mathcal{V}$, therefore
the objects $X_{D(\Lambda e_{i})},$ and  $ X_{D(\Lambda e_{i})/\mathrm{soc}D(\Lambda e_{i})}$ are in $\mathcal{M}$.

Now if $T(Z)$ appears, then there is a monomorphism $Z\rightarrow \mathrm{rad}e_{i}\Lambda $, thus
$Ze_{0}=0$, which implies that $T(Z)\in \mathcal{M}$. Therefore the almost split sequence $(a)$ lies in
$\mathcal{M}$. This completes the proof.
\end{proof}

\begin{proposition}\label{InjectiveIsom}
The injective indecomposables in $\mathcal{M}$ are isomorphic to the objects $T(e_{i}\Lambda )$ with $i\neq 0$ and $X_{D(\Lambda e_{0})}$.
\end{proposition}

\begin{proof}
We first prove that for any $X\in \mathcal{M}$ there is an exact sequence in $\mathcal{M}$
$$X\stackrel{g}{\rightarrow } W\stackrel{h}{\rightarrow }X^{\prime }$$
where $W$ is a direct sum of objects of the form $T(e_{i}\Lambda ) $ with $i\neq 0$ and $X_{D(\Lambda e_{0})}$. 

Observe that the morphism $X_{D(\Lambda e_{0})}=\mu :Q\rightarrow e_{0}\Lambda $ is a sink morphism in
$\mathrm{proj}\, \Lambda $. Now suppose $X=u:P\rightarrow (e_{0}\Lambda )^{\nu }$ for $\nu $ a natural number
and $P$ a projective. We have $u=(u_{1},...,u_{\nu })^{t}:P\rightarrow (e_{0}\Lambda )^{\nu }$. Here
$\mathrm{Im}u_{i}\subset \mathrm{rad}(e_{0}\Lambda )$ for $i=1,...,\nu$, then since $\mu $ is a sink morphism
there are morphisms $s_{i}:P\rightarrow Q$ such that $\mu s_{i}=u_{i}$ for $i=1,...,\nu $.

Take $s=(s_{1},...,s_{\nu })^{t}:P\rightarrow Q^{\nu }$ and $\mu ^{\nu }:Q^{\nu }\rightarrow (e_{0}\Lambda )^{\nu }$,
then $\mu ^{\nu }s=u$.

We consider the following objects in $\mathcal{M}$:
$$W=(P\longrightarrow 0)\oplus (Q^{\nu }\stackrel{\mu ^{\nu }}{\longrightarrow }(e_{0}\Lambda )^{\nu }), \quad
X^{\prime }=(Q^{\nu }\longrightarrow 0).$$

Then we have the exact sequence:
$$X\stackrel{g}{\longrightarrow }W\stackrel{h}{\longrightarrow }X^{\prime },$$
where $g=(g_{1},g_{2})^{t}$ with $g_{1}=(id_{P},0),\  g_{2}=(s,id _{(e_{0}\Lambda )^{\nu }})$, 
$h=(h_{1},h_{2})$ with $h_{1}=(-s,0), h_{2}=(id_{Q^{\nu }},0)$. 

Therefore if $X$ is an indecomposable injective object in $\mathcal{M}$, it is a direct summand of $W$, consequently
$X$ is isomorphic to some object of the form $T(e_{i}\Lambda )$ with $i\neq 0$ or $X\cong X_{D(\Lambda e_{0})}$.

Now the objects $T(e_{i}\Lambda )$ with $i\neq 0$, and $X_{D(\Lambda e_{0})}$ are injectives in $\mathcal{P}^{1}(\Lambda )$, so they are injectives in $\mathcal{M}$, this proves our result.
\end{proof}

We recall that we have a functor $\mathrm{Cok}:\mathcal{M}\rightarrow \mathcal{V}$ such that $\mathrm{Cok}X=\mathrm{Coker}X$ for any $X:X^{1}\rightarrow X^{2}\in \mathcal{P}(\Lambda )$.

This functor induces an equivalence from the category $\mathcal{M}$ module the ideal $\mathcal{I}$ in $\mathcal{M}$  generated by those morphisms which are factorized through objects of the form $P\rightarrow 0$. For $X, Y\in \mathcal{M}$ we denote by $\mathcal{I}(X,Y)$ the $\textsf{F}$-vector space consisting of those morphisms from 
$X$ to $Y$ which are factorized through objects of the form $P\rightarrow 0$.

We have the following useful result.

\begin{lemma}\label{lemaIdeal} 
Let $f:X\rightarrow Y$ be a morphism in $\mathcal{M}$.

\begin{enumerate}
\item If  $\,Y=X_{F(e_{i}\Lambda )}$, and $f\in \mathcal{I}(X,Y)$, then $f=0$.

\item Suppose  $X$, $Y$ are indecomposable objects with $\mathrm{Cok}(X)\neq 0$ and $\mathrm{Cok}(Y)\neq 0$. Then $f$ is an isomorphism if and only if $\mathrm{Cok}(f)$ is an isomorphism.

\item  If  $X$ is indecomposable and $\mathrm{Cok}(X)\neq 0$,
then $f$ is a section if and only if $\mathrm{Cok}(f)$ is a section. If $Y$ is indecomposable with $\mathrm{Cok}(Y)\neq 0$, then
$f$ is a retraction if and only if $\mathrm{Cok}(f)$ is a retraction.

\item The ideal $\mathcal{I}$ is admissible in the sense of the definition $1.6$ of \cite{SL}.

\item Let $X$ be indecomposable and $\mathrm{Cok}(X)\neq 0$. If $f$ is a source morphism then $\mathrm{Cok}(f)$ is a source morphism. If $\mathrm{Cok}(f)$ is a source morphism and $Y$ has not direct summnds of the form $P\rightarrow 0$, then there is a source morphism $(f,f^{\prime })^{T}:X\rightarrow Y\oplus Z$ with $Z=Q\rightarrow 0$.

\item Let $Y$ be indecomposable with $\mathrm{Cok}(Y)\neq 0$. If $f$ is a sink morphism then $\mathrm{Cok}(f)$ is
a sink morphism. If $\mathrm{Cok}(f)$ is a sink morphism and $X$ has not direct summands of the form $P\rightarrow 0$, then there is an  object of the form $Z=Q\rightarrow 0$ and a sink morphism $(f,f^{\prime }):X\oplus Z  \rightarrow Y$.

\item If $X$ is an indecomposable object of the category $\mathcal{M}$ and $\mathrm{Cok}(X)\neq 0$, then
$$\mathrm{End}_{\mathcal{M}}(X)/\mathrm{rad}\,\mathrm{End}_{\mathcal{M}}(X)\cong \mathrm{End}_{\mathcal{V}}(\mathrm{Cok}(X))/\mathrm{rad}\,\mathrm{End}_{\mathcal{V}}(\mathrm{Cok}(X));$$
as $\textsf{F}$-algebras.

\end{enumerate}
\end{lemma} 

\begin{proof}
\begin{enumerate}
\item Here $Y$ has the form $u:e_{i}\Lambda \rightarrow (e_{0}\Lambda )^{\ell}$, where $\ell=1$ or $\ell=p$ and $u$ is a monomorphism. Then if $(s,0):(P\rightarrow 0)\rightarrow Y$ is a morphism, we have $us=0$, so $s=0$.

\item If $f$ is an isomorphism, clearly $\mathrm{Cok}(f)$ is an isomorphism. Conversely if $\mathrm{Cok}(f)$ is an isomorphism, there is a morphism $s:Y\rightarrow X$ with 
$$\mathrm{Cok}(sf)=\mathrm{Cok}(id_{X}) \text{ and } \mathrm{Cok}(fs)=\mathrm{Cok}(id_{Y}),$$ 
then $sf=id_{X}+h$ and $fs=id_{Y}+g$ with $h\in \mathcal{I}(X,X)$ and
$g\in \mathcal{I}(Y,Y)$. But $h\in \mathrm{rad}\mathrm{End}_{\mathcal{M}}(X)$ which is a nilpotent ideal, so $sf=u_{X}$ is an isomorphism (similarly $fs=u_{Y}$ is an isomorphism). Therefore $f$ is an isomorphism because
$$u_{X}^{-1}sf=id_{X} \text{ and } fu_{X}^{-1}s=id_{Y}.$$ 

\item If $f$ is a section, clearly $\mathrm{Cok}(f)$ is a section. Assume now $\mathrm{Cok}(f)$ is a section. There
is a morphism $s:Y\rightarrow X$ such that $\mathrm{Cok}(sf)$ is the identity on $X$, so by the previous item, $sf$ is an isomorphism. Then, $f$ is a section. The second part is proved in a similar way.

\item Let $f:X\rightarrow Y$ be a morphism between indecomposable objects $X$ and $Y$, which are not of the form $P\rightarrow 0$, and $f\in \mathcal{I}(X,Y)$. It is clear that $f\in \mathrm{rad}^{2}(X,Y)$. Then, the condition $(1)$ of Definition $1.6$ of \cite{SL}, holds.

Consider $f$ a source morphism and $X$ not of the form $P\rightarrow 0$. For any object $W\in \mathcal{M}$ and any morphism $g\in \mathcal{I}(X,W)$, we have $g=vu$ with
$u:X\rightarrow (P\rightarrow 0)$ for some $\Lambda $-projective module $P$. As $u$ can not be a section, there is a morphism $h:Y\rightarrow (P\rightarrow 0)$ such that $hf=u$, then $g=vhf$ with $vh\in \mathcal{I}(Y,W)$. This proves that condition $(2)$ of $1.6$ of \cite{SL} is satisfied. Condition $(3)$ of the mentioned definition holds in a similar way.

\item Suppose $f:X\rightarrow Y$ is a source morphism with $\mathrm{Cok}(X)\neq 0$. From $(2)$ of Lemma $1.7$ of \cite{SL} we conclude that $\mathrm{Cok}(f)$ is a source morphism.

Now let $\mathrm{Cok}(f)$ be a source morphism. The category $\mathcal{M}$ has almost split sequences, then  there is a source morphism 
 $g:X\rightarrow W$, for some $W\in \mathcal{M}$. By the first part of our item $\mathrm{Cok}(g):\mathrm{Cok}(X)\rightarrow \mathrm{Cok}(W)$ is a source morphism. If $W=Z_{1}\oplus Z$ with $Z_{1}$ without direct summands of the form $P\rightarrow 0$ and $Z=Q\rightarrow 0$, we have $\mathrm{Cok}(Z_{1})\cong \mathrm{Cok}(Y)$. By our hypothesis on $Y$ and $Z_{1}$, observe that $Y$ and $Z_{1}$ are minimal projective covers of $\mathrm{Cok}(Y)$ and $\mathrm{Cok}(Z_{1})$, respectively. Then $Z_{1}\cong Y$. From here we obtain our result.

\item The proof of this item is similar to the previous one.

\item Here $\mathcal{I}(X,X)\subset \mathrm{rad}\,\mathrm{End}_{\mathcal{M}}(X)$, therefore the functor
$\mathrm{Cok}$ induces the required isomorphism.

\end{enumerate}
\end{proof}

\begin{proposition}\label{ARseq} 
Suppose
$$X\stackrel{u}{\rightarrow }Y\stackrel{v}{\rightarrow }Z$$
is an almost split sequence in $\mathcal{M}$. Then if $\mathrm{Cok}X\neq 0, \mathrm{Cok}Y\neq 0$ and
$\mathrm{Cok}Z\neq 0$, the sequence
$$0\longrightarrow \mathrm{Cok}X\xrightarrow{\ \mathrm{Cok}\, u\ }\mathrm{Cok}Y\xrightarrow{\ \mathrm{Cok}\, v\ } \mathrm{Cok}Z\longrightarrow 0$$
is an almost split sequence in $\mathcal{V}$.
\end{proposition} 

\begin{proof}
From Proposition 35 of \cite{BD2017}, we have that the functor $\mathrm{Cok}$ induces an
equivalence $\mathcal{M}/\mathcal{I}\rightarrow \mathcal{V}$. By 2 of the previous lemma, $\mathcal{I}$ is an admissible ideal in the sense of Definition $1.6$ of \cite{SL}. Then our result follows from Lemma $1.7$ of \cite{SL}.
\end{proof}


\section{The Euler form}

In this section, for $X,Y\in \mathcal{M}$ we calculate in Proposition \ref{EulerForm}  the Euler form

$$\langle X,Y\rangle =\mathrm{dim}_\textsf{F}\mathrm{Hom}_{\mathcal{M}}(X,Y)-\mathrm{dim}_\textsf{F}\mathrm{Ext}_{\mathcal{M}}(X,Y).$$

For doing this we first observe that the category $\mathcal{M}$ is a full subcategory of $\mathcal{P}(\Lambda )$ closed under extensions, then for $X,Y\in \mathcal{M}$ we have
$$\langle X,Y\rangle =\mathrm{dim}_\textsf{F}\mathrm{Hom}_{\mathcal{P}(\Lambda )}(X,Y)-\mathrm{dim}_\textsf{F}\mathrm{Ext}_{\mathcal{P}(\Lambda )}(X,Y).$$

As a consequence of Proposition \ref{EulerForm} we will prove Proposition \ref{IsomObjects} which will be used in  subsequent sections.

The Euler form will be expressed in terms of the \textit{coordinates} of $X$ and $Y$ according with the following.

\begin{definition}
Let $X:X^{1}\rightarrow X^{2}$ be an object in $\mathcal{M}$ with
$X^{1}=\oplus _{0<i\in \mathscr{P}}(e_{i}\Lambda )^{d^{X}_{i}}$, $X^{2}=(e_{0}\Lambda )^{d^{X}_{0}}$.
Then the coordinates of $X$ is the element $\mathrm{cd}(X)\in \mathbb{Q}^{|\mathscr{P}|}$ such that $\mathrm{cd}(X)(i)=d^{X}_{i}$, for $i\in \mathscr{P}$.
\end{definition}

Let $\mathscr{P}$ be a $p$-equipped poset, define for every $i\in\mathscr{P}$
$$\ell^{(r)}_{i,i}=\begin{cases} 
1, & \text{ if }  i \text{ is strong},\\
p, & \text{ if }  i \text{ is weak};
\end{cases}$$
and for every pair $i,j\in\mathscr{P}$
$$\ell^{(r)}_{i,j}=\begin{cases} 
\ell \ell_{i,i}\ell_{j,j}/p, & \text{ if }  i<^\ell j,\\
0, & \text{ otherwise};
\end{cases}$$
$$\ell^{(c)}_{i,j}=\begin{cases} 
\ell, & \text{ if }  i\leq^\ell j,\\
0, & \text{ otherwise}.
\end{cases}$$

 Associated to the $p$-equipped poset $\mathscr{P}$ we have two bilinear forms:
 
 $$\mathfrak{b}_{\mathscr{P}}^{(r)}, \mathfrak{b}_{\mathscr{P}}^{(c)}:\mathbb{Q}^{|\mathscr{P}|}\times \mathbb{Q}^{|\mathscr{P}|}\rightarrow \mathbb{Q}$$
  defined for $d_{1},d_{2}\in \mathbb{Q}^{|\mathscr{P}|}$ as follows:

$$\mathfrak{b}_{\mathscr{P}}^{(r)}(d_{1},d_{2})=\sum _{0\leq i\in \mathscr{P}}\ell^{(r)}_{i,i}d_{1}(i)d_{2}(i)+\sum _{0<i<j \in \mathscr{P}}\ell^{(r)}_{i,j}d_{1}(i)d_{2}(j)-d_{1}(0)\sum _{0<i\in \mathscr{P}}\ell^{(r)}_{i,i}d_{2}(i);$$

$$\mathfrak{b}_{\mathscr{P}}^{(c)}(d_{1},d_{2})=\sum _{0\leq i\in \mathscr{P}}\ell^{(c)}_{i,i}d_{1}(i)d_{2}(i)+\sum _{0<i<j \in \mathscr{P}}\ell^{(c)}_{i,j}d_{1}(i)d_{2}(j)-pd_{1}(0)\sum _{0<i\in \mathscr{P}}\ell^{(c)}_{i,i}d_{2}(i);$$
and the corresponding quadratic forms on $\mathbb{Q}^{|\mathscr{P}|}$,

$$q^{(r)}(d)= q_{\mathscr{P}}^{(r)}(d)=\sum _{0\leq i\in \mathscr{P}}\ell^{(r)}_{i,i}d(i)^{2}+\sum _{0<i<j \in \mathscr{P}}\ell^{(r)}_{i,j}d(i)d(j)-d(0)\sum _{0<i\in \mathscr{P}}\ell^{(r)}_{i,i}d(i);$$

$$q^{(c)}(d)=q_{\mathscr{P}}^{(c)}(d)=\sum _{0\leq i\in \mathscr{P}}\ell^{(c)}_{i,i}d(i)^{2}+\sum _{0<i<j \in \mathscr{P}}\ell^{(c)}_{i,j}d(i)d(j)-pd(0)\sum _{0<i\in \mathscr{P}}\ell^{(c)}_{i,i}d(i).$$

\begin{proposition}\label{EulerForm}
For $X,Y\in \mathcal{M}$ the Euler quadratic form
$$\langle X,Y\rangle =\mathfrak{b}_{\mathscr{P}}(\mathrm{cd}(Y),\mathrm{cd}(X)).$$
\end{proposition}

\begin{proof}
We recall from \cite{Bautista04} that for $X:X^{1}\rightarrow X^{2}\in \mathcal{P}(\Lambda)$ there exists an exact sequence 
$$0\longrightarrow T(X^{1})\longrightarrow T(X^{2}) \oplus I(X^{1}) \longrightarrow X \longrightarrow 0;$$
then for $Y:Y^{1}\rightarrow Y^{2}\in \mathcal{P}(\Lambda)$, we have
$$0\longrightarrow \mathrm{Hom}_{\mathcal{P}(\Lambda)}(X,Y)\longrightarrow \mathrm{Hom}_{\mathcal{P}(\Lambda)}(T(X^{2}),Y) \oplus \mathrm{Hom}_{\mathcal{P}(\Lambda)}(I(X^{1}),Y)$$ 
$$\hspace{4cm}\longrightarrow \mathrm{Hom}_{\mathcal{P}(\Lambda)}(T(X^{1}),Y)\longrightarrow \mathrm{Ext}_{\mathcal{P}(\Lambda)}(X,Y)\longrightarrow 0.$$

From here we obtain an expression for the Euler form, with $Y:Y^{1}\rightarrow Y^{2}$
$$\mathrm{dim}_\textsf{F} \mathrm{Hom}_{\mathcal{P}(\Lambda)}(X,Y) - \mathrm{dim}_\textsf{F}\mathrm{Ext}_{\mathcal{P}(\Lambda)}(X,Y)$$
$$= \mathrm{dim}_\textsf{F} \mathrm{Hom}_{\mathcal{P}(\Lambda)}(T(X^{2}),Y) + \mathrm{dim}_\textsf{F} \mathrm{Hom}_{\mathcal{P}(\Lambda)}(I(X^{1}),Y) - \mathrm{dim}_\textsf{F}\mathrm{Hom}_{\mathcal{P}(\Lambda)}(T(X^{1}),Y);$$
by using the proposition 3.1 of \cite{Bautista04}
$$\mathrm{dim}_\textsf{F} \mathrm{Hom}_{\mathcal{P}(\Lambda)}(X,Y) - \mathrm{dim}_\textsf{F}\mathrm{Ext}_{\mathcal{P}(\Lambda)}(X,Y)$$
$$=\mathrm{dim}_\textsf{F} \mathrm{Hom}_{\Lambda}(X^{2},Y^{2}) + \mathrm{dim}_\textsf{F} \mathrm{Hom}_{\Lambda}(X^{1},Y^{1}) - \mathrm{dim}_\textsf{F} \mathrm{Hom}_{\Lambda}(X^{1},Y^{2}).$$

Consider
$$X:\oplus _{0<i\in \mathscr{P}}(e_{i}\Lambda )^{d^{X}_{i}}\longrightarrow (e_{0}\Lambda )^{d^{X}_{0}};$$
$$Y:\oplus _{0<i\in \mathscr{P}}(e_{i}\Lambda )^{d^{Y}_{i}}\longrightarrow (e_{0}\Lambda )^{d^{Y}_{0}};$$
and notice that 
$$\mathrm{dim}_\textsf{F} \mathrm{Hom}_{\Lambda}(e_i\Lambda, e_j\Lambda)= \ell_{j,i}.$$

Now our result is obtained by a straightforward calculation. 
\end{proof}

\begin{proposition}\label{IsomObjects}
Suppose $X, Y$ are indecomposable objects in $\mathcal{M}$ such that $\mathrm{Ext}_{\mathcal{M}}(X,X)=0$ and
$\mathrm{Ext}_{\mathcal{M}}(Y,Y)=0$. Then $X\cong Y$ if and only if $\mathrm{cd}(X)=\mathrm{cd}(Y)$.
\end{proposition}

\begin{proof}
We have that $\mathcal{M}$ is a full subcategory of $\mathcal{P}^{1}(\Lambda )$, closed under extensions.
Then for $X,Y$ indecomposable objects of $\mathcal{M}$, with $X:X^{1}\rightarrow X^{2}$, $Y:Y^{1}\rightarrow Y^{2}$
$$\mathrm{Ext}_{\mathcal{M}}(X,X)=\mathrm{Ext}_{\mathcal{P}^{1}(\Lambda )}(X,X),  \mathrm{Ext}_{\mathcal{M}}(Y,Y)=\mathrm{Ext}_{\mathcal{P}^{1}(\Lambda )}(Y,Y).$$
Then from Theorem $7.1$, and Theorem $8.1$ of [5], follows that if $\mathrm{Ext}_{\mathcal{M}}(X,X)=0$ and
$\mathrm{Ext}_{\mathcal{M}}(Y,Y)=0$ then $X\cong Y$ if and only if $X^{1}\cong Y^{1}$ and $X^{2}\cong Y^{2}$
as $\Lambda $-modules. In our case, this is equivalent to $\mathrm{cd}(X)=\mathrm{cd}(Y)$.
\end{proof}

\begin{corollary}\label{corIsomObjects}
Let $X,Y$ be indecomposable objects in $\mathcal{M}$ such that $\mathrm{Ext}_{\mathcal{M}}(X,X)$ is cero, $\mathrm{cd}(X)=\mathrm{cd}(Y)$,  and
$\mathrm{dim}_{\textsf{F}}\mathrm{End}_{\mathcal{M}}(X)=\mathrm{dim}_{\textsf{F}}\mathrm{End}_{\mathcal{M}}(Y)$, then
$X\cong Y$.
\end{corollary}

\begin{proof}
Since $\mathrm{Ext}_{\mathcal{M}}(X,X)=0$, 
$$\mathrm{dim}_{\textsf{F}}\mathrm{End}_{\mathcal{M}}(X)=\langle X,X \rangle =q(\mathrm{cd}(X))=q(\mathrm{cd}(Y))$$
$$= \langle Y,Y \rangle =\mathrm{dim}_{\textsf{F}}\mathrm{End}_{\mathcal{M}}(Y) -
\mathrm{dim}_{\textsf{F}}\mathrm{Ext}_{\mathcal{M}}(Y,Y)$$ 
Then $\mathrm{Ext}_{\mathcal{M}}(Y,Y)=0$, and our result follows
from Proposition \ref{IsomObjects}.
\end{proof}


\section{Pseudo hereditary projectives}

In this section we consider pseudo hereditary projective modules according with the definition given bellow. We study some
properties of these objects in the categories $\mathcal{U}$, $\mathcal{V}$ and $\mathcal{M}$.

\begin{definition}\label{PseudoHP}
Let $\mathcal{A}$ be a Krull-Schmidt category with an exact structure having enough projectives and
enough injectives. We say that an indecomposable objet $M\in \mathcal{A}$ is \textit{pseudo hereditary projective} if for
any chain of irreducible morphisms in $\mathcal{A}$:
$$X_{1}\rightarrow X_{2}\rightarrow \cdots \rightarrow X_{l}\rightarrow M$$
one has that $X_{1},...,X_{l}$ are projective objects in $\mathcal{A}$.
\end{definition}

Recall that a projective right module is said to be \textit{hereditary} if its submodules are all projective.

\begin{proposition}\label{PseudoHisH}
A projective $\Lambda $-module in $\mathcal{U}$ is pseudo hereditary if and only if it is hereditary.
\end{proposition}

\begin{proof}
If $P$ is a hereditary projective module in $\mathcal{U}$, clearly it is pseudo hereditary projective. Now suppose $P$ is a pseudo hereditary projective, then for each $Q$ indecomposable direct summand of $\mathrm{rad}P$ we have an irreducible morphism $Q\rightarrow P$ in $\mathrm{mod}\, \Lambda $, since $Q\in \mathcal{U}$, then $Q\rightarrow P$ is an irreducible morphism in $\mathcal{U}$, therefore $Q$ is projective, so $\mathrm{rad}P$ is a projective
$\Lambda $-module. 

If $Q_{1}$ is a direct summand of $\mathrm{rad}^{2}P$, then we have irreducible morphisms
$Q_{1}\rightarrow Z\rightarrow P$, then $Q_{1}$ is projective and $\mathrm{rad}^{2}P$ is projective. Proceeding
in this way we can prove that $\mathrm{rad}^{l}P$ is projective for all $l\in \mathbb{N}$. This implies that $P$ is a hereditary projective
$\Lambda $-module.
\end{proof}

\begin{proposition}\label{XisPiffCokisP}
Let $X$ be an indecomposable in $\mathcal{M}$. Then $X$ is projective in $\mathcal{M}$
if and only if $\mathrm{Cok}X$ is a no null projective object of $\mathcal{V}$.
\end{proposition}

\begin{proof}
If $X$ is projective, then $\mathrm{Cok}X\neq 0$, otherwise $X\cong T(e_{i}\Lambda )$ for some $i\in \mathscr{P}, i\neq0$,  but by Proposition \ref{ASsequence}, there is an almost split sequence in $\mathcal{M}$ ending in $T(e_{i}\Lambda )$.

Suppose now that $X$ is projective and $\mathrm{Cok}X$ is not a projective object in $\mathcal{V}$. In this case, we have an almost
split sequence in $\mathcal{V}$:
$$\tau \mathrm{Cok}X\rightarrow E\rightarrow \mathrm{Cok}X.$$

Now by Proposition \ref{InjectiveIsom} the indecomposable injectives in $\mathcal{M}$ are isomorphic to objects
$T(e_{i}\Lambda )$ or $X_{D(\Lambda e_{0})}$. Here $\tau \mathrm{Cok}X$ is not injective in $\mathcal{V}$,
therefore $X_{\tau \mathrm{Cok}X}$ is not injective in $\mathcal{M}$. Then we have an almost split sequence in
$\mathcal{M}$:
$$X_{\tau \mathrm{Cok}X}\rightarrow Y\rightarrow W.$$
If $\mathrm{Cok}W=0$, then $W\cong T(e_{i}\Lambda )$ for some $i$, but then by Proposition \ref{ASsequence}
there is an almost split sequence:
$$X_{D(\Lambda e_{i})}\rightarrow X_{D(\Lambda e_{i})/\mathrm{soc}D(\Lambda e_{i})}\oplus T(Z)\rightarrow T(e_{i}\Lambda ).$$

Then $X_{D(\Lambda e_{i})}\cong X_{\tau\mathrm{Cok}X}$, so $D(\Lambda e_{i})\cong \tau\mathrm{Cok}X$ which is
not possible. Therefore $\mathrm{Cok}W\neq 0$, this implies $\mathrm{Cok}Y\neq 0$. By Proposition \ref{ARseq}, we have an almost split sequence  starting in $\tau \mathrm{Cok}X$ and ending in $\mathrm{Cok}W$, that is $\mathrm{Cok}X\cong \mathrm{Cok}W$, so $X\cong W$, which contradicts $X$ is projective. We conclude that $\mathrm{Cok}X$ is a projective object in $\mathcal{V}$.

Now consider $\mathrm{Cok}X\neq 0$ and $\mathrm{Cok}X$ projective in $\mathcal{V}$. If $X$ is not projective, there is an
almost split sequence in $\mathcal{M}$:
$$Z\rightarrow Y\rightarrow X;$$
where $Z$ is not injective so $\mathrm{Cok}Z\neq 0$, and as $\mathrm{Cok}X\neq 0$, then
$\mathrm{Cok}Y\neq 0$. By Proposition \ref{ARseq}, we have an almost split sequence ending in $\mathrm{Cok}X$, which is a contradiction. Therefore $X$ is projective and the proof is complete.
\end{proof}

\begin{lemma}\label{functorres}
For every $M\in \mathcal{U}$, the restriction to the socle of $M$ induces an injective morphism of
$\textsf{F}$-algebras:
$$\mathrm{res}:\mathrm{End}_{\Lambda }(M)\rightarrow \mathrm{End}_{\Lambda }(\mathrm{soc}(M)).$$
In particular if $\mathrm{soc}(M)$ is a simple module and $\Lambda =\Lambda ^{(r)}$, then $\mathrm{End}_{\Lambda }(M)=\textsf{F}$.
\end{lemma}

\begin{proof}
We have $\mathrm{soc}(M)=Me_\text{\large \Fontlukas m}$, then for $f\in \mathrm{End}_{\Lambda }(M)$
$$f(\mathrm{soc}(M))=f(Me_\text{\large \Fontlukas m})=f(M)e_\text{\large \Fontlukas m}=\mathrm{soc}(f(M))\subset \mathrm{soc}(M).$$
If $f(\mathrm{soc}(M))=0$, then $\mathrm{soc}(f(M))=0$ and so $f=0$. To finish the proof, recall that the only simple projective of $\Lambda ^{(r)}$ is isomorphic to $\textsf{F}$.
\end{proof}

\begin{proposition}\label{radicalLambdar}
Consider $\Lambda =\Lambda ^{(r)}$, and $i\in \mathscr{P}$.
If $i$ is a weak point and for all $j$ with $i<j$, we have $i<^{p}j$, then
$$\mathrm{rad}(e_{i}\Lambda )\cong T_{i}^{p},$$
where $T_{i}$ is an indecomposable $\Lambda $-module, and $\mathrm{End}_{\Lambda }(T_{i})=\textsf{F}$.

Otherwise, $\mathrm{rad}(e_{i}\Lambda )$ is indecomposable and $\mathrm{End}_{\Lambda }(\mathrm{rad}(e_{i}\Lambda ))\cong \textsf{G}$ if $i$ is weak and $\mathrm{End}_{\Lambda }(\mathrm{rad}(e_{i}\Lambda ))=\textsf{F}$ if $i$ is strong.
\end{proposition}

\begin{proof}
Suppose $i$ is weak and for all $j$ with $i<j$, we have $i<^{p}j$. As $i$ is a weak point, $e_{i}\Lambda e_{\text{\large \Fontlukas m}}\cong G$, then $\mathrm{soc}\, e_{i}\Lambda =(e_{\text{\large \Fontlukas m}}\Lambda )^{p}$,
Then we have an inclusion:
$$\iota :\mathrm{rad}\, e_{i}\Lambda \longrightarrow (e_{0}\Lambda )^{p}.$$

Now consider $T_{i}$ the $\Lambda $-submodule of $e_{0}\Lambda $ given by
$$T_{i}=\bigoplus_{x\in\mathscr{P}, i<x}e_{0}\Lambda e_{x};$$
then $\mathrm{Im}\iota \subset T_{i}^{p}$. 

Now for any $x>i$, we have $i<^{p}x$, so
$$\mathrm{dim}_{\textsf{F}}\, e_{i}\Lambda e_{x} = \mathrm{dim}_{\textsf{F}} (\mathrm{rad}\, e_{i}\Lambda) e_x
=\begin{cases} 
p, & \text{ if }  x \text{ is strong},\\
p^2, & \text{ if }  x \text{ is weak};
\end{cases}$$
also
$$\mathrm{dim}_{\textsf{F}}\, e_{0}\Lambda e_{x} = \mathrm{dim}_{\textsf{F}} (T_i)e_x
=\begin{cases} 
1, & \text{ if }  x \text{ is strong},\\
p, & \text{ if }  x \text{ is weak};
\end{cases}$$
therefore
$$\mathrm{rad}(e_{i}\Lambda)\cong T_{i}^{p}.$$

We have $\mathrm{dim}_{\textsf{F}}(\mathrm{soc}(T_{i}))=1$, which implies that the socle of $T_{i}$ is simple and then
$\mathrm{End}_{\Lambda }(T_{i})=\textsf{F}$, by Lemma \ref{functorres}.

Suppose there is a $j\in \mathscr{P}$ with $i<^{\ell}j$ and $\ell<p$. Here
$\mathrm{rad}(e_{i}\Lambda )=\oplus _{i<t}e_{i}\Lambda e_{t}$ and each $e_{i}\Lambda e_{t}$ is a left
$\textsf{G}$-vector space, then the left multiplication by the elements of  $\textsf{G}$ gives a monomorphism of $\textsf{F}$-algebras
$$\textsf{G}\rightarrow \mathrm{End}_{\Lambda }(\mathrm{rad}(e_{i}\Lambda )).$$

Take now $j\in \mathscr{P}$, with $i<^{\ell}j$ and $\ell<p$. Consider $u$ an endomorphism of
$\mathrm{rad}(e_{i}\Lambda)$. From Section \ref{SectionPequipped}, 
$$e_{i}\Lambda e_j\cong A_{\ell} \text{ and } e_{i}\Lambda e_\text{\large \Fontlukas m}\cong A\varepsilon;$$
then $u$ induces  $\textsf{F}$-linear transformations $f:A_{\ell}\rightarrow A_{\ell}$,
and $g:A\varepsilon \rightarrow A\varepsilon $,
such that $u(ze_{i,j})=f(z)e_{i,j}$, for all $z\in A_{\ell}$ and $u(we_{i,\text{\large \Fontlukas m}})=g(w)e_{i,\text{\large \Fontlukas m}}$ for all $w\in A\varepsilon $.

Since any $\textsf{F}$-linear endomorphism of $A\varepsilon$ is given by the left multiplication for some element in $A$, we have
an $a\in A$ such that $u(we_{i,j})=awe_{i,j}$ for all $w\in A\varepsilon$.

Consider $\varepsilon=\varepsilon_1$ and $1=\varepsilon_1 +\varepsilon _{2}+...+\varepsilon _{p}$ a decomposition of $1\in A$ into a sum of
primitive orthogonal idempotents. We know that for each $i=1,...,p$, there are elements $x_{i}\in A\varepsilon, y_{i}\in \varepsilon A\varepsilon_{i}$ such that $\varepsilon _{i}=x_{i}\varepsilon y_{i}$, therefore
$$1=\sum _{i=1}^{p}x_{i} \varepsilon y_{i}.$$ 

Also  $x_{i}\varepsilon e_{j,\text{\large \Fontlukas m}}\in e_{j}\Lambda e_{\text{\large \Fontlukas m}}$, so if
$z\in A_{\ell}$,
$$u(ze_{i,j}x_{i}\varepsilon e_{j,\text{\large \Fontlukas m}})=u(ze_{i,j})x_{i}\varepsilon e_{j, \text{\large \Fontlukas m}}=f(z)e_{i,j}x_{i}\varepsilon e_{j,\text{\large \Fontlukas m}}=f(z)x_{i}\varepsilon e_{i,\text{\large \Fontlukas m}}.$$
Moreover:
\begin{equation}\label{endsocle}
u(ze_{i,j}x_{i}\varepsilon e_{j, \text{\large \Fontlukas m}})=u(zx_{i}\varepsilon e_{i,\text{\large \Fontlukas m}})=azx_{i}\varepsilon e_{i, \text{\large \Fontlukas m}};
\end{equation}
thus $f(z)x_{i}\varepsilon =azx_{i}\varepsilon $.

Then 
$$f(z)=\sum _{i=1}^{p}f(z)x_{i}\varepsilon y_{i}=\sum _{i=1}^{p}azx_{i}\varepsilon y_{i}=az;$$
for all $z\in A_{\ell}$. We have $1\in A_{\ell}$ and $f(1)=a\in A_{\ell}$, therefore
$a=g_{0}1+...+g_{r}\vartheta ^{r}$ with $\vartheta$ as in (\ref{defvartheta}), $g_0,\ldots g_r \in \textsf{G}$, $g_{r}\neq 0$ and $r\leq \ell-1<p-1$. Suppose $r\geq 1$, in this case $\vartheta ^{\ell-r}\in A_{\ell}$, because $\ell-r\leq \ell-1$, but $f(\vartheta ^{\ell-r})=a\vartheta ^{\ell-r}=g_{0}\vartheta ^{\ell-r}+..+g_{r}\vartheta ^{\ell}$ which is not in $A_{\ell}$, therefore $a=g_{0}=g$.

Equation (\ref{endsocle}) implies that $u$ and the left multiplication by $g$
coincide when restricted to the socle of $\mathrm{rad}(e_{i}\Lambda )$, so by Lemma \ref{functorres}, $u$ is the left multiplication by $g$. Then $\mathrm{End}_{\Lambda }(\mathrm{rad}(e_{i}\Lambda ))\cong \textsf{G}$, and $\mathrm{rad}(e_{i}\Lambda )$
is an indecomposable $\Lambda $-module. 

Now if $i$ is strong then $\mathrm{soc}(\mathrm{rad}(e_{i}\Lambda ))$ is
a simple $\Lambda $-module, this implies that $\mathrm{rad}(e_{i}\Lambda )$ is indecomposable and its endomorphism ring coincides with $\textsf{F}$. The proof is complete.
\end{proof}

\begin{proposition}\label{radicalLambdac}
If $\Lambda =\Lambda ^{(c)}$, then the radical of any  indecomposable projective right $\Lambda$-module is indecomposable.
\end{proposition}

\begin{proof}
In this case for all $i\in \mathscr{P}$, $\mathrm{soc}(e_{i}\Lambda)=e_{i}\Lambda e_{\text{\large \Fontlukas m}}\cong \textsf{G}$. Therefore the socle of $e_{i}\Lambda$ is simple, consequently $\mathrm{rad}(e_{i}\Lambda)$ is indecomposable.
\end{proof}

\begin{lemma}\label{irreducibleInVandM}
Suppose $g:X\rightarrow Y$ is a morphism in $\mathcal{M}$, with $X$ and $Y$ indecomposable objects
such that $\mathrm{Cok}X\neq 0$ and  $\mathrm{Cok}Y\neq 0$. Then  $\mathrm{Cok}(g):\mathrm{Cok}X\rightarrow \mathrm{Cok}Y$ is an irreducible morphism
in $\mathcal{V}$, if and only if $g$ is an irreducible morphism in $\mathcal{M}$.
\end{lemma}

\begin{proof}
Suppose $\mathrm{Cok}(g)$ is irreducible. If $g=vu$ with $u:X\rightarrow Z, v:Z\rightarrow Y$, then $\mathrm{Cok}(g)=\mathrm{Cok}(v)\mathrm{Cok}(u)$. Therefore $\mathrm{Cok}(u)$ is a section or $\mathrm{Cok}(v)$ is a retraction. 
By $3.$ of Lemma \ref{lemaIdeal}, $u$ is a section or $v$ is a retraction. This implies that $g$ is an irreducible morphism.

Suppose now that $g$ is an irreducible morphism. Then by Lemma $1.7 (1)$ of \cite{SL},  $\mathrm{Cok}(g)$ is an irreducible morphism in $\mathcal{V}$. The proof is complete.
\end{proof}

\begin{proposition}\label{psHereditaryEquivF}
An indecomposable projective $X\in \mathcal{M}$ is a pseudo hereditary projective
if and only if $X\cong X_{F(e_{i}\Lambda )}$ where $e_{i}\Lambda $ is a hereditary projective in $\mathrm{mod}\, \Lambda $ and $F:\mathcal{U}\rightarrow \mathcal{V}$ is the equivalence of categories of \cite{BD2017}, Proposition $31$.
\end{proposition}

\begin{proof}
If $X$ is a hereditary projective, $X$ is projective, then by Proposition \ref{XisPiffCokisP}, $\mathrm{Cok}X$ is a projective object in $\mathcal{V}$. Take now any chain of irreducible morphisms between indecomposable objects in $\mathcal{V}$:
$$M_{1}\rightarrow \cdots \rightarrow M_{l}\rightarrow \mathrm{Cok}X.$$

By Lemma \ref{irreducibleInVandM} there are irreducible morphisms
$$X_{M_{1}} \rightarrow \cdots \rightarrow X_{M_{l}}\rightarrow X;$$
since $X$ is pseudo-hereditary, $X_{M_{1}}, \ldots , X_{M_{l}}$ are projectives. Then  $M_{1}, \ldots , M_{l}$ are projectives in $\mathcal{V}$, using Proposition \ref{XisPiffCokisP}. Therefore $\mathrm{Cok}X$ is pseudo hereditary projective in
$\mathcal{V}$. 

Since $F:\mathcal{U}\rightarrow \mathcal{V}$ is an equivalence of categories, $\mathrm{Cok}X\cong F(e_{i}\Lambda )$, with $e_{i}\Lambda $ pseudo-hereditary. Then by Proposition \ref{PseudoHisH}, $e_{i}\Lambda $ is an hereditary
projective in $\mathcal{U}$. Therefore $X\cong X_{F(e_{i}\Lambda )}$ where $e_{i}\Lambda $ is a hereditary projective.

Conversely, assume $X=X_{F(e_{i}\Lambda )}$ with $e_{i}\Lambda $ a hereditary projective. We claim that
$X$ is a pseudo hereditary projective object in $\mathcal{M}$. Here $e_{i}\Lambda $ is a pseudo hereditary projective
object in the category $\mathcal{U}$, then $F(e_{i}\Lambda )$ is a pseudo hereditary projective object of $\mathcal{V}$.

Let $X_{1}\rightarrow \cdots \rightarrow X_{l}\rightarrow X_{F(e_{i}\Lambda )}$ be a chain of irreducible morphisms between
indecomposable objects in $\mathcal{M}$.  By $1.$ of Lemma \ref{lemaIdeal}, there are not non zero morphisms from an object of the form $P\rightarrow 0$ to $X_{F(e_{i}\Lambda )}$,
this implies that $\mathrm{Cok}X_{l}$ is non zero, therefore the irreducible morphism $X_{l}\rightarrow X_{F(e_{i}\Lambda )} $ is sent into an irreducible morphism $\mathrm{Cok}X_{l}\rightarrow F(e_{i}\Lambda )$, this implies that
$\mathrm{Cok}X_{l}$ is a projective object in $\mathcal{V}$,  so $\mathrm{Cok}X_{l}\cong F(e_{u}\Lambda )$ for some $u$, and $X_{l}=X_{F(e_{u}\Lambda )}$.  By Proposition \ref{XisPiffCokisP}, $X_{l}$ is a projective, as before $X_{l-1}$ is not of the form $P\rightarrow 0$, again as before there is an irreducible morphism $\mathrm{Cok}X_{l-1}\rightarrow \mathrm{Cok}X_{l}$. Then $\mathrm{Cok}X_{l-1}$ is projective and consequently $X_{l-1}$ is a projective
object in $\mathcal{M}$, following this way we can prove that $X_{1},\ldots ,X_{l}$ are projective objects of $\mathcal{M}$. This proves that $X$ is a pseudo hereditary projective object of $\mathcal{M}$. The proof is complete.
\end{proof}


\section{Sections in $\mathcal{M}$}

We recall that $\mathcal{M}$ is the full subcategory of $\mathcal{P}(\Lambda )$ whose objects are morphisms of the
form $P\stackrel{u}{\rightarrow }(e_{0}\Lambda )^{\nu}$ where $\nu $ is a non negative integer number. By Proposition $8$, this category has almost split sequences. In this section we describe the Auslander-Reiten component of the object
$X_{F(e_{\text{\large \Fontlukas m}}\Lambda )}$, which is the minimal projective presentation of the simple injective $F(e_{\text{\large \Fontlukas m}}\Lambda )$, in $\mathcal{V}$, where $e_{\text{\large \Fontlukas m}}\Lambda $ is the simple projective object in $\mathcal{U}$. We will need the following fact.

\begin{proposition}\label{isomTroughIrreducible}
Assume $X\in \mathcal{M}$ is an indecomposable projective, and $Y\rightarrow X$ 
$Y_{1}\rightarrow X$ are irreducible  morphisms with  $Y$ and $Y_{1}$  indecomposable, then $Y\cong Y_{1}$.
\end{proposition}

\begin{proof}
By Proposition \ref{XisPiffCokisP}, $\mathrm{Cok}X$ is a no null indecomposable projective object in $\mathcal{V}$, therefore $\mathrm{Cok}X=F(e_{i}\Lambda ) $ for some $i\in \mathscr{P}$. By $1.$ of  Lemma \ref{lemaIdeal}, $\mathrm{Cok}Y\neq 0$ and $\mathrm{Cok}(Y_{1})\neq 0$, then $\mathrm{Cok}Y\cong F(Z)$ and $\mathrm{Cok}Y_{1}\cong F(Z_{1})$ for some indecomposable objects $Z, Z_{1}\in \mathcal{U}$. There are irreducible
morphisms $\mathrm{Cok}Y\rightarrow \mathrm{Cok}X$ and $\mathrm{Cok}Y_{1}\rightarrow \mathrm{Cok}X$, because of 4. Lemma \ref{lemaIdeal} and (1) Lemma 1.7 of \cite{SL}. Since $F$ is an equivalence of categories, we have irreducible morphisms $Z\rightarrow e_{i}\Lambda $ and
$Z_{1}\rightarrow e_{i}\Lambda $ in $\mathcal{U}$, therefore $Z$ and $Z_{1}$ are direct summands of $\mathrm{rad}(e_{i}\Lambda )$. By Propositions \ref{radicalLambdar} and \ref{radicalLambdac}, $Z\cong Z_{1}$. That is $Y\cong Y_{1}$, as we wanted to prove.
\end{proof}

Now we are ready to construct some suitable sections within the Auslander-Reiten component.

\begin{theorem}\label{constructComponent}
Let $\hat{\mathcal{C}}$ be the Auslander-Reiten component of $X_{F(e_{\text{\large \Fontlukas m}}\Lambda )}$. Then there exists a set of sections $\{\mathcal{S}_{i}\}_{i\in I}$ in $\hat{\mathcal{C}}$, where I is either the set of natural numbers or $I=\{1,2,\ldots,n\}$, with the following properties 
\begin{enumerate}[(1)]
\item If $X\in \mathcal{S}_{i}$ and $X$ is not projective, then $i>1$ and $\tau X\in \mathcal{S}_{i-1}$.
\item If $X\in \mathcal{S}_{i}$ and $X$ is not injective then $\tau^{-1} X \in \mathcal{S}_{i+1}$.
\item If $X\rightarrow Y$ is an irreducible morphism with $Y\in \mathcal{S}_{i}$ projective, then
$X\in \mathcal{S}_{i}$.
\item If $i\neq j $, then $\mathcal{S}_{i}\cap \mathcal{S}_{j}=\emptyset$.
\item $\displaystyle \hat{\mathcal{C}}=\bigcup _{i\in I}\mathcal{S}_{i}.$
\end{enumerate} 
\end{theorem}

\begin{proof}
Let $\mathcal{S}_{1}$ be the set of objects in $\hat{\mathcal{C}}$ which are pseudo hereditary projectives.
Clearly $X_{F(e_{\text{\large \Fontlukas m}}\Lambda )}\in \mathcal{S}_{1}$. 

We claim that $\mathcal{S}_{1}$ is a section. Consider $X\rightarrow Y$ an irreducible morphism with $X\in \mathcal{S}_{1}$ and $Y\in \hat{\mathcal{C}}$. If $Y$ is projective, take
$Z_{1}\rightarrow \cdots \rightarrow Z_{l}\rightarrow Y$ a chain of irreducible morphisms. By Proposition \ref{isomTroughIrreducible}, $X\cong Z_{l}$ and $X$ is pseudo hereditary projective, so $Z_{1},\ldots,Z_{l-1}$ are projectives, thus $Y\in \mathcal{S}_{1}$. Now, if $Y$ is not projective, there is an irreducible morphism $\tau Y\rightarrow X$, then $\tau Y$ is a pseudo hereditary projective, so $\tau Y\in \mathcal{S}_{1}$. This proves that $\mathcal{S}_{1}$ is a section.

Observe that $\mathcal{S}_{1}$ is connected because for any $X\in \mathcal{S}_{1}$, there is a chain of irreducible morphisms
from $X_{F(e_{\text{\large \Fontlukas m}}\Lambda )}$ to $X$.

Suppose now we have constructed $\mathcal{S}_{1}, \ldots ,\mathcal{S}_{l}$ with properties $(1)$, $(3)$ and $(4)$ for 
$i=1, \ldots ,l$ and $(2)$ for $i=1,\ldots,l-1$.

First, assume that not all objects in $\mathcal{S}_{l}$ are injective. Consider two sets
$$\underline{S}_{l}=\{Y\in \hat{\mathcal{C}} | \tau Y\in \mathcal{S}_{l}\};$$ 
and $\mathcal{T}_{l}$, the set of projective indecomposable objects $Z$ for which there is a chain of irreducible morphisms
$X\rightarrow Z_{1}\rightarrow \cdots \rightarrow Z_{t}=Z$ with  $Z_{1}, \ldots ,Z_{t-1}$ projectives, $t\in \mathbb{N}$ and $X\in \underline{\mathcal{S}}_{l}$.

Take 
$$\mathcal{S}_{l+1}=\underline{S}_{l}\cup \mathcal{T}_{l}.$$  
Suppose there is an object $X\in \mathcal{S}_{j}\cap \mathcal{S}_{l+1}$ for some $j<l+1$. If $X$ is not projective, $X\in \underline{S}_{l}$, then $\tau X\in \mathcal{S}_{l}$ and $\tau X\in \mathcal{S}_{j-1}$ which contradicts our induction hypothesis, therefore $X$ must be projective. But then $X\in \mathcal{T}_{l}$, so there is a chain of irreducible morphisms
  $Z\rightarrow X_{1}\rightarrow \cdots \rightarrow X_{l}=X$ where $X_{1}, \ldots ,X_{l}$ are projectives and 
  $Z\in \underline{\mathcal{S}}_{l}$. Here $X\in \mathcal{S}_{j}$, then by $(3)$, $X_{l-1}, \ldots , X_{1}, Z$ are in $\mathcal{S}_{j}$, so $Z\in \mathcal{S}_{l+1}\cap \mathcal{S}_{j}$ with $Z$ non-projective, but, by our previous case, this can not occur. Therefore $\mathcal{S}_{l+1}\cap \mathcal{S}_{j}=\emptyset $ for $j<l+1$.
  
Let us prove that $\mathcal{S}_{l+1}$ is a section. Take $X\in \mathcal{S}_{l+1}$, we have two cases $X\in \underline{\mathcal{S}}_{l}$ or $X\in \mathcal{T}_{l}$. Consider the first case and $X\rightarrow Y$ an irreducible morphism with $Y\in \hat{\mathcal{C}}$. If $Y$ is projective, then by definition $Y\in \mathcal{T}_{l}\subset \mathcal{S}_{l+1}$. If $Y$ is not projective, then there is an
  irreducible morphism $\tau X\rightarrow \tau Y$, with $\tau X\in \mathcal{S}_{l}$, then either
  $\tau ^{2}Y\in \mathcal{S}_{l}$ or $\tau Y\in \mathcal{S}_{l}$, but not both, so either
  $\tau Y\in \mathcal{S}_{l+1}$ or $Y\in \mathcal{S}_{l+1}$ but not both.
  
Now if $X\in \mathcal{T}_{l}$ there is a chain $Z\rightarrow X_{1}\rightarrow \cdots \rightarrow X_{t}=X$ of irreducible morphisms with $Z\in \underline{\mathcal{S}}_{l}$ and $X_{1},\ldots ,X_{t}$ projectives. Suppose there is an
  irreducible morphism $X\rightarrow Y$, with $Y\in \hat{\mathcal{C}}$. Then if $Y$ is projective
  $Y\in \mathcal{T}_{l}\subset \mathcal{S}_{l+1}$. If $Y$ is not projective there is an irreducible morphism
  $\tau Y\rightarrow X$, by Proposition \ref{isomTroughIrreducible}, $\tau Y=X_{t-1}\in \mathcal{T}_{l}\subset \mathcal{S}_{l+1}$.
  
Observe that it can not happen that both $Y$ and $\tau Y$ are in $\mathcal{S}_{l+1}$, otherwise $Y$ is not projective, so $Y\in \underline{\mathcal{S}}_{l}$, and $\tau Y\in \mathcal{S}_{l}\cap \mathcal{S}_{l+1}$, but we have already proved that the above intersection is empty.
  
Using the induction hypothesis and the following statements we prove that the set of sections $\mathcal{S}_{1},..,\mathcal{S}_{l+1}$ satisfies four of the required conditions:
  \begin{enumerate}[(1)]
  \item If $X\in \mathcal{S}_{l+1}$ and $X$ is not projective, then $X\in \underline{S}_{l}$, so
  $\tau X \in \mathcal{S}_{l}$.
  \item If $X\in \mathcal{S}_{l}$ is a not injective object, $\tau (\tau ^{-1}X)=X$, therefore
  $\tau ^{-1}X\in \underline{\mathcal{S}}_{l}\subset \mathcal{S}_{l+1}$.
  \item Suppose we have an irreducible morphism $X\rightarrow Y$ with $X\in \hat{\mathcal{C}}$ and $Y$ a projective which lies in $\mathcal{S}_{l+1}$. Then there is a chain of irreducible morphisms $Z\rightarrow Z_{1}\rightarrow \cdots \rightarrow Z_{t}\rightarrow Y$ with $Z\in \underline{\mathcal{S}}_l$ and $Z_{1},\ldots ,Z_{t}$ projectives, then by Proposition \ref{isomTroughIrreducible},
  $X=Z_{t}$, so $X\in \mathcal{T}_{l}\subset \mathcal{S}_{l+1}$.
  \item We have already proved that $\mathcal{S}_{i}\cap \mathcal{S}_{l+1}=\emptyset$, for $1\leq i<l+1$.   
\end{enumerate}  
If for some $l\in \mathbb{N}$, all objects in $\mathcal{S}_{l}$ are injectives we set $I=\{1,\ldots,l\}$. 

At last, let us prove that $\displaystyle \hat{\mathcal{C}}=\bigcup _{i\in I}\mathcal{S}_{i}$. It is enough to prove that if $\displaystyle Y\in \bigcup _{i\in I}\mathcal{S}_{i}$ and $W\rightarrow Y$ or $Y\rightarrow W$ are irreducible morphisms with $W\in \hat{\mathcal{C}}$, then $\displaystyle W\in \bigcup _{i\in I}\mathcal{S}_{i}.$
  
We have $Y\in \mathcal{S}_{i}$ for some 
$i\in I$, if $Y\rightarrow W$ is an irreducible morphism then either $W\in \mathcal{S}_{i}$ or $W$ is not projective and $\tau W \in \mathcal{S}_{i}$. 

In the second case, $\tau W$ is not injective, and by $(2)$, $W=\tau ^{-1}(\tau W)\in \mathcal{S}_{i+1}$. Notice that if all objects in $\mathcal{S}_{l}$ are injectives, $i\leq l-1$. Anyway, we have $\displaystyle W\in \bigcup _{i\in I}\mathcal{S}_{i}$.

If $W\rightarrow Y$ is an irreducible morphism and $Y$ is projective, by
$(3)$, $W\in \mathcal{S}_{i}$, if $Y$ is not projective, then by $(1)$, $\tau Y \in \mathcal{S}_{i-1}$ and
we have an irreducible morphism $\tau Y\rightarrow W$, then by the above case $\displaystyle W\in \bigcup _{i\in I}\mathcal{S}_{i}$.

This finishes our proof.
  \end{proof}

\begin{corollary}\label{tauNesProyectivo}
If $X\in \hat{\mathcal{C}}$, there is a non negative integer $n(X)$ such that
$\tau ^{n(X)}X$ is projective.
\end{corollary}

\begin{proof}
Given $X\in \hat{\mathcal{C}}$, by $(5)$ and $(4)$ of Theorem \ref{constructComponent} there is a unique $i(X)\in \mathbb{N}$ such that $X\in \mathcal{S}_{i(X)}$.

We will prove our corollary by induction on $i(X)$. If $i(X)=1$,  $X\in \mathcal{S}_{1}$,
then $X$ is projective and $n(X)=0$. Suppose now proved our statement for all $Y$ with $n(Y)<l$. We will prove our
corollary for $X$ with $i(X)=l$. We have that $X\in \mathcal{S}_{l}$; if $X$ is projective $n(X)=0$, otherwise $\tau X\in \mathcal{S}_{l-1}$, so there is a natural number $m=n(\tau X)$, such that $\tau ^{m}\tau X$ is projective. Then
$n(X)=n(\tau X)+1$.
\end{proof}

We can give another property of the Auslander-Reiten component $\hat{\mathcal{C}}$.

\begin{proposition}\label{NoCycles}
There are not directed cycles in $\hat{\mathcal{C}}$.
\end{proposition}

\begin{proof}
Suppose there is $X_{1}\rightarrow X_{2}\rightarrow \cdots \rightarrow X_{u}\rightarrow  X_{1}$, a directed cycle of irreducible morphisms in $\hat{\mathcal{C}}$. Using the notation of the proof of Corollary \ref{tauNesProyectivo}, we have that $X_1\in \mathcal{S}_{i(X_1)}$ and $X_2\in \mathcal{S}_{i(X_1)}$ or $\tau X_2\in \mathcal{S}_{i(X_1)}$, in both cases, $i(X_1)\leq i(X_2)$. Following the same idea with all the irreducible morphisms of the cycle, $i(X_{1})=i(X_{2})= \cdots =i(X_{u})$.

Therefore all the $X_{i}$ are in the same $\mathcal{S}_{l}$ for some $l\in \mathbb{N}$. If some $X_{s}$ is not projective, then we have an irreducible morphism
$\tau X_{s}\rightarrow X_{v}$ with $v\equiv s-1$ module $u$. By Proposition \ref{isomTroughIrreducible}, $\tau X_{s}=X_{v^{\prime}}$ with $v^{\prime }\equiv s-2$ module $u$, therefore $X_{s}$ and $\tau X_{s}$ are objects in $\mathcal{S}_{l}$, which contradicts that it is a section. We conclude that all the $X_{i}$ in the cycle are projectives. 

Each $X_{i}=X_{F(e_{i}\Lambda )}$, therefore we have a cycle of irreducible morphisms between projective modules in $\mathcal{U}$. Every irreducible morphism between projectives in $\mathcal{U}$ is a monomorphism, so there are not such cycle. This implies that we can not have directed cycles of irreducibles in $\hat{\mathcal{C}}$.
\end{proof}

We relate two objects $X, Y\in \hat{\mathcal{C}}$ as follows: $X\leq Y$ if and only if $X=Y$ or there is a chain of irreducible morphisms from $X$ to $Y$. This is a partial order in $\hat{\mathcal{C}}$ which allows us to prove more properties of the component.

\begin{proposition}\label{compIrrmorphism}
Suppose $X\in \hat{\mathcal{C}}$, and $f:Y\rightarrow X$ a non-zero morphism in $\mathcal{M}$, with
  $Y$ indecomposable. Then, there exists $Y_{1}\in \hat{\mathcal{C}}$ such that $Y\cong Y_{1}$. Moreover if $f$ is not an isomorphism then $f$ is a sum of compositions of irreducible morphisms.
\end{proposition}

\begin{proof}
We prove our proposition by induction on the order we just define in $\hat{\mathcal{C}}$. If $X\in \hat{\mathcal{C}}$ is a minimal object, $X=X_{F(e_{\text{\large \Fontlukas m}}\Lambda )}$, then $f$ is an isomorphism and $Y\cong X\in \hat{\mathcal{C}}$. 

Suppose now our proposition has been proved for all $Z<X$. Take
  $$\bigoplus _{i=1}^{r}X_{i}\xrightarrow{(u_{1}, \ldots ,u_{r})} X$$
a sink morphism with $X_{i}\in \hat{\mathcal{C}}$, for all $i\in\{1,2,\ldots, r\}$. We may assume $f$ is not an isomorphism, therefore there is a morphism $t=(t_{1},...,t_{r})^{T}:Y\rightarrow \bigoplus _{i=1}^{r}X_{i}$ such that $f=\sum _{i=1}^{r}u_{i}t_{i}.$

Since $u_{i}:X_{i}\rightarrow X$ is an irreducible morphism, $X_{i}<X$, for all $i\in\{1,2,\ldots, r\}$. Here $f$ is not zero, therefore $u_{i}\neq 0$, for some $i\in\{1,2,\ldots, r\}$. By our induction hypothesis there is a $Y_{1}\in \hat{\mathcal{C}}$ such that $Y_{1}\cong Y$. Moreover  each $t_{i}$ is  zero or an isomorphism or a sum of compositions of irreducible morphism, then $f$ is a sum of compositions of irreducible morphisms. The proof is complete.
\end{proof}

\begin{proposition}\label{ExtNulo}
If $X\in \hat{\mathcal{C}}$, then $\mathrm{Ext}_{\mathcal{M}}(X,X)=0$.
\end{proposition}

\begin{proof}
Suppose we have a non trivial exact sequence in $\mathcal{M}$:
\begin{equation}\label{exactSeq}
X\stackrel{f}{\longrightarrow }E\stackrel{g}{\longrightarrow }X.
\end{equation}
Take $Y$ an indecomposable summand of $E$ such that $fp_{Y}\neq 0$, where $p_{Y}:E\rightarrow Y$ is the projection. Consider now $\sigma _{Y}:Y\rightarrow E$ the inclusion. We claim that $g\sigma _{Y}\neq 0$. Otherwise there is a $s:Y\rightarrow X$ such that $fs=\sigma _{Y}$. Then $id_{Y}=p_{Y}\sigma _{Y}=p_{Y}fs$. Thus $Y$ is a direct summand of $X$ which is indecomposable. Therefore $s$ is an isomorphism, then $s^{-1}=p_{Y}f$. We have $sp_{Y}:E\rightarrow X$ and $(sp_{Y})f=s(p_{Y}f)=id_{X}$. But this can not happen, because $(\ref{exactSeq})$ is a non trivial exact sequence. Therefore $g\sigma _{Y}\neq 0$. 
 
Then we have a non zero non isomorphism from $Y$ to $X$, so by Proposition \ref{compIrrmorphism},
there are $Y_1\in \hat{\mathcal{C}}$ with $Y\cong Y_{1}$, and a finite chain of irreducible morphisms from $Y_1$ to $X$.

In a similar way, $p_{Y}f$ is not an isomorphism. We also have a non zero non isomorphism from $X$ to $Y$, so again by Proposition \ref{compIrrmorphism}, there is a finite chain of irreducible morphisms from $X$ to $Y_1$, so we have an oriented cycle in $\hat{\mathcal{C}}$, which contradicts Proposition \ref{NoCycles}. 
  
Therefore there are not non trivial exact sequences of the form $(\ref{exactSeq})$, which completes the proof.
  \end{proof}

In the next result we characterize an object of $\hat{\mathcal{C}}$, using its coordinates.

\begin{proposition}\label{determineObjectByCoordinates}
Let $X\in \hat{\mathcal{C}}$ and $Y$ be an indecomposable object in $\mathcal{M}$ such that
$\mathrm{cd}(X)=\mathrm{cd}(Y)$. Then
\begin{enumerate}
\item[(i)] If $Y\in \hat{\mathcal{C}}$, then $X\cong Y$.
\item[(ii)] If $Y$ is a projective or an injective object of $\mathcal{M}$, then $X\cong Y$
\item[(iii)] If $\mathrm{End}_{\mathcal{M}}(X)\cong \mathrm{End}_{\mathcal{M}}(Y)$, then $X\cong Y$.
\end{enumerate}
\end{proposition}

\begin{proof}
By Proposition \ref{ExtNulo}, $\mathrm{Ext}_{\mathcal{M}}(X,X)=0$.  In $(i)$ and $(ii)$, $\mathrm{Ext}_{\mathcal{M}}(Y,Y)=0$. Then our result follows in these two cases by Proposition \ref{IsomObjects}. Item $(iii)$ follows from 
Corollary \ref{corIsomObjects}. 
\end{proof}

For $M$ and $N$, indecomposable objects in a Krull-Schmidt $\textsf{F}$-category $\mathcal{A}$, we denote
$$K(M)=\mathrm{End}_{\mathcal{A}}(M)/\mathrm{rad}\mathrm{End}_{\mathcal{A}}(M);$$
$$\mathrm{Irr}(M,N)=\mathrm{Hom}_{\mathcal{A}}(M,N)/\mathrm{rad}^{2}(M,N).$$
Clearly $\mathrm{Irr}(M,N)$ is a $K(N)$-$K(M)$-bimodule.

With the next definition and result, we will study some properties of $K(X)$ and $\mathrm{Irr}(X,Y)$, in the case $X, Y \in \hat{\mathcal{C}}$.

\begin{definition}
An irreducible morphism in a Krull-Schmidt category is called \textit{left homogeneous} if it has the form
$$M^{m}\rightarrow N;$$ 
with $M$ and $N$ indecomposable objects and $m\in \mathbb{N}$. 

Similarly an irreducible morphism is called \textit{right homogeneous} if it has the form 
$$M\rightarrow N^{n};$$ 
for some $M$ and $N$ indecomposable objects and $n\in \mathbb{N}$. 

A left homogeneous morphism $M^{m}\rightarrow N$ is called \textit{maximal} if for any irreducible morphism $M^{m^{\prime }}\rightarrow N$, we have $m^{\prime }\leq m$. 

The definition for a maximal right homogeneous morphism is similar.
  \end{definition}
  
From \cite{Bautista82}, if $M$ and $N$ are indecomposable objects in a Krull-Schmidt $\textsf{F}$-category $\mathcal{A}$, then there is a maximal left homogeneous morphism $M^{m}\rightarrow N$ if and only if $\mathrm{dim}_{K(M)}\mathrm{Irr}(M,N)=m$. Analogously, there is a maximal right homogeneous morphism $M\rightarrow N^{n}$ if and only if $n=\mathrm{dim}_{K(N)}\mathrm{Irr}(M,N)$. Moreover, we have the following result.
  
\begin{proposition}\label{dimDeBimodulos}
For $M$ and $N$, indecomposable objects in a Krull-Schmidt $\textsf{F}$-category $\mathcal{A}$, if there is an almost split sequence in $\mathcal{A}$ starting in $\tau N$ and ending in $N$, then
$$\mathrm{dim}_{K(M)}\mathrm{Irr}(M,N)=\mathrm{dim}_{K(M)}\mathrm{Irr}(\tau N,M).$$
\end{proposition}

\begin{proof}
Observe that there is an irreducible morphism $\tau N\rightarrow M^{u}$ if and only if $M^{u}$ is a direct summand of the middle term of the almost split sequence ending in $\tau N$, and this happen if and only if $M^{u}\rightarrow N$ is an irreducible morphism, therefore $\tau N\rightarrow M^{m}$ is maximal right homogeneous irreducible morphism if and only if $M^{m}\rightarrow N$ is a maximal left homogeneous irreducible morphism.
\end{proof}

\begin{lemma}\label{IsoIrrCok}
Let $X, Y$ be indecomposable objects in the category $\mathcal{M}$ such that
  $\mathrm{Cok}(X)\neq 0$ and $\mathrm{Cok}(Y)\neq 0$, then the functor $\mathrm{Cok}:\mathcal{M}\rightarrow \mathcal{V}$ induces an isomorphism of $\textsf{F}$-vector spaces.
$$ \mathrm{Cok}:\mathrm{Irr}(X,Y)\rightarrow \mathrm{Irr}(\mathrm{Cok}(X),\mathrm{Cok}(Y)).$$
\end{lemma}

\begin{proof}
The functor $\mathrm{Cok}:\mathcal{M}\rightarrow \mathcal{V}$ induces an equivalence
$\mathrm{Cok}:\overline{\mathcal{M}}\rightarrow \mathcal{V}$, where $\overline{\mathcal{M}}$ is the category $\mathcal{M}$ module the ideal $\mathcal{I}$, generated by those morphisms which factorizes through objects of the form $P\rightarrow 0$. By Lemma \ref{lemaIdeal} the ideal $\mathcal{I}$ is admissible in the sense of \cite{SL}. Then our result follows from the proof of  Proposition $2.9$ of \cite{SL}.
\end{proof}

\begin{proposition}\label{KXisomKtauX}
If $X$ is a non projective indecomposable object in $\mathcal{M}$, then
$$K(X)\cong K(\tau X);$$
as $\textsf{F}$-algebras.
\end{proposition}

\begin{proof}
Consider  
$$\tau X\rightarrow Y\rightarrow X;$$
an almost split sequence in $\mathcal{M}$. Suppose $\mathrm{Cok}(X)\neq 0$. Since the above almost split sequence is an exact sequence, $Y$ is not of the form $P\rightarrow 0$, therefore $\mathrm{Cok}(Y)\neq 0$.  Here
   $\tau X$ is not injective, therefore $\mathrm{Cok}(\tau X)\neq 0$.
   
Proposition \ref{ARseq} gives us the following  almost split sequence in $\mathcal{V}$:
$$0\rightarrow \mathrm{Cok}(\tau X)\rightarrow \mathrm{Cok}(Y)\rightarrow \mathrm{Cok}(X)\rightarrow 0.$$
   
Then $K(\mathrm{Cok}(\tau X))\cong K(\mathrm{Cok}(X))$ as $\textsf{F}$-algebras, by Proposition \ref{EndXisomEndtauX}. Using 7 of Lemma \ref{lemaIdeal}, we have the following isomorphisms of $\textsf{F}$-algebras:
    $$K(\tau X)\cong K(\mathrm{Cok}(\tau X))\cong K(\mathrm{Cok}(X))\cong K(X).$$

Now if $\mathrm{Cok}(X)=0$, then $X=T(e_{i}\Lambda )$ for some $i\in \mathcal{P}$. Taking into account Proposition \ref{ASsequence}, $\tau X=X_{D(e_{i}\Lambda )}$. Therefore:
    $$K(X)\cong K(e_{i}\Lambda )\cong K(D(\Lambda e_{i}))\cong K(\tau X).$$
    The proof is complete. 
\end{proof}

\begin{proposition}\label{irrIgualtauIrr}
If $X$ and $Y$ are non projective indecomposable objects in $\mathcal{M}$
    $$\mathrm{dim}_{\textsf{F}}\mathrm{Irr}(X,Y)=\mathrm{dim}_{\textsf{F}}(\tau X, \tau Y).$$
\end{proposition}

\begin{proof}
Using twice  Proposition \ref{dimDeBimodulos} we have
$$\mathrm{dim}_{\textsf{F}}\mathrm{Irr}(X,Y)=\mathrm{dim}_{\textsf{F}}K(X)\, \mathrm{dim}_{K(X)}\mathrm{Irr}(\tau Y,X) = \mathrm{dim}_{\textsf{F}}\mathrm{Irr}(\tau Y, X)$$
$$= \mathrm{dim}_{\textsf{F}}K(\tau Y)\, \mathrm{dim}_{K(\tau Y)}\mathrm{Irr}(\tau X,\tau Y) =\mathrm{dim}_{\textsf{F}}(\tau X, \tau Y).$$ 
\end{proof}

\begin{proposition}\label{oneDim}
If $X$ and $Y$ are in $\hat{\mathcal{C}}$, then $\mathrm{Irr}(X,Y)$ is zero, or one-dimensional over $K(X)$ or over $K(Y)$.
\end{proposition}

\begin{proof}
Assume $\mathrm{Irr}(X,Y)\neq 0$. By Corollary \ref{tauNesProyectivo} we know that there are non negative integers $n(X)$ and $n(Y)$ such that
    $\tau ^{n(X)}X$ and $\tau ^{n(Y)}Y$ are projectives.  Then we have two cases: 
    
    (1) $n(X)\geq n(Y)$, 
    
    (2) $n(X)<n(Y)$.
    
Consider case (1). Applying Proposition \ref{irrIgualtauIrr} $n(Y)$ times we obtain:
$$\mathrm{dim}_{\textsf{F}}\mathrm{Irr}(X,Y)=\mathrm{dim}_{\textsf{F}}\mathrm{Irr}(Z,X_{F(e_{i}\Lambda )});$$
where $Z=\tau ^{n(Y)}X$ and $X_{F(e_{i}\Lambda )}=\tau ^{n(Y)}Y$, for some $i\in \mathscr{P}$. By Lemma \ref{lemaIdeal}, $\mathrm{Cok}Z\neq 0$ and 
$$\mathrm{dim}_{\textsf{F}}\mathrm{Irr}(Z,X_{F(e_{i}\Lambda )})=\mathrm{dim}_{\textsf{F}}\mathrm{Irr}(\mathrm{Cok}(Z), F(e_{i}\Lambda ))=\mathrm{dim}_{\textsf{F}}\mathrm{Irr}(W,e_{i}\Lambda );$$ 
where $F(W)=\mathrm{Cok}(Z)$.

Therefore $W$ is a direct summand of $\mathrm{rad}(e_{i}\Lambda)$ and by Propositions \ref{radicalLambdar} and \ref{radicalLambdac}, $\mathrm{rad}(e_{i}\Lambda )\cong W^{\ell}$ with $\ell=1$ or $\ell=p$.

Then 
$$\mathrm{dim}_{\textsf{F}}\mathrm{Irr}(X,Y)=\dim _{\textsf{F}}\mathrm{Irr}(W,e_{i}\Lambda )=\ell\mathrm{dim}_{\textsf{F}}K(W).$$
    
Moreover 
    $$\mathrm{dim}_{\textsf{F}}\mathrm{Irr}(X,Y)=\mathrm{dim}_{\textsf{F}}K(X)\mathrm{dim}_{K(X)}\mathrm{Irr}(X,Y);$$
    and $\mathrm{dim}_{\textsf{F}}K(X)=\mathrm{dim}_{\textsf{F}}K(W)$, so
    $$\mathrm{dim}_{K(X)}\mathrm{Irr}(X,Y)=\ell.$$

Then if $\ell=1$, we have that $\mathrm{Irr}(X,Y)$ is one-dimensional over $K(X)$. If $\ell=p$, then $\Lambda =\Lambda ^{(r)}$, with $i$ a weak element of $\mathscr{P}$, and 
    $$\mathrm{dim}_{\textsf{F}}K(e_{i}\Lambda )=\mathrm{dim}_{\textsf{F}}e_{i}\Lambda e_{i}=p.$$
    
We have 
    $$p\,\mathrm{dim}_{\textsf{F}}\mathrm{soc}(W)=\mathrm{dim}_{\textsf{F}}(W^{p})=\mathrm{dim}_{\textsf{F}}\mathrm{soc}\,\mathrm{rad}(e_{i}\Lambda )$$
  $$=\mathrm{dim}_{\textsf{F}}\mathrm{soc}(e_{i}\Lambda )=\mathrm{dim}_{\textsf{F}}(e_{i}\Lambda e_{\text{\large \Fontlukas m}})=p.$$

This implies $\mathrm{dim}_{\textsf{F}}\mathrm{soc}(W)=1$, moreover $\mathrm{soc}(W)$ is a $K(W)$-space, thus $K(W)=\textsf{F}$ and consequently $K(X)=\textsf{F}$. Then 
$$\mathrm{dim}_{\textsf{F}}\mathrm{Irr}(X,Y)=p=\mathrm{dim}_{\textsf{F}}K(e_{i}\Lambda )=\mathrm{dim}_{\textsf{F}}K(Y);$$
  therefore $\mathrm{Irr}(X,Y)$ is one-dimensional over $K(Y)$.

Now consider $n(Y)>n(X)$, then $Y$ is not projective. By Proposition \ref{dimDeBimodulos} we have
  $$\mathrm{dim}_{\textsf{F}}\mathrm{Irr}(X,Y)=\mathrm{dim}_{\textsf{F}}\mathrm{Irr}(\tau Y,X)$$
  here $n(\tau Y)\geq n(X)$, then by $(1)$, $\mathrm{dim}_{\textsf{F}}\mathrm{Irr}(\tau Y,X)$ is equal to
  $\mathrm{dim}_{\textsf{F}}K(X)$ or $\mathrm{dim}_{\textsf{F}}K(\tau Y)$. Therefore $\mathrm{dim}_{\textsf{F}}\mathrm{Irr}(X,Y)$
  is equal to $\mathrm{dim}_{\textsf{F}}K(Y)$ or $\mathrm{dim}_{\textsf{F}}K(X)$ (see Proposition \ref{KXisomKtauX}), this implies that
  $\mathrm{Irr}(X,Y)$ is one-dimensional over $K(X)$ or over $K(Y)$, as we wanted to prove.
\end{proof}


\section{Representations and Corepresentations}

Let $\mathscr{P}$ be a $p$-equipped partially ordered set and $\textsf{G}/\textsf{F}$ a normal extension of fields of degree equal to $p$. We have the algebras $\Lambda ^{(r)}=\Lambda (\mathcal{R}^{(r)})$ and  $\Lambda ^{(c)}=\Lambda (\mathcal{R}^{(c)})$, where $\mathcal{R}^{(r)}$ and $\mathcal{R}^{(c)}$ are the admissible systems defined in section $2$. Moreover we
have the categories $\mathcal{U}^{(r)}$, $\mathcal{V}^{(r)}$, and $\mathcal{M}^{(r)}$, associated to $\Lambda ^{(r)}$ and the corresponding categories $\mathcal{U}^{(c)}$, $\mathcal{V}^{(c)}$, $\mathcal{M}^{(c)}$ associated to
$\Lambda ^{(c)}$. Denote by $\hat{\mathcal{C}}^{(r)}$, the Auslander-Reiten component in $\mathcal{M}^{(r)}$ of the object
$X_{F(e_{\text{\large \Fontlukas m}}\Lambda ^{(r)})}$, and by $\hat{\mathcal{C}}^{(c)}$ the component of  $X_{F(e_{\text{\large \Fontlukas m}}\Lambda ^{(c)})}$ in $\mathcal{M}^{(c)}$. The main purpose of this section is to prove the existence of a bijection $\kappa :\hat{\mathcal{C}}^{(r)}\rightarrow \hat{\mathcal{C}}^{(c)}$ which is an isomorphism among the underlying graphs.

\begin{definition}
An indecomposable object $Z$ in $\mathcal{U}^{(r)}$ or  in $\mathcal{V}^{(r)}$ or in $\mathcal{M}^{(r)}$ is called strong if $K(Z)\cong \textsf{F}$ and weak if $K(Z)\cong \textsf{G}$. 

If $W$ is an indecomposable object in 
$\mathcal{U}^{(c)}$ or  in $\mathcal{V}^{(c)}$ or in $\mathcal{M}^{(c)}$, it is called strong if $K(W)\cong \textsf{G}$ and weak if $K(W)\cong \textsf{F}$.
\end{definition}

Take $\Lambda $ equal to $\Lambda ^{(r)}$ or $\Lambda ^{(c)}$. Then
for $i\in \mathscr{P}$ we have $\mathrm{End}_{\mathcal{M}}(0\rightarrow e_{i}\Lambda )\cong e_{i}\Lambda e_{i}$
and $\mathrm{End}_{\mathcal{M}}(e_{i}\Lambda \rightarrow 0)\cong e_{i}\Lambda e_{i}$. Therefore $T(e_{i}\Lambda )=(0\rightarrow e_{i}\Lambda )$ is strong (weak) if and only if $i$ is strong (weak) and also $I(e_{i}\Lambda )=(e_{i}\Lambda \rightarrow 0)$ is strong (weak) if and only if $i$ is strong (weak).

\begin{proposition}\label{objectWS}
If $X$ is an object of $\hat{\mathcal{C}}^{(r)}$ or $\hat{\mathcal{C}}^{(c)}$ then $X$ is strong or weak.
\end{proposition}

\begin{proof}
By Corollary \ref{tauNesProyectivo} there is a non-negative integer $n(X)$ such that $\tau ^{n(X)}X=X_{F(e_{i}\Lambda )}$ for some $i\in \mathscr{P}$, using Propositions \ref{EndXisomEndtauX} and \ref{KXisomKtauX},
$$K(X)\cong K(X_{F(e_{i}\Lambda )})\cong K(F(e_{i}\lambda ))\cong K(e_{i}\lambda )\cong e_{i}\Lambda e_{i};$$
as $\textsf{F}$-algebras, therefore $K(X)\cong \textsf{F}$ or
$K(X)\cong \textsf{G}$ as $\textsf{F}$-algebras. This proves our assertion.
\end{proof}

For a $p$-equipped poset $\mathscr{P}$ and a point $i\in \mathscr{P}$, let us denote by $\mathscr{P}^{\geq i}$ the subposet of $\mathscr{P}$ consisting in all the points greater or equal than $i$.

\begin{lemma}\label{lemmaSlender}
Let $\Lambda $ be $\Lambda ^{(r)}$ or $\Lambda ^{(c)}$, then for $i\in \mathscr{P}$, the object
$\mathrm{rad}(e_{i}\Lambda )$ is projective if and only if in the Hasse diagram of $\mathscr{P}$ there is only one arrow $i\rightarrow j$, for some $j\in \mathscr{P}$, and for all $u\geq j$, $i\leq ^{\ell}u$ if and only if $j\leq ^{\ell}u$.
\end{lemma}

\begin{proof}
We have a minimal projective presentation of $\mathrm{rad}(e_{i}\Lambda )$:
$$\bigoplus _{i\rightarrow s}(e_{s}\Lambda )^{n(s)}\rightarrow \mathrm{rad}(e_{i}\Lambda )\rightarrow 0;$$
where $n(s)=\mathrm{dim}_{\textsf{F}}(e_{i}\Lambda e_{s})/\mathrm{dim}_{\textsf{F}}(e_{s}\Lambda e_{s})$. From Propositions \ref{radicalLambdar} and \ref{radicalLambdac}, we have $\mathrm{rad}(e_{i}\Lambda )=Z^{n}$, with $Z$ indecomposable and $n=1$ or $n=p$.

Suppose $\mathrm{rad}(e_{i}\Lambda )$ is a projective module. Therefore, for some $j\in \mathscr{P}$, in the Hasse diagram of $\mathscr{P}$ there is only one arrow $i\rightarrow j$  with $\mathrm{rad}(e_{i}\Lambda )\cong (e_{j}\Lambda)^{n(j)}$, that is $Z=e_{j}\Lambda$ and $n(j)=n$. There are two options for $n$:

\begin{enumerate}
\item When $n=1$, we are in the case $\mathrm{rad}(e_{i}\Lambda )\cong e_{j}\Lambda $, and
$$\mathrm{dim}_\textsf{F}(e_{i}\Lambda e_{\text{\large \Fontlukas m}})=\mathrm{dim}_\textsf{F}(e_{j}\Lambda e_{\text{\large \Fontlukas m}}).$$
Then $i$ is weak (strong) if and only if $j$ is weak (strong), and for all $u\geq j$ we have
$$\mathrm{dim}_\textsf{F}(e_{i}\Lambda e_{u})=\mathrm{dim}_\textsf{F}(e_{j}\Lambda e_{u}).$$
The right side of the above equality coincides with $\ell$, where $i\leq ^{\ell}u$ and the left side coincides with
$\ell^{\prime }$ where $j\leq ^{\ell^{\prime }}u$. Therefore for all $u\geq j$, $i\leq ^{\ell}u$ if and only if
$i\leq ^{\ell}u$.

\item In the case $\ell=p$, by Propositions \ref{radicalLambdar} and \ref{radicalLambdac}, $\Lambda = \Lambda ^{(r)}$, the point $i$ is weak, $j$ is strong and for all $u\geq j$, $i\leq ^{p}u$, and since $j$ is strong for all $u\geq j$, $j\leq ^{p}u$. 
\end{enumerate}
In both cases, we have proved that $\mathscr{P}$ satisfies the conditions of our proposition.

Now suppose that in the Hasse diagram of $\mathscr{P}$ there is only one arrow $i\rightarrow j$, for some $j\in \mathscr{P}$, and for all $u\geq j$, $i\leq ^{\ell}u$ if and only if $j\leq ^{\ell}u$. We claim that $\mathrm{rad}(e_{i}\Lambda )$ is projective. First, we have an epimorphism
$$e_{j}\Lambda ^{n(j)}\rightarrow \mathrm{rad}(e_{i}\Lambda ).$$
As $i\leq ^{\ell}j $ if and only if $j\leq ^{\ell}j$, then $\ell=1$ or $\ell=p$.

If $\ell=1$, the points $i$ and $j$ are weak, then by Propositions \ref{radicalLambdar} and \ref{radicalLambdac}, $\mathrm{rad}(e_{i}\Lambda )$ is indecomposable. In this case $n(j)=1$. Consider $u\in \mathscr{P}$ such that $u\geq j$, we have $i\leq ^{l}u$ if and only if $j\leq ^{l}u$, then $\mathrm{dim}_\textsf{F}(e_{i}\Lambda e_{u})=\mathrm{dim}_\textsf{F}(e_{j}\Lambda e_{u})$. Thus
$$\mathrm{rad}(e_{i}\Lambda )\cong e_{j}\Lambda .$$

If $\ell=p$, then $j$ is strong. Therefore $i\leq^p u$ and $j\leq^p u$ for every $u\geq j$. 

When $i$ is strong, $\mathrm{dim}_\textsf{F}(e_{i}\Lambda e_{u})=\mathrm{dim}_\textsf{F}(e_{j}\Lambda e_{u})$, also $n(j)=1$ and there is an isomorphism $e_{j}\Lambda \rightarrow \mathrm{rad}(e_{i}\Lambda)$.

Now suppose $i$ is a weak point and $\Lambda =\Lambda ^{(r)}$. Then
$\mathrm{dim}_\textsf{F}(e_{i}\Lambda e_{j})=p$ and $\mathrm{dim}_\textsf{F}(e_{j}\Lambda e_{j})=1$, so $n(j)=p$ and we have an epimorphism $(e_{j}\Lambda )^{p}\rightarrow \mathrm{rad}(e_{i}\Lambda )$. For $u>j$ and $u$ weak, we have $p^{2}=\mathrm{dim}_\textsf{F}(e_{i}\Lambda e_{u})=p\mathrm{dim}_\textsf{F}(e_{j}\Lambda e_{u})$. In the case $u\geq j$, with $u$ strong,
one has $p=\mathrm{dim}_\textsf{F}(e_{i}\Lambda e_{u})=p\mathrm{dim}_\textsf{F}(e_{j}\Lambda e_{u})$. Therefore in this case
$(e_{j}\Lambda )^{p}\cong \mathrm{rad}(e_{i}\Lambda )$. 

If $\Lambda =\Lambda ^{c}$ then $p=\mathrm{dim}_\textsf{F}(e_{i}\Lambda e_{j})=\mathrm{dim}_\textsf{F}(e_{j}\Lambda e_{j})$. So there is an epimorphism
$e_{j}\Lambda \rightarrow \mathrm{rad}(e_{i}\Lambda )$. Moreover for $u\geq j$, we have $p=\mathrm{dim}_\textsf{F}(e_{i}\Lambda e_{u})=\mathrm{dim}_\textsf{F}(e_{i}\Lambda e_{u})$. This implies 
$$e_{j}\Lambda \cong \mathrm{rad}(e_{i}\Lambda).$$
\end{proof}

Let us relate the hereditary projective $\Lambda$-modules, with the shape of a $p$-equipped poset.

\begin{definition}
A subposet $\mathscr{T}$ of a $p$-equipped poset $\mathscr{P}$, is called \textit{slender} if its Hasse diagram has the form
$$\mathscr{T}=\{i_{1}\rightarrow i_{2}\rightarrow ...\rightarrow i_{s}\rightarrow j_{1}\rightarrow ...\rightarrow j_{t}\};$$
with $i_a\leq^1 i_b$, for $a, b\in \{1,2,\ldots, s\}$ such that $a<b$, and for $c\in \{1,2,\ldots, t\}$ the point $j_c$ is strong.

Notice that $i_a$ is a weak point for all $a\in \{1,2,\ldots, s\}$.
\end{definition}

\begin{proposition}\label{equivSlender}
Let $\mathscr{P}$ be a $p$-equipped poset and $i\in \mathscr{P}$. Denote by $\mathscr{P}^{\geq i}$ the subposet of $\mathscr{P}$ consisting in all the points greater or equal than $i$. The following sentences are equivalent:
\begin{enumerate}[(1)]
\item $e_i\Lambda^{(r)}$ is an hereditary projective $\Lambda^{(r)}$-module.
\item $e_i\Lambda^{(c)}$ is an hereditary projective $\Lambda^{(c)}$-module.
\item The subposet $\mathscr{P}^{\geq i}$ is slender.
\end{enumerate}
\end{proposition}

\begin{proof}
We will prove first the equivalence of (1) and (3).

The proof will be done by induction on the cardinality of $\mathscr{P}^{\geq i}$. If this cardinality is one, then $e_{i}\Lambda^{(r)}$ is simple and clearly $(1)$ is equivalent to $(3)$. Suppose our claim is proved for those $j \in \mathscr{P}$ with cardinality of $\mathscr{P}^{\geq j}$ smaller than the cardinality of $\mathscr{P}^{\geq i}$. We will prove the equivalence of
$(1)$ and $(3)$ for $\mathscr{P}^{\geq i}$.

If $e_{i}\Lambda ^{(r)}$ is hereditary, $\mathrm{rad}\left(e_{i}\Lambda ^{(r)}\right)$ is projective. By Lemma \ref{lemmaSlender}, there is only one arrow $i\rightarrow j$, for some $j\in \mathscr{P}$ and $i\leq^\ell j$ if and only if $j\leq^\ell j$, so $i$ is weak if $j$ is weak. Then $\mathscr{P}^{\geq i}$ is slender because $\mathscr{P}^{\geq j}$ is slender by our induction hypothesis. 

Suppose now that $\mathscr{P}^{\geq i}$ is a slender subset of $\mathscr{P}$. Clearly, in the Hasse diagram of $\mathscr{P}^{\geq i}$ there is only one arrow $i\rightarrow j$, for some $j\in \mathscr{P}$, and for all $u\geq j$, $i\leq ^{\ell}u$ if and only if $j\leq ^{\ell}u$. From Lemma \ref{lemmaSlender} we have that $\mathrm{rad}\left(e_{i}\Lambda^{(r)}\right)$ is projective and it is isomorphic to $\left(e_{j}\Lambda^{(r)}\right)^{n}$ for $n=1$ or $n=p$. But $\mathscr{P}^{\geq j}$ is slender, so by induction hypothesis $e_{j}\Lambda ^{(r)}$ is hereditary, this implies that $e_{i}\Lambda^{(r)}$ is hereditary projective. We conclude that (1) and (3) are equivalent.

Using the same method, one can prove the equivalence of (2) and (3), which implies the equivalence of (1) and (2). The proof is complete.
\end{proof}    
  
\begin{remark}
From Propositions \ref{oneDim} and \ref{objectWS}, for every object $X\in \hat{\mathcal{C}}$, we have $\mathrm{rad}(\mathrm{End}_{\mathcal{M}}(X))=0$. 

For every object $M\in \mathcal{U}$ with $X_{F(M)}\in \hat{\mathcal{C}}$, then $\mathrm{rad}(\mathrm{End}_{\Lambda }(M))=0$ and $\mathrm{End}_{\Lambda }(M)\cong \textsf{F}$ or $\mathrm{End}_{\Lambda }(M)\cong \textsf{G}$, as $\textsf{F}$-algebras. Moreover, if $M$ is indecomposable, we have that $\mathrm{dim}_\textsf{F}(\mathrm{End}_{\Lambda }(M))$ divides $\mathrm{dim}_\textsf{F}(Me_{i})$ for all $i\in \mathscr{P}$.
\end{remark}

In the following for $\mathbb{Q}^{|\mathscr{P}|}$ we set, for $i\in \mathscr{P}$. the function
$\mathfrak{e}_{i}:\mathscr{P}\rightarrow \mathbb{Q}$ such that
$$\mathfrak{e}_{i}(j)=
\begin{cases}
i & \text{ if } j=i;\\
0 & \text{otherwise.}\\
\end{cases}$$

Consider two linear functions $w, s:\mathbb{Q}^{|\mathscr{P}|}\rightarrow \mathbb{Q}^{|\mathscr{P}|}$ defined as follows
$$w(\mathfrak{e}_{i})=
\begin{cases}
\vspace{3mm}\mathfrak{e}_{i} & \text{ if } i \text{ is weak};\\
\dfrac{1}{p}\mathfrak{e}_{i} & \text{ if } i \text{ is strong};\\
\end{cases}$$
$$s(\mathfrak{e}_{i})=
\begin{cases}
\mathfrak{e}_{i} & \text{ if } i \text{ is strong};\\
p\mathfrak{e}_{i} & \text{ if } i \text{ is weak}.\\
\end{cases}$$

\begin{proposition}\label{radicalCoordinates}
For $i\in \mathscr{P}$ consider $\mathrm{rad}\left(e_{i}\Lambda ^{(r)}\right)=T_{i}^{n}$, where $T_{i}$ is
indecomposable and $n =1$ or $n =p$. Suppose $X_{F\left(e_{i}\Lambda ^{(r)}\right)}\in \hat{\mathcal{C}}^{(r)}$, and
$X_{F\left(e_{i}\Lambda ^{(c)}\right)}\in \hat{\mathcal{C}}^{(c)}$. Then if $T_{i}$ is strong $\mathrm{rad}\left(e_{i}\Lambda ^{(c)}\right)$ is strong and
$$s\left(\mathrm{cd}\left(X_{F(T_{i})}\right)\right)=\mathrm{cd}\left(X_{F\left(\mathrm{rad}\left(e_{i}\Lambda ^{(c)}\right)\right)}\right).$$
If $T_i$ is weak then $\mathrm{rad}(e_{i}\Lambda )$ is weak and
$$ w\left(\mathrm{cd}\left(X_{F(T_{i})}\right)\right)=\mathrm{cd}\left(X_{F\left(\mathrm{rad}\left(e_{i}\Lambda{(c)}\right)\right)}\right).$$
\end{proposition}

\begin{proof}
Suppose $T_{i}$ is strong, then if $i$ is strong $\mathrm{rad}\left(e_{i}\Lambda ^{(c)}\right)$ is indecomposable with 
simple socle. Here $e_{i}\Lambda ^{(c)}e_{j}=\textsf{G}e_{i,j}$, for every $j\geq i$. The  left multiplication by elements of $\textsf{G}$ gives an
injective morphism of $\textsf{F}$-algebras, $\textsf{G}\rightarrow \mathrm{End}_{\Lambda }\left(\mathrm{rad}\left(e_{i}\Lambda ^{(c)}\right)\right)$. By Lemma \ref{functorres} we have an injective morphism of $\textsf{F}$-algebras 
$$\mathrm{End}_{\Lambda ^{(c)}}\left(\mathrm{rad}\left(e_{i}\Lambda ^{(c)}\right)\right) \rightarrow \mathrm{End}_{\Lambda ^{(c)}}\left(\mathrm{soc}\left(\mathrm{rad}\left(e_{i}\Lambda ^{(c)}\right)\right)\right)=\textsf{G};$$
then $\mathrm{End}_{\Lambda ^{(c)}}\left(\mathrm{rad}\left(e_{i}\Lambda ^{(c)}\right)\right)=\textsf{G}$, and $\mathrm{rad}\left(e_{i}\Lambda ^{(c)}\right)$ is strong. 

If $i$ is weak, since $T_{i}$ is strong, by Proposition \ref{radicalLambdar}, $i<^{p}j$ for all $j>i$, therefore $\mathrm{rad}\left(e_{i}\Lambda ^{(c)}\right)e_{j}=\textsf{G}e_{i,j}$. Then as before we obtain that
the endomorphism ring of $\mathrm{rad}\left(e_{i}\Lambda ^{(c)}\right)$ is $\textsf{G}$, so $\mathrm{rad}\left(e_{i}\Lambda ^{(c)}\right)$ is strong.

Now suppose $T_{i}$ is weak, then $i$ is weak, and there is a $j>i$ with $i<^{\ell}j$ and $\ell<p$ (see Proposition \ref{radicalLambdar}). 

We have  $\mathrm{rad}\left(e_{i}\Lambda^{(c)}\right)e_{j}=B_{\ell}e_{i,j}$, where $B_{\ell}=\textsf{F}\oplus \textsf{F}\xi \oplus ...\oplus \textsf{F}\xi ^{\ell-1}$. Take $u\in \mathrm{End}_{\Lambda ^{(c)}}\left(\mathrm{rad}\left(e_{i}\Lambda ^{(c)}\right)\right)$, for all $z\in B_{\ell}$ there is a $\textsf{F}$-linear map $f:B_{\ell}\rightarrow B_{\ell}$ such that $u(ze_{i,j})=f(z)e_{i,j}$. Moreover there is a $g\in \textsf{G}$ such that $u(xe_{i,\text{\large \Fontlukas m}})=gxe_{i,{\text{\large \Fontlukas m}}}$ for all $x\in \textsf{G}$. Then
$$u(ze_{i,j}e_{j,\text{\large \Fontlukas m}})=u(ze_{i,j})e_{j, \text{\large \Fontlukas m}}=f(z)e_{i,\text{\large \Fontlukas m}},$$
on the other hand,
$$u(ze_{i,j}e_{j,\text{\large \Fontlukas m}})=u(ze_{i,\text{\large \Fontlukas m}})=gze_{i,\text{\large \Fontlukas m}}.$$

Therefore $f(z)=gz$, wich implies $g=f_{0}+\cdots +f_{r}\xi ^{r}$ for some $f_0,\ldots, f_r \in \textsf{F}$ and $r< \ell$. Notice that $r>0$ contradicts $gz\in B_\ell$. Then $g\in\textsf{F}$ and 
$$\mathrm{End}_{\Lambda^{(c)}}\left(\mathrm{rad}\left(e_{i}\Lambda ^{(c)}\right)\right)=\textsf{F}.$$ 
That is $\mathrm{rad}\left(e_{i}\Lambda ^{(c)}\right)$ is weak.

For $\mathrm{rad}\left(e_{i}\Lambda ^{(r)}\right)\cong T_{i}^{p}$ with $T_{i}$ indecomposable, $i$ is weak and for all $j>i$, $i<^{p}j$. Then
$$\mathrm{cd}\left(X_{F\left(\mathrm{rad}\left(e_{i}\Lambda ^{(r)}\right)\right)}\right)=p\,\mathrm{cd}\left(X_{F(T_{i})}\right);$$
and 
$$\mathrm{cd}\left(X_{F\left(\mathrm{rad}\left(e_{i}\Lambda ^{(r)}\right)\right)}\right)=\sum _{i< z}n_{z}\mathfrak{e}_{z}+p\,\mathfrak{e}_{0};$$
where $n_{z}=\mathrm{dim}_{\textsf{F}}\left(e_{i}\Lambda ^{(r)}e_{z}\right)/\mathrm{dim}_{\textsf{F}}\left(e_{z}\Lambda ^{(r)}e_{z}\right)$. We have $n_{z}=p$, for any $z>i$. Consequently
$$\mathrm{cd}\left(X_{F(T_{i})}\right)=\sum _{i< z}\mathfrak{e}_{z}+\mathfrak{e}_{0}.$$

Calling $\mathcal{F}$ the set of strong points of $\mathscr{P}$ and $\mathcal{W}$ the set of its weak points
$$s\left(\mathrm{cd}\left(X_{F(T_{i})}\right)\right)=\sum _{\substack{i<z\\ z \in \mathcal{F}}} \mathfrak{e}_{z}+\sum _{\substack{i<z\\ z \in \mathcal{W}}}p\,\mathfrak{e}_{z}+\mathfrak{e}_{0}.$$

Now 
$$\mathrm{cd}\left(X_{F\left(\mathrm{rad}\left(e_{i}\Lambda ^{(c)}\right)\right)}\right)=\sum _{i< z}m_{z}\mathfrak{e}_{z}+\mathfrak{e}_{0};$$
where $m_{z}=\mathrm{dim}_\textsf{F}\left(e_{i}\Lambda ^{(c)}e_{z}\right)/\mathrm{dim}_\textsf{F}\left(e_{z}\Lambda ^{(c)}e_{z}\right)$. We have $i< ^{p}z$ for each $z$ greater than $i$. Then for $z$ weak, $m(z)=p$ and for $z$ strong $m(z)=1$. Therefore
$$s\left(\mathrm{cd}\left(X_{F(T_{i})}\right)\right)=\mathrm{cd}\left(\mathrm{rad}\left(e_{i}\Lambda ^{(c)}\right)\right).$$

If $\mathrm{rad}\left(e_{i}\Lambda ^{(r)}\right)$ is indecomposable, it is weak and
$$\mathrm{cd}\left(X_{F\left(\mathrm{rad} \left(e_{i}\Lambda ^{(r)}\right)\right)}\right)=\sum _{\substack{i<z\\ z \in \mathcal{F}}}p\,\mathfrak{e}_{z}+\sum _{\substack{i<^{\ell}z\\ z \in \mathcal{W}}}\ell\mathfrak{e}_{z}+p\,\mathfrak{e}_{0};$$
then
$$w\left(\mathrm{cd}\left(X_{F\left(\mathrm{rad} \left(e_{i}\Lambda ^{(r)}\right)\right)}\right)\right)=\sum _{\substack{i<z\\ z \in \mathcal{F}}}\mathfrak{e}_{z}+\sum _{\substack{i<^{\ell}z\\ z \in \mathcal{W}}}\ell\mathfrak{e}_{z}+\mathfrak{e}_{0};$$
thus $w\left(\mathrm{cd}\left(X_{F\left(\mathrm{rad}\left(e_{i}\Lambda ^{(r)}\right)\right)}\right)\right)=\mathrm{cd}\left(X_{F\left(\mathrm{rad}\left(e_{i}\Lambda ^{(c)}\right)\right)}\right)$.

When the point $i$ is strong, $\mathrm{rad}\left(e_{i}\Lambda ^{(r)}\right)$ is strong, as before:
$$s\left(\mathrm{cd}\left(X_{F\left(\mathrm{rad}\left(e_{i}\Lambda ^{(r)}\right)\right)}\right)\right)=\mathrm{cd}\left(X_{F\left(\mathrm{rad}\left(e_{i}\Lambda ^{(c)}\right)\right)}\right).$$

The proof is complete.
\end{proof}

\begin{proposition}\label{relatedSequences}
Let $(a):\tau X \rightarrow E \rightarrow X$ be an almost split sequence in $\mathcal{M}^{(r)}$ and
$(b):\tau X ^{\prime }\rightarrow E^{\prime }\rightarrow X^{\prime }$ be an almost split sequence in $\mathcal{M}^{(c)}$, with $X\in \hat{\mathcal{C}}^{(r)}$ and $X^{\prime }\in \hat{\mathcal{C}}^{(c)}$ such that:
\begin{enumerate}
\item The initial object $\tau X$ is strong if and only if $\tau X^{\prime }$ is strong, and  $s\left(\mathrm{cd}(\tau X)\right)=\mathrm{cd}(\tau X^{\prime })$ if $\tau X$ is strong, or $w\left(\mathrm{cd}(\tau X)\right)=\mathrm{cd}(\tau X^{\prime })$ if $\tau X$ is weak.
\item There is a bijective correspondence $\sigma $ between the isomorphism classes of the indecomposable sumands
of $E$ and those of $E^{\prime }$ such that if $W$ is an indecomposable direct summand of $E$ and 
$W^{\prime }$ is a direct summand of $E^{\prime }$ such that the isomorphism class of $W^{\prime }$ is the image by $\sigma $ of
the isomorphism class of $W$, then $s\left(\mathrm{cd}(W)\right)=\mathrm{cd}(W^{\prime })$ if $W$ is strong and
$w\left(\mathrm{cd}(W)\right)=\mathrm{cd}(W^{\prime })$ if $W$ is weak.
\end{enumerate} 

Then $s\left(\mathrm{cd}(X)\right)=\mathrm{cd}(X^{\prime })$ if $X$ is strong, and $w\left(\mathrm{cd}(X)\right)=\mathrm{cd}(X^{\prime })$ if $X$ is weak.
\end{proposition}

\begin{proof}
Let $X_{1},\ldots ,X_{l}, Y_{1},\ldots ,Y_{t}$ be representatives of the isomorphism classes of the indecomposable
direct summands of $E$ with $X_{1},\ldots ,X_{l}$ strong and $Y_{1},\ldots,Y_{t}$ weak and let 
$Z_{1},\ldots ,Z_{l}, W_{1},\ldots ,W_{t}$ representatives of the isomorphism classes of the indecomposable direct summands
of $E^{\prime }$ with $Z_{1},\ldots ,Z_{l}$ strong and $W_{1},\ldots ,W_{t}$ weak such that
$s\left(\mathrm{cd}(X_{i})\right)=\mathrm{cd}(Z_{i})$ for $i\in \{1,\ldots ,l\}$ and $w\left(\mathrm{cd}(Y_{i})\right)=\mathrm{cd}(W_{i})$ for
$i\in \{1,\ldots ,t \}$.

 Suppose $\tau X$ strong, then $\tau X^{\prime }$, $X$ and $X^{\prime }$ are  strong. For each $i\in \{1,\ldots ,l\}$ we have $\mathrm{Irr}(\tau X, Z_{i})$ is one dimensional over $K(\tau X)$ or over $K(Z_i)$, but both $\tau X$ and $Z_i$ are strong, then $K(\tau X)=\textsf{F}$ and $K(Z_i)=\textsf{F}$. Therefore $\mathrm{dim}_{K(Z_i)}\mathrm{Irred}(\tau X,Z_i)=1$.

Now for $i\in \{1,\ldots ,t \}$ we have $K(Y_{i})\cong \textsf{G}$, so $\mathrm{dim}_{K(Y_{i})}\mathrm{Irr}(\tau X, Y_i)=1$. Then $E\cong X_{1}\oplus ...X_{l}\oplus Y_{1}\oplus ...\oplus Y_{t}$. Moreover 
$K(\tau X^{\prime })\cong \textsf{G}$ and $K(Z_{i})\cong \textsf{G}$, therefore 
$\mathrm{dim}_{K(Z_{i})}\mathrm{Irr}(\tau X^{\prime },Z_{i})=1$, and $\mathrm{dim}_{K(W_{i})}\mathrm{Irr}(\tau X^{\prime }, W_{i})=p$, because $K(W_{i})=\textsf{F}$.

Thus 
$E^{\prime }\cong Z_{1}\oplus \cdots \oplus Z_{l}\oplus W_{1}^{p}\oplus \cdots \oplus W_{t}^{p}$. From this we obtain
$$\mathrm{cd}(X)=\sum _{i=1}^{l}\mathrm{cd}(X_{i})+\sum _{i=1}^{t}\mathrm{cd}(Y_{i})-\mathrm{cd}(\tau X);$$
$$\mathrm{cd}(X^{\prime })=\sum _{i=1}^{l}\mathrm{cd}(Z_{i})+\sum _{i=1}^{t}p\,\mathrm{cd}(W_{i})-\mathrm{cd}(\tau X^{\prime }).$$

Now observe that $s=pw$, then
$$s(\mathrm{cd}(X))=\sum _{i=1}^{l}s(\mathrm{cd}(X_{i}))+\sum _{i=1}^{t}s(\mathrm{cd}(Y_{i}))-s(\mathrm{cd}(\tau X))$$
$$=\sum _{i=1}^{l}\mathrm{cd}(Z_{i})+\sum _{i=1}^{t}pw(\mathrm{cd}(Y_{i}))-\mathrm{cd}(\tau X^{\prime})$$
$$=\sum _{i=1}^{l}\mathrm{cd}(Z_{i})+\sum _{i=1}^{t}p\,\mathrm{cd}(W_{i})-\mathrm{cd}(\tau X^{\prime})=\mathrm{cd}(X).$$

When $\tau X$ is weak,
$$\mathrm{dim}_{K(X_{i})}\mathrm{Irr}(\tau X, X_{i})=p,\hspace{5mm}\mathrm{dim}_{K(Y_{i})}\mathrm{Irr}(\tau X, Y_{i})=1;$$
and
$$\mathrm{dim}_{K(Z_{i})}\mathrm{Irr}(\tau X^{\prime },Z_{i})=1,\hspace{5mm}\mathrm{dim}_{K(Z_{i})}\mathrm{Irr}(\tau X^{\prime },Z_{i})=1.$$

In this case $E\cong X_{1}^{p}\oplus\cdots\oplus X_{l}^{p}\oplus Y_{1}\oplus\cdots\oplus Y_{t}$ and
$E^{\prime }=Z_{1}\oplus\cdots\oplus Z_{l}\oplus W_{1}\oplus\cdots\oplus W_{t}.$
Therefore
$$\mathrm{cd}(X)=\sum _{i=1}^{l}p\,\mathrm{cd}(X_{i})+\sum _{i=1}^{t}\mathrm{cd}(Y_{i})-\mathrm{cd}(\tau X);$$
$$\mathrm{cd}(X^{\prime })=\sum _{i=1}^{l}p\,\mathrm{cd}(Z_{i})+\sum _{i=1}^{t}\mathrm{cd}(W_{i})-\mathrm{cd}(\tau X ^{\prime }).$$

From this we obtain :
$$w(\mathrm{cd}(X))=\sum _{i=1}^{l}p\,w(\mathrm{cd}(X_{i}))+\sum _{i=1}^{t}w(\mathrm{cd}(Y_{i}))-w(\mathrm{cd}(\tau X))$$
$$=\sum _{i=1}^{l}s(\mathrm{cd}(X_{i}))+\sum _{i=1}^{t}w(\mathrm{cd}(Y_{i}))-w(\mathrm{cd}(\tau X))$$
$$=\sum _{i=1}^{l}\mathrm{cd}(Z_{i})+\sum _{i=1}^{t}\mathrm{cd}(W_{i})-\mathrm{cd}(\tau X ^{\prime })=\mathrm{cd}(X^{\prime }).$$

The proof is complete.
\end{proof}

Now we can formulate the main result of this section:

\begin{theorem}\label{mainBijection}
For a $p$-equipped poset $\mathscr{P}$, with associated algebras $\Lambda^{(r)}$ and $\Lambda^{(c)}$, there is a  bijection 
$$\kappa :\hat{\mathcal{C}}^{(r)}\rightarrow \hat{\mathcal{C}}^{(c)};$$
which is
an isomorphism between the underlying graphs and has the following properties.
\begin{enumerate}[(i)]
\item $X\in \hat{\mathcal{C}}^{(r)}$ is projective  (injective) if and only if $\kappa (X)$ is projective (injective). Moreover $\kappa \left(X_{F\left(e_{i}\Lambda ^{(r)}\right)}\right)=X_{F\left(e_{i}\Lambda ^{(c)}\right)}$.
\item If $X$ is not projective then $\tau (\kappa (X))=\kappa (\tau X)$.
\item If $X$ is not injective then $\tau ^{-1}\kappa (X)=\kappa (\tau ^{-1}X)$.
\item For any $X\in \hat{\mathcal{C}}^{(r)}$,
$$ \mathrm{cd}(\kappa (X))={\it s }(\mathrm{cd}(X)) \textrm{ if } X \textrm{ is strong};$$
$$ \mathrm{cd}(\kappa (X))={\it w }(\mathrm{cd}(X)) \textrm{ if } X \textrm{ is weak}.$$
\end{enumerate}
\end{theorem}

\begin{proof}
For $X=X_{F\left(e_{i}\Lambda ^{(r)}\right)}$, we define $\kappa (X)=X_{F\left(e_{i}\Lambda ^{(c)}\right)}$.

If $X$ is strong, we have $\mathrm{cd}\left(X_{F\left(e_{i}\Lambda ^{(r)}\right)}\right)=\mathfrak{e}_{i}+\mathfrak{e}_{0}$, then
$$s(\mathrm{cd}(X))=\mathfrak{e}_{i}+\mathfrak{e}_{0}=\mathrm{cd}\left(X_{F\left(e_{i}\Lambda ^{(c)}\right)}\right).$$

In the case $X$ weak, $\mathrm{cd}\left(X_{F\left(e_{i}\Lambda ^{(r)}\right)}\right)=\mathfrak{e}_{i}+p\mathfrak{e}_{0}$ and 
$$w(\mathrm{cd}(X))=\mathfrak{e}_{i}+\mathfrak{e}_{0}=\mathrm{cd}\left(X_{F\left(e_{i}\Lambda ^{(c)}\right)}\right).$$

Thus for $X\in \hat{\mathcal{C}}^{(r)}$ a projective object, condition $(iv)$ holds.

If $X\in \hat{\mathcal{C}}^{(r)}$ is an arbitrary object, there is a non negative number $n(X)$ such that $\tau ^{n(X)}X=\underline{P}(X)$ with
$\underline{P}(X)$ projective. We say that $\kappa $ is well defined in $X$ if 
$\kappa (\underline{P}(X)) \in \hat{\mathcal{C}}^{(c)}$ and $\tau ^{-n(X)}\kappa (\underline{P}(X))$ is defined. In this case
we set 
$$\kappa (X)=\tau ^{-n(X)}\kappa (\underline{P}(X)).$$ 
Observe that if $\kappa $ is well defined in a non-projective object $X\in \hat{\mathcal{C}}^{(r)}$, then $\kappa $ is well  defined in $\tau (X)$ and $\tau \kappa (X)=\kappa (\tau X)$.

By Proposition \ref{KXisomKtauX}, if $\kappa $ is well defined in $X$, then $X$ is strong (weak) if and only if $\kappa (X)$ is strong (weak). Notice that $\kappa $ is well defined in a projective $X=X_{F\left(e_{i}\Lambda ^{(r)}\right)}$ if $\kappa (X)=X_{F\left(e_{i}\Lambda ^{(c)}\right)}\in \hat{\mathcal{C}}^{(c)}$.

In the following to each $X\in \hat{\mathcal{C}}^{(r)}$ we associate $\phi _{X}$, a $\mathbb{Q}$ linear endomorphism of
$\mathbb{Q}^{\mathscr{P}}$, which is equal to $s$ if $X$ is strong and equal to $w$ if $X$ is weak.

We will prove by induction on the order in $\hat{\mathcal{C}}^{(r)}$, that $\kappa $ is well defined in any $X\in \hat{\mathcal{C}}^{(r)}$ and $\mathrm{cd}(\kappa (X))=\phi _{X}(\mathrm{cd}(X))$.

If $X$ is minimal $X=X_{F\left(e_{\text{\large \Fontlukas m}}\Lambda ^{(r)}\right)}$ is projective, $\kappa (X)=X_{F\left(e_{\text{\large \Fontlukas m}}\Lambda ^{(c)}\right)}\in \hat{\mathcal{C}}^{(c)}$. Therefore
$\kappa  $ is well defined in $X$ and $\mathrm{cd}(\kappa (X))=\phi _{X}(\mathrm{cd}(X))$.  Assume  our claim is true for all $X<W$. To prove it for $W$, we need the following four facts:

$ $

$(A)$  Suppose $X\in \hat{\mathcal{C}}^{(r)}$ and $X<W$, we want to prove that $X$ is injective (projective) if and only if
$\kappa (X)$ is injective (projective). 

If $J$ is the injective $T\left(e_{i}\Lambda ^{(r)}\right)$ or the injective $X_{D\left(\Lambda ^{(r)}e_{0}\right)}$ we denote by $\underline{J}$ the injective  $T\left(e_{i}\Lambda ^{(c)}\right)$ in the first case and $X_{D\left(\Lambda ^{(c)}e_{0}\right)}$ in the second case (see Proposition \ref{InjectiveIsom}). Observe that for $\Lambda$ equal to $\Lambda ^{(r)}$ or to $\Lambda ^{(c)}$, we have
$X_{D(\Lambda e_{0})}=X_{F(\mathrm{rad}(e_{0}\Lambda ))}$. Then using  Proposition \ref{radicalCoordinates} we deduce 
the equality $\phi _{J}(\mathrm{cd}(J))=\mathrm{cd}(\underline{J})$. Then if $X=J$ is injective we have
$\phi _{X}(\mathrm{cd}(X))=\mathrm{cd}(\kappa (X))$ and $\phi _{X}(\mathrm{cd}(X))=\mathrm{cd}(\underline{J})$. Therefore by $(ii)$ of Proposition \ref{determineObjectByCoordinates} we get $\kappa (X)\cong \underline{J}$, so $\kappa (X)$ is injective.

Conversely assume $\kappa (X)$ is injective, then by Proposition \ref{InjectiveIsom}, $\kappa (X)\cong \underline{J}$ for some injective $J$ of $\mathcal{M}^{(r)}$. Then $\mathrm{cd}(\kappa (X))=\phi _{X}(\mathrm{cd}(X))=\phi _{X}(\mathrm{cd}(J) )$, so $\mathrm{cd}(X)=\mathrm{cd}(J)$, again by $(ii)$ of Proposition \ref{determineObjectByCoordinates}, $X\cong J$.

In a similar way it is proved that $X$ is projective if and only if $\kappa (X)$ is projective.

$ $

$(B)$   We prove that if $X_{1}\rightarrow X_{2}$ is an irreducible morphism with $X_{1}, X_{2}\in \hat{\mathcal{C}}^{(r)}$ and $X_{2}<W$, then there is an irreducible morphism $\kappa (X_{1})\rightarrow \kappa (X_{2})$. First assume $X_{2}=X_{F\left(e_{i}\Lambda ^{(r)}\right)}$, then $\mathrm{rad}\left(e_{i}\Lambda ^{(r)}\right)\cong Z^{l}$ with $Z$ indecomposable and $l=1$ or $l=p$, so $X_{1}=X_{F(Z)}$. By Proposition \ref{radicalCoordinates}, 
$\phi _{X_{1}}(\mathrm{cd}(X_{1}))=\mathrm{cd}\left(X_{F(\mathrm{rad}\left(e_{i}\Lambda ^{(c)}\right)}\right)$.
 Here by induction hypothesis
$\phi _{X_{1}}\mathrm{cd}(X_{1})=\mathrm{cd}(\kappa (X_{1}))$. 

Then 
$\mathrm{cd}(\kappa (X_{1}))=\mathrm{cd}\left(X_{F\left(\mathrm{rad}\left(e_{i}\Lambda ^{(c)}\right)\right)}\right)$, and by $(iii)$ of Proposition \ref{determineObjectByCoordinates},
$\kappa (X_{1})=X_{F(\mathrm{rad}(e_{i}\Lambda ^{(c)}))}$. Therefore there is an irreducible morphism $\kappa (X_{1})\rightarrow \kappa (X_{2})$.

Now if $X_{2}$ is not projective and $X_{1}$ is projective we have an irreducible morphism 
$\tau (X_{2})\rightarrow X_{1}$. By the above, there is an irreducible morphism
$\kappa (\tau (X_{2}))\rightarrow \kappa (X_{1})$. We have $\tau (\kappa X_{2})=\kappa (\tau (X_{2}))$. Then one obtains an irreducible morphism $\tau \kappa (X_{2})\rightarrow \kappa  (X_{1})$, therefore there
is an irreducible morphism $\kappa (X_{1})\rightarrow \kappa (X_{2})$. Now take $n(X_{1})$ and $n(X_{2})$ such that
$\tau ^{n(X_{1})}X_{1}$ is projective and $\tau ^{n(X_{2})}X_{2}$ is projective. 

If $n(X_{2})\leq n(X_{1})$, there
is an irreducible morphism $\tau ^{n(X_{2})}(X_{1})\rightarrow \tau ^{n(X_{2}})(X_{2})$, by the previous case
there is an irreducible morphism 
$$\tau ^{n(X_{1})}(\kappa (X_{1}))=\kappa (\tau ^{n(X_{1}}(X_{1}))\rightarrow \kappa (\tau ^{n(X_{1})}(X_{2})) 
=\tau ^{n(X_{1})}(\kappa (X_{1}));$$
so we obtain an irreducible morphism $\kappa (X_{1})\rightarrow \kappa (X_{2})$. In case $n(X_{1})<n(X_{2})$, $X_{2}$ is not a projective object, so we have an irreducible morphism $\tau X_{2}\rightarrow X_{1}$ and
$n(X_{1})\leq n(X_{1})-1=n(\tau X_{1})$, so by the above case there is an irreducible morphism
$\tau (\kappa (X_{2}))=\kappa (\tau X_{2})\rightarrow X_{1}$, so we obtain an irreducible morphism
$\kappa (X_{1})\rightarrow \kappa (X_{2})$.

$  $

$(C)$ Suppose  $X\in  \hat{\mathcal{C}}^{(r)}$ with $X<W$ and
$E\rightarrow X$, $E^{\prime }\rightarrow \kappa (X)$ are sink morphisms.  We will prove that $\kappa $ induces
a bijection between the isomorphism classes of the indecomposable summands of $E$ and those of $E^{\prime }$.

Suppose first $X=X_{F\left(e_{i}\Lambda ^{(r)}\right)}$, then $\kappa (X)=X_{F\left(e_{i}\Lambda ^{(c)}\right)}$. We have a
sink morphism $\mathrm{rad}\left(e_{i}\Lambda ^{(r)}\right)\rightarrow e_{i}\Lambda ^{(r)}$. As $F$ is an equivalence, there is a sink morphism $F\left(\mathrm{rad}\left(e_{i}\Lambda ^{(r)}\right)\right)\rightarrow F\left(e_{i}\Lambda ^{(r)}\right)$. By $6$ of Lemma \ref{lemaIdeal}, there is a sink morphism  $X_{F(\mathrm{rad}(e_{i}\Lambda ^{(r)}))}\oplus T\rightarrow X_{F(e_{i}\Lambda ^{(r)})}$ with $T$ of the form $Q\rightarrow 0$. By $1$ of  Lemma \ref{lemaIdeal}, $T=0$. Here
$\mathrm{rad}\left(e_{i}\Lambda ^{(r)}\right)\cong Z^{l}$ with $Z$ an indecomposable $\Lambda ^{(r)}$-module and $l=1$ or
$l=p$. Therefore we have the sink morphism $X_{F(Z)}^{l}\rightarrow X_{F\left(e_{i}\Lambda ^{(r)}\right)}$. Similarly we
have the sink morphism $X_ {F\left(\mathrm{rad}\left(e_{i}\Lambda ^{(c)}\right)\right)}\rightarrow X_{F\left(e_{i}\Lambda ^{(c)}\right)}$.
So in this case $E\cong X_{F(Z)}^{l}$ and $E^{\prime }=X_ {F\left(\mathrm{rad}\left(e_{i}\Lambda ^{(c)}\right)\right)}$. Clearly our statement holds.

When $X$ is not projective, by $(A)$,
$\kappa (X)$ is not projective. Then there are almost split sequences:
$$\tau X\rightarrow E\rightarrow X, \quad  \tau \kappa (X)\rightarrow E^{\prime }\rightarrow \kappa (X).$$
Let $X_{1},\ldots,X_{l},Y_{1},\ldots,Y_{t}$ representatives of the isomorphism classes of the indecomposable direct summands 
of $E$.  We have irreducible morphisms $X_{i}\rightarrow X$, $Y_{j}\rightarrow X$, and then irreducible morphisms
$\kappa (X_{i})\rightarrow \kappa (X)$, $\kappa (Y_{j})\rightarrow \kappa (X)$. Therefore 
$\kappa (X_{1}),\ldots,\kappa (X_{l}), \kappa (Y_{1}),\ldots,\kappa (Y_{t})$ are direct summands of $E^{\prime}$.
If $X$ is strong, then $\kappa (X)$ is strong, in this case we have:
$$E=X_{1}\oplus \cdots\oplus X_{l}\oplus Y_{1}\oplus \cdots\oplus Y_{t}$$
$$E^{\prime }=\kappa (X_{1})\oplus \cdots \oplus \kappa (X_{l})\oplus \kappa (Y_{1})^{p}\oplus \cdots \oplus \kappa (Y_{t})^{p}\oplus E^{\prime }_{0}.$$

From the above we obtain:
$$\sum _{i=1}^{l}\mathrm{cd}(X_{i})+\sum _{i=1}^{t}\mathrm{cd}(Y_{i})=\mathrm{cd}(X)+\mathrm{cd}(\tau X).$$
Applying $s$ to the previous equality we get:
$$\sum _{i=1}^{l}\mathrm{cd}(\kappa (X_{i}))+\sum _{i=1}^{t}p\mathrm{cd}(\kappa (Y_{i}))=\mathrm{cd}(\kappa (X))+\mathrm{cd}(\tau \kappa (X))$$
$$=\sum _{i=1}\mathrm{cd}(\kappa (X_{i}))+\sum _{i=1}^{t}p\mathrm{cd}(\kappa (Y_{i})+\mathrm{cd}(E^{\prime }_{0}).$$
Then we have $\mathrm{cd}(E_{0}^{\prime })=0$, consequently $E_{0}^{\prime }=0$ and we have proved our statement in this case.  If $X$ is weak one proves in a similar way that $\kappa $ induces a bijection between the isomorphism classes of the indecomposable direct summands of $E$ and those of $E^{\prime }$.

$ $

$(D)$    Take source morphisms:
$$\tau W\rightarrow E, \quad \kappa (\tau W)\rightarrow E^{\prime };$$
we want to prove that $\kappa $ induces a bijection between the isomorphism classes of the indecomposable summands of $E$ and those of $E^{\prime }$.

Suppose $X$ is an indecomposable direct summand of $E^{\prime}$, this implies the existence of an
irreducible morphism $\kappa (\tau W)\rightarrow X$. If $X=X_{F\left(e_{i}\Lambda ^{(c)}\right)}$ is a projective object, then
$\kappa (\tau W)=X_{F\left(\mathrm{rad}\left(e_{i}\Lambda ^{(c)}\right)\right)}$. We have $\mathrm{rad}\left(e_{i}\Lambda ^{(r)}\right)=Z^{l}$
for some indecomposable object $Z$, and $l=1$ or $l=p$. By Proposition \ref{radicalCoordinates}, $\phi _{X_{F(Z)}}\left(\mathrm{cd}\left(X_{F(Z)}\right)\right)=\mathrm{cd}(\kappa (\tau W))$, now $\tau W$ is strong (weak) if and only if $X_{F(Z)}$ is strong (weak).
Then $\phi _{X_{F(Z)}}=\phi _{\tau W}$ and $\phi _{\tau W}\left(\mathrm{cd}\left(X_{F(Z)}\right)\right)=\phi _{\tau W}(\mathrm{cd}(\tau W))$. Thus we obtain $\mathrm{cd}\left(X_{F(Z)}\right)=\mathrm{cd}(\tau W)$. From Propositions \ref{radicalLambdar} and \ref{radicalLambdac}, we obtain $\mathrm{End}(X_{F(Z)})=K(X_{F(Z)})\cong K(\tau W)=\mathrm{End}_{\mathcal{M}}(\tau W)$.  By Corollary \ref{corIsomObjects}
$X_{F(Z)}=\tau (W)$, so we have an irreducible morphism $X_{F(Z)}\rightarrow X_{F\left(e_{i}\Lambda ^{(r)}\right)}$, where 
$X_{F\left(e_{i}\Lambda ^{(c)}\right)}$ is a direct summand of $E^{\prime }$ and
$\kappa \left(X_{F\left(e_{i}\Lambda ^{(r)}\right)}\right)=X_{F\left(e_{i}\Lambda ^{(c)}\right)}=X$. 

If $X$ is not projective, therefore we have an irreducible morphism $\tau X \rightarrow \tau \kappa ( W)$.  By $(C)$ there is an irreducible morphism $Z\rightarrow \tau W$ such that $\kappa (Z)=\tau X$.  Here $\kappa (Z)$ is not injective, then by $(A)$,  $Z$ is not injective, so we have an irreducible morphism
$\tau W\rightarrow \tau ^{-1}Z$, we have $\tau ^{-1}Z<W$. By our  induction hypothesis $\tau \kappa (\tau ^{-1}Z)=\kappa ( Z)$, so $\kappa (\tau ^{-1}Z)=\tau ^{-1}\kappa (Z)$. Therefore $\kappa $ induces a bijection between the isomorphism classes of the indecomposable summands of $E$ and those of $E^{\prime}$.

$ $

Now we are ready to prove that $\kappa $ is well defined in $W$. 

Suppose first $W=X_{F\left(e_{i}\Lambda ^{(r)}\right)}$ is a projective object of $\mathcal{M}$. Take $Z$ an indecomposable such that $\mathrm{rad}\left(e_{i}\Lambda ^{(r)}\right)\cong Z^{l}$ with $l=1$ or $l=p$.
We have an irreducible morphism $X_{F(Z)}\rightarrow W$. By our induction hypothesis $\kappa $ is well defined in $X_{F(Z)}$, then we have $\kappa \left(X_{F(Z)}\right)\in \hat{\mathcal{C}}^{(c)}$ and 
$\mathrm{cd}\left(\kappa \left(X_{F(Z)}\right)\right)=\phi _{X_{F(Z)}}\left(\mathrm{cd}\left(X_{F(Z)}\right)\right)$. By Proposition \ref{radicalCoordinates}, $X_{F(Z)}$ is strong (weak) if and only if $X_{F\left(\mathrm{rad}\left(e_{i}\Lambda ^{(c)}\right)\right)}$ is
strong (weak) and $\phi _{X_{F(Z)}}\left(\mathrm{cd}\left(X_{F(Z)}\right)\right)=\mathrm{cd}\left(X_{F\left(\mathrm{rad}\left(e_{i}\Lambda ^{(c)}\right)\right)}\right)$. Therefore:
$$\mathrm{cd}\left(\kappa \left(X_{F(Z)}\right)\right)=\mathrm{cd}\left(X_{F\left(\mathrm{rad}\left(e_{i}\Lambda ^{(c)}\right)\right)}\right). $$
Here $\mathrm{End}_{\mathcal{M}}\left(\kappa \left(X_{F(Z)}\right)\right)\cong \mathrm{End}_{\mathcal{M}}\left(X_{F\left(\mathrm{rad}\left(e_{i}\Lambda ^{(c)}\right)\right)}\right)$ and  $\kappa \left(X_{F(Z)}\right)\in \hat{\mathcal{C}}^{(c)}$, then
$$\mathrm{Ext}_{\mathcal{M}}\left(\kappa \left(X_{F(Z)}\right), \kappa \left(X_{F(Z)}\right)\right)=0.$$ 
Using Corollary \ref{corIsomObjects}, $\kappa \left(X_{F(Z)}\right)=X_{F\left(\mathrm{rad}\left(e_{i}\Lambda ^{(c)}\right)\right)}$, therefore
$X_{F\left(e_{i}\Lambda ^{(c)}\right)}\in \hat{\mathcal{C}}^{(c)}$. We conclude that $\kappa $ is well defined in
$X_{F\left(e_{i}\Lambda ^{(r)}\right)}$ and as we saw  at the beginning of the proof  $\phi _{W}(\mathrm{cd}(W))=\mathrm{cd}(\kappa (W))$.

Suppose $W$ is not projective, then by $(A)$, $\kappa (\tau W)$ is not an injective object of $\mathcal{M}^{(c)}$.

By induction hypothesis $\kappa $ is well defined in $\tau W$, therefore there is a non negative integer
$m$ such that $\tau ^{m}\tau W=\underline{P}(\tau W)$ is a projective object in $\hat{\mathcal{C}}^{(r)}$,
$\kappa (\underline{P}(W))\in \hat{\mathcal{C}}^{(c)}$, $\tau ^{-m}\kappa (\underline{P}(\tau W))$ is defined and
 $\kappa (\tau W)=\tau ^{-m}\kappa (\underline{P}(W)) $. Here $\kappa (\tau W)$ is not injective, so $\tau ^{-1}\kappa (\tau W)=
\kappa ^{-(m+1)}\underline{P}(\tau W)$ is defined. We have $\tau ^{m+1}W=\underline{P}(\tau W)$, then
$\kappa $ is well defined in $W$ and $\kappa (W)=\tau ^{-1}(\kappa (\tau W))$. Therefore we have almost split sequences

$$\tau W \rightarrow E\rightarrow W, \quad \kappa (\tau W)\rightarrow E^{\prime }\rightarrow \kappa (W).$$

By $(D)$, there exists a bijection between the isomorphism classes of the indecomposable direct summands of $E$ and
those of $E^{\prime }$. From Proposition \ref{relatedSequences}, we obtain $\phi _{W}(\mathrm{cd}(W))=\mathrm{cd}(\kappa (W))$. This proves that $\kappa $ is well defined in $W$. 

We deduce from  $(B)$, $(C)$ and $(D)$ that $\kappa $ induces an isomorphism between the underlying graphs of $\hat{\mathcal{C}}^{(r)}$ and $\hat{\mathcal{C}}^{(c)}$. Property $(i)$ follows from $(A)$, 
property $(ii)$ follows from the definition. To prove $(iii)$, take $X$ non injective in $\hat{\mathcal{C}}^{(r)}$, 
take $\tau ^{-1}X$, this is not projective,  so by $(ii)$, $\tau \kappa (\tau ^{-1}X)=\kappa (X)$, therefore,
$\kappa (\tau ^{-1}X)=\tau ^{-1}(\kappa (X))$. Finally $(iv)$ follows from $\kappa $ is well defined in the objects of $\hat{\mathcal{C}}^{(r)}$. The proof of the Theorem is complete.
  \end{proof}


\section{Socle-projective and top-injective modules}

We formulate the results of the previous section for the category $\mathcal{U}$ of the socle projective $\Lambda $-modules and the category $\mathcal{V}$ of the top-injective $\Lambda $-modules, where $\Lambda $ is $\Lambda ^{(r)}$ or $\Lambda ^{(c)}$, associated to a $p$-equipped poset $\mathscr{P}$.

\begin{theorem}\label{sectionTopInj}
Suppose  $\mathcal{C}_{\mathcal{V}}$ is the component of the Auslander-Reiten graph of
$\mathcal{V}$ containing $e_{0}\Lambda /\mathrm{soc}(e_{0}\Lambda )$. Then there is a set $I$ which is the the set
of the natural numbers or $\{1,\ldots ,n\}$ and for each $i\in I$ there is a section $\mathcal{L}_{i}$ of  $\,\mathcal{C}_{\mathcal{V}}$ with the following properties:
\begin{enumerate} 
\item If $N\in \mathcal{L}_{i}$ is not projective, then $i>1$ and $\tau N\in \mathcal{L}_{i-1}$
\item If $N\in \mathcal{L}_{i}$ is not injective, then $\tau ^{-1}N\in \mathcal{L}_{i+1}$.
\item If $Q$ is projective in $\mathcal{L}_{i}$ and there is an irreducible morphism $N\rightarrow Q$, then
$N\in \mathcal{L}_{i}$
\item $\mathcal{L}_{i}\cap \mathcal{L}_{j}=\emptyset $ for $i\neq j$
\item $\mathcal{C}_{\mathcal{V}}=\bigcup _{i\in I}\mathcal{L}_{i}$.
\item There are not oriented cycles in $\mathcal{C}_{\mathcal{V}}$.
\item If $I=\{1,\ldots ,n\}$, then  $\mathcal{C}_{\mathcal{V}}$ is the Auslander-Reiten graph of $\mathcal{V}$. 
\end{enumerate}
\end{theorem}

\begin{proof}
We define 
$$\mathcal{L}_{i}=\{N\in \mathcal{V}| N=\mathrm{Cok}(X)\neq 0, X\in \mathcal{S}_{i}\};$$ 
for each $i\in I$, where $\mathcal{S}_{i}$ are the sections constructed in Theorem \ref{constructComponent}.

We first see that $\mathcal{L}_{i}$ is a section. Indeed if $N\in \mathcal{L}_{i}$, and there is an irreducible
morphism $N\rightarrow N_{1}$, then $N=\mathrm{Cok}(X)$, for some $X\in \mathcal{S}_{i}$. Now $N_{1}\cong \mathrm{Cok}(X_{1})$, where $X_{1}\in \mathcal{M}$, then by Lemma \ref{irreducibleInVandM} there is an irreducible morphism $X\rightarrow X_{1}$, but this implies that
$X_{1}\in \mathcal{S}_{i}$ or $\tau X_{1}\in \mathcal{S}_{i}$. Therefore $N_{1}=\mathrm{Cok}(X_{1})\in \mathcal{L}_{i}$ or $\tau N_{1}=\mathrm{Cok}(\tau X_{1})\in \mathcal{L}_{i}$. Clearly   $\tau N_{1}$ and $N_{1}$ can not be both in $\mathcal{L}_{i}$.

Items $(1)$, $(2)$, $(3)$, $(4)$ and $(5)$ follow from the corresponding items of
Theorem \ref{constructComponent}. Item $(6)$ follows from Propositions \ref{psHereditaryEquivF} and \ref{NoCycles}.

In case $I$ is finite, and $N$ is an indecomposable in $\mathcal{V}$, there is a non zero morphism $e_{0}\Lambda /\mathrm{soc}(e_{0}\Lambda )  \rightarrow N$. We have $e_{0}\Lambda /\mathrm{soc}(e_{0}\Lambda )\in \mathcal{C}_{\mathcal{V}}$. Suppose $N$ is not isomorphic to any module in
$\mathcal{C}_{\mathcal{V}}$.  By $(6)$, there are not oriented cycles in $\mathcal{C}_{\mathcal{V}}$, so we have an order in $\mathcal{C}_{\mathcal{V}}$, where $N<N_{1}$ if there is a chain of irreducible morphisms from $N$ to $N_{1}$. In this order take $Z$ maximal in $\mathcal{C}_{\mathcal{V}}$, such that there is a non zero morphism from $Z$ to $N$. But $N$ is not isomorphic to $Z$, so this morphism factorizes through a source morphism $Z\rightarrow E$. This implies that there is a non zero morphism $N_{1}\rightarrow N$ with $N_{1}\in \mathcal{C}_{\mathcal{V}}$ and
$N_{1}$ direct summand of $E$. Then there is an irreducible morphism $Z\rightarrow N_{1}$, so
$Z<N_{1}$, a contradiction. Therefore $N$ is isomorphic to some module in $\mathcal{C}_{\mathcal{V}}$.
This proves $(7)$, and the proof of our theorem is complete.
\end{proof}

The next result follows from the previous theorem, using the equivalence of categories $F:\mathcal{U}\rightarrow \mathcal{V}$.

\begin{theorem}\label{sectionSocProj}
Let $\mathcal{C}_{\mathcal{U}}$ be the component of the Auslander-Reiten graph of
$\mathcal{U}$ containing $e_{\text{\large \Fontlukas m}}\Lambda $, the simple projective in $\mathrm{mod}\, \Lambda $. Then there is a set $I$ which is the the set of the natural numbers or $\{1,\ldots , n\}$ and for each $i\in I$ there is a section $\mathcal{T}_{i}$ of
$\mathcal{C}_{\mathcal{U}}$ with the following properties:
\begin{enumerate} 
\item If $N\in \mathcal{T}_{i}$ is a not projective, then $i>1$ and $\tau N\in \mathcal{T}_{i-1}$
\item If $N\in \mathcal{T}_{i}$ is not injective, then $\tau ^{-1}N\in \mathcal{T}_{i+1}$.
\item If $P$ is a projective $\Lambda $-module  in $\mathcal{T}_{i}$ and there is an irreducible morphism $M\rightarrow Q$, then
$M\in \mathcal{T}_{i}$
\item $\mathcal{T}_{i}\cap \mathcal{T}_{j}=\emptyset $ for $i\neq j$
\item $\mathcal{C}_{\mathcal{U}}=\bigcup _{i\in I}\mathcal{T}_{i}$.
\item There are not oriented cycles in $\mathcal{C}_{\mathcal{U}}$.
\item If $I=\{1,\ldots ,n\}$, then  $\mathcal{C}_{\mathcal{U}}$ is the Auslander-Reiten graph of $\mathcal{U}$. 
\end{enumerate}
\end{theorem}

\begin{proposition}
If $M\in \mathcal{C}_{\mathcal{V}}$ or in $M\in \mathcal{C}_{\mathcal{U}}$, then
$M$ is strong or weak. Moreover $K(M)=\mathrm{End}_{\Lambda }(M)$, therefore $\mathrm{End}_{\Lambda }(M)$
is isomorphic to $\textsf{F}$ or to $\textsf{G}$.
\end{proposition}

\begin{proof}
If $M\in \mathcal{C}_{\mathcal{V}}$, then $X_{M}\in \hat{\mathcal{C}}$ and $X_{M}$ is strong or weak (see Proposition \ref{objectWS}). By $7.$ of Lemma \ref{lemaIdeal}, $K(M)\cong K(X_{M})$, so $M$ is strong or weak. 

If $M\in \mathcal{C}_{\mathcal{U}}$, then $F(M)\in \mathcal{C}_{\mathcal{V}}$, so $M$ is strong or weak. 

Finally, $\mathrm{End}_{\Lambda }(M)=K(M)$, by $6.$ of Theorems \ref{sectionTopInj} and \ref{sectionSocProj}.
\end{proof}

\begin{lemma}
\begin{enumerate}
\item[(i)] If $N,N_{1}$ lie in $\mathcal{C}_{\mathcal{V}}$, and $N\rightarrow N_{1}$ is an irreducible morphism then
$\mathrm{Irr}(N,N_{1})$ is one-dimensional over $K(N)$ or over $K(N_{1})$.
\item[(ii)] If there is an irreducible morphism $M\rightarrow M_{1}$, where $M,M_{1}\in \mathcal{C}_{\mathcal{U}}$, then
$\mathrm{Irr}(M,M_{1})$ is one dimensional over $K(M)$ or over $K(M_{1})$.
\end{enumerate}
\end{lemma}

\begin{proof}
By Lemma \ref{IsoIrrCok} we have an isomorphism of $\textsf{F}$-vector spaces
$$\mathrm{Irr}(X_N, X_{N_1})\rightarrow \mathrm{Irr}(N, N_1);$$
and the equivalence $F$ of $\textsf{F}$-categories gives isomorphisms of $\textsf{F}$-vector spaces
$$\mathrm{Irr}\left(X_{F(M)}, X_{F(M_1)} \right)\rightarrow \mathrm{Irr}(F(M), F(M_1))\rightarrow 
\mathrm{Irr}(M, M_1).$$ 

Using $7.$ of Lemma \ref{lemaIdeal},
$$\mathrm{dim}_{\textsf{F}}K\left(X_N \right)=\mathrm{dim}_{\textsf{F}}K(N) \text{ and } \mathrm{dim}_{\textsf{F}}K\left(X_{N_1} \right)=\mathrm{dim}_{\textsf{F}}K(N_1);$$
using also $F$
$$\mathrm{dim}_{\textsf{F}}K\left(X_ {F(M)} \right)=\mathrm{dim}_{\textsf{F}}K(F(M))=\mathrm{dim}_{\textsf{F}}K(M);$$
and  
$$\mathrm{dim}_{\textsf{F}}K\left(X_ {F(M_1)} \right)=\mathrm{dim}_{\textsf{F}}K(F(M_1))=\mathrm{dim}_{\textsf{F}}K(M_1).$$

We know that if $\mathrm{Irr}(X_N,X_{N_1})\neq 0$, then $\mathrm{dim}_{\textsf{F}}\mathrm{Irr}(X_N,X_{N_1}) $ is equal to $\mathrm{dim}_{\textsf{F}}K(X_N)$ or to $\mathrm{dim}_{\textsf{F}}K(X_{N_1})$, by Proposition \ref{oneDim}. So $\mathrm{dim}_{\textsf{F}}\mathrm{Irr}(N,N_1)\neq 0$,  implies this is equal to $\mathrm{dim}_{\textsf{F}}K(N)$ or to $\mathrm{dim}_{\textsf{F}}K(N_1)$, proving $(i)$. In the same way if $\mathrm{dim}_{\textsf{F}}\mathrm{Irr}(M,M_1)\neq 0$, then this is equal to $\mathrm{dim}_{\textsf{F}}K(M)$ or to $\mathrm{dim}_{\textsf{F}}K(M_1)$ which proves $(ii)$.
\end{proof}

In the following we denote by $\mathcal{C}_{\mathcal{V}}^{(r)}$, the component of the Auslander-Reiten graph of $\mathcal{V}^{(r)}$ containing $D\left(\Lambda ^{(r)}e_{0}\right)$.  Denote by $\mathcal{C}^{(r)}_{\mathcal{U}}$, the component of the Auslander-Reiten graph of $\mathcal{U}^{(r)}$ containing $e_{\text{\large \Fontlukas m}}\Lambda ^{(r)}$. Similarly, we have the corresponding graphs $\mathcal{C}_{\mathcal{V}}^{(c)}$ and $\mathcal{C}_{\mathcal{U}}^{(c)}$
when $\Lambda =\Lambda ^{(c)}$.
 
Take now $\kappa :\hat{\mathcal{C}}^{(r)}\rightarrow \hat{\mathcal{C}}^{(c)}$ from Theorem \ref{mainBijection}. For any
$X\in \hat{\mathcal{C}}^{(r)}$ we have $\mathrm{cd}(\kappa (X))=\phi _{X}(\mathrm{cd}(X))$, where
$\phi _{X}=s$ if $X$ is strong and $\phi _{X}=w$ in  case $X$ is weak. We have that $X$ is of the form $T(e_{i}\Lambda ^{(r)})$
if and only if $\kappa (X)$ is of the form $T(e_{i}\Lambda ^{(c)})$, for some $i\in \mathscr{P}$. Therefore $\mathrm{Cok}(X)=0$ if and only if
$\mathrm{Cok}(\kappa (X))=0$ for $X\in \hat{\mathcal{C}}^{(r)}$. Then $\kappa$ induces a bijection
of graphs:
$$\beta : \mathcal{C}_{\mathcal{V}}^{(r)}\rightarrow \mathcal{C}_{\mathcal{V}}^{(c)};$$
such that $\beta (\mathrm{Cok}(X))=\mathrm{Cok}(\kappa (X))$. Now using the equivalence of categories
$F:\mathcal{U}\rightarrow \mathcal{V}$, there is a bijection of graphs:
$$\alpha :\mathcal{C}_{\mathcal{U}}^{(r)}\rightarrow \mathcal{C}_{\mathcal{U}}^{(c)};$$
such that for $M\in \mathcal{C}_{\mathcal{U}}^{(r)}$ 
$$F(\alpha (M))=\mathrm{Cok}\left(\kappa \left(X_{F(M)}\right)\right)=\beta  (F(M)).$$

\begin{definition}
Suppose $\Lambda $ is $\Lambda ^{(r)}$ or $\Lambda ^{(c)}$, then for
$N\in \mathcal{V}$ we define $\mathrm{cd}(N)=\mathrm{cd}(X_{N})$ and for
$M\in \mathcal{U}$, we define $\mathrm{cd}(M)=\mathrm{cd}(X_{F(M)})$. For $M$ a $\Lambda $-module
we define: $\underline{\mathrm{dim}}(M) \in \mathbb{Q}^{\mathscr{P}}$ as the function such that for $i\in \mathscr{P}$,
$\underline{\mathrm{dim}}(M)(i)=\mathrm{dim}_{e_{i}\Lambda e_{i}}e_{i}M$, and $\underline{\mathrm{dim}}_{F}(M)(i)=\mathrm{dim}_{F}(e_{i}M)$.
\end{definition}

Observe that for $M\in \Lambda ^{(r)}\, \mathrm{mod}$ one has $\underline{\mathrm{dim}}_{F}(M)=s(\underline{\mathrm{dim}}(M))$,
and for $M^{\prime }\in \Lambda ^{(c)}\, \mathrm{mod}$, we have 
$\underline{\mathrm{dim}}_{F}(M^{\prime })=w^{-1}(\underline{\mathrm{dim}}(M^{\prime }))$.

\begin{theorem}\label{BijectionU}
The function 
$$\alpha :\mathcal{C}_{\mathcal{U}}^{(r)}\rightarrow \mathcal{C}_{\mathcal{U}}^{(c)};$$
is an isomorphism of the underlying graphs with the following properties:
\begin{enumerate}
\item For $i\in \mathscr{P}$, $i\neq 0$, $e_{i}\Lambda ^{(r)}\in \mathcal{C}_{\mathcal{U}}^{(r)}$ if and only if $e_{i}\Lambda ^{(c)}\in \mathcal{C}_{\mathcal{U}}^{(c)}$. In this case
$\alpha \left(e_{i}\Lambda ^{(r)}\right)=e_{i}\Lambda ^{(c)}$.

\item If $M\in \mathcal{C}_{\mathcal{U}}^{(r)}$ is not projective, $\alpha (M)$ is not projective and $\tau (\alpha (M))=\alpha (\tau (M))$.

\item Let $M$ be in $\mathcal{C}_{\mathcal{U}}^{(r)}$, then
$$\mathrm{cd}(\alpha (M))=\phi _{M}(\mathrm{cd}(M));$$
$$\underline{\mathrm{dim}}(\alpha (M))=\phi _{M}(\underline{\mathrm{dim}}(M));$$
where $\phi _{M}=s$ for $M$ strong and $\phi _{M}=w$ for $M$ weak.

\item For every $M\in \mathcal{C}_{\mathcal{U}}^{(r)}$, if $M$ is strong
$$\underline{\mathrm{dim}}_{F}(\alpha (M))=w^{-1}(\underline{\mathrm{dim}}_{F}(M));$$
if $M$ is weak:
$$\underline{\mathrm{dim}}_{F}(\alpha (M))=s^{-1}(\underline{\mathrm{dim}}_{F}(M)).$$
\end{enumerate}
\end{theorem}

\begin{proof}
Take $M, M_{1}\in \mathcal{C}_{\mathcal{U}}^{(r)}$, then there is an irreducible morphism
$M\rightarrow M_{1}$ if and only if there is an irreducible morphism $F(M)\rightarrow F(M_{1})$. By Lemma \ref{irreducibleInVandM}
there is an irreducible morphism $F(M)\rightarrow F(M_{1})$  if and only if there is an irreducible morphism $X_{F(M)}\rightarrow X_{F(M_{1})}$, and
by Theorem \ref{mainBijection}, this occurs if and only if there is an irreducible morphism 
$\kappa \left(X_{F(M)}\right)\rightarrow \kappa \left(X_{F(M_{1})}\right)$. Again by Lemma \ref{irreducibleInVandM}, there is an irreducible morphism
$\kappa \left(X_{F(M)}\right)\rightarrow \kappa \left(X_{F(M_{1})}\right)$ if and only if there is an irreducible morphism
$\mathrm{Cok}\left(\kappa \left(X_{F(M)}\right)\right)\rightarrow \mathrm{Cok}\left(\kappa \left(X_{F(M_{1})}\right)\right).$ By definition
$F(\alpha (M))=\mathrm{Cok}\left(\kappa \left(X_{F(M)}\right)\right)$ and $F(\alpha (M_{1}))=\mathrm{Cok}\left(\kappa \left(X_{F(M_{1})}\right)\right)$, therefore there is an irreducible morphism $M\rightarrow M_{1}$ if and only if there is an irreducible morphism
$\alpha (M)\rightarrow \alpha (M_{1})$. This proves that $\alpha $ is an isomorphism between the corresponding underlying graphs. 

Now, let us prove the properties.

\begin{enumerate}
\item If $e_{i}\Lambda ^{(r)}\in \mathcal{C}_{\mathcal{U}}^{(r)}$,  
$$F\left(\alpha \left(e_{i}\Lambda ^{(r)}\right)\right)=\mathrm{Cok}\left(\kappa \left(X_{F\left(e_{i}\Lambda ^{(r)}\right)}\right)\right)=\mathrm{Cok}\left(X_{F\left(e_{i}\Lambda ^{(c)}\right)}\right)=F\left(e_{i}\Lambda ^{(c)}\right);$$
therefore $\alpha \left(e_{i}\Lambda ^{(r)}\right)=e_{i}\Lambda ^{(c)}$. Conversely, for every $e_{i}\Lambda ^{(c)}\in \mathcal{C}^{(c)}_{\mathcal{U}}$, we have $X_{F\left(e_{i}\Lambda ^{(c)}\right)}\in \hat{\mathcal{C}}^{(c)}$. By $(i)$ of Theorem \ref{mainBijection} there is a projective $X_{F\left(e_{j}\Lambda ^{(r)}\right)}\in \hat{\mathcal{C}}^{(r)}$ such that $\kappa \left(X_{F\left(e_{j}\Lambda ^{(r)}\right)}\right)
=X_{F\left(e_{j}\Lambda ^{(c)}\right)}=X_{F\left(e_{i}\Lambda ^{(c)}\right)}$. Therefore
$j=i$ and $X_{F\left(e_{i}\Lambda ^{(c)}\right)}\in \hat{\mathcal{C}}^{(c)}$ implies $e_{i}\Lambda ^{(c)}\in \mathcal{C}_{\mathcal{U}}^{(c)}$. 

\item Follows from $(ii)$ of Theorem \ref{mainBijection}.

\item The first part follows from $(iv)$ of Theorem \ref{mainBijection}. To prove the second part consider an object of the form $e_{i}\Lambda ^{(r)}$. For $\Lambda $ equal to $\Lambda ^{(r)}$ or to $\Lambda ^{(c)}$, we have $\underline{\mathrm{dim}}(e_{i}\Lambda )(j)=\mathrm{dim}_{e_{j}\Lambda e_{j}}(e_{i}\Lambda e_{j})$. Then if $i$ is strong and $j$ strong 
$$\underline{\mathrm{dim}}\left(e_{i}\Lambda ^{(r)}\right)(j)=1, \quad \underline{\mathrm{dim}}\left(e_{i}\Lambda ^{(c)}\right)(j)=1;$$
if $i$ is strong and $j$ weak
$$\underline{\mathrm{dim}}\left(e_{i}\Lambda ^{(r)}\right)(j)=1, \quad \underline{\mathrm{dim}}\left(e_{i}\Lambda ^{(c)}\right)(j)=p.$$
Therefore if $e_{i}\Lambda ^{(r)}$ is strong  
$$\underline{\mathrm{dim}}\left(e_{i}\Lambda ^{(c)}\right)=s\left(\underline{\mathrm{dim}}\left(e_{i}\Lambda ^{(r)}\right)\right).$$
Now if $i$ is weak and $j$ strong
$$\underline{\mathrm{dim}}\left(e_{i}\Lambda ^{(r)}\right)(j)=p, \quad \underline{\mathrm{dim}}\left(e_{i}\Lambda ^{(c)}\right)(j)=1;$$
if $i$ is weak and $j$ weak with $i<^{l}j$, 
$$\underline{\mathrm{dim}}\left(e_{i}\Lambda ^{(r)}\right)(j)=l, \quad \underline{\mathrm{dim}}\left(e_{i}\Lambda ^{(c)}\right)(j)=l.$$
Then, for $e_{i}\Lambda ^{(r)}$ weak:
$$\underline{\mathrm{dim}}\left(e_{i}\Lambda ^{(c)}\right)=w\left(\underline{\mathrm{dim}}\left(e_{i}\Lambda ^{(r)}\right)\right).$$

Now we will prove the second part of our item by induction on the order in $\mathcal{C}^{(r)}_{\mathcal{U}}.$

If $M$ is a minimal element, then $M$ is the simple projective of $\mathrm{mod}\, \Lambda ^{(r)}$ and 
$\alpha (M) $ is the simple projective in $\mathrm{mod}\, \Lambda ^{(c)}$. Then our result follows from above.

Suppose the result is proved for all $M_{1}<M$. If $M=e_{i}\Lambda ^{(r)}$,
then $\alpha (M)=e_{i}\Lambda ^{(c)}$ and our result is already proved. So suppose $M$ is not projective, this implies that $\alpha (M)$ is not projective, then we have almost split sequences:
$$0\rightarrow \tau M\rightarrow E\rightarrow M\rightarrow 0;$$
$$0\rightarrow \tau (\alpha (M))\rightarrow E^{\prime }\rightarrow \alpha (M)\rightarrow 0.$$

Now $M$ is strong (weak) if and only if $\alpha (M)$ is strong (weak). If $M$ is strong

$$E=\bigoplus _{i=1}^{t}M_{i}\oplus \bigoplus _{i=1}^{z}U_{i}, \quad E^{\prime }=\bigoplus _{i=1}^{t}\alpha (M_{i})\oplus \bigoplus _{i=1}^{z}\alpha (U_{i})^{p};$$
where $M_{1},\ldots , M_{t}$ are strong indecomposable $\Lambda ^{(r)}$-modules and $U_{1},\ldots ,U_{z}$ are weak indecomposable $\Lambda ^{(r)}$-modules, $\alpha (M_{1}),\ldots ,\alpha (M_{t})$ are strong indecomposable
$\Lambda^{(c)}$-modules and $\alpha (U_{1}),\ldots ,\alpha (U_{z})$ are weak indecomposable $\Lambda ^{(c)}$-modules.

In the case $M$ is weak then $\alpha (M)$ is weak and we have:
$$E=\bigoplus _{i=1}^{t}M_{i}^{p}\oplus \bigoplus _{i=1}^{z}U_{i}, \quad E^{\prime }=\bigoplus _{i=1}^{t}\alpha (M_{i}) \oplus \bigoplus _{i=1}^{z}\alpha (U_{i});$$
where as before $M_{1},\ldots ,M_{t}$ are strong indecomposable $\Lambda ^{(r)}$-modules and $U_{1},\ldots ,U_{z}$ are weak indecomposable modules $\Lambda ^{(r)}$-modules.

For every $i\in \{1,\ldots, t\}$ and $j\in \{1,\ldots, z\}$ we have that $\tau (M), M_i$ and $U_j$ are smaller than $M$. Then as in the proof of Theorem \ref{mainBijection},  if $M$ is strong 
$$\underline{\mathrm{dim}}(\alpha (M))=\sum _{i=1}^{t}\underline{\mathrm{dim}}(\alpha (M_{i})+\sum _{i=1}^{z}p\,\underline{\mathrm{dim}}(\alpha (U_{i}))-\underline{\mathrm{dim}}(\alpha (\tau (M))$$
$$=\sum _{i=1}^{t}s(\underline{\mathrm{dim}}(M_{i}))+\sum _{i=1}^{z}p\,w(\underline{\mathrm{dim}}(U_{i}))-s(\underline{\mathrm{dim}}(\tau (M)))$$
$$=\sum _{i=1}^{t}s(\underline{\mathrm{dim}}(M_{i}))+\sum _{i=1}^{z}s(\underline{\mathrm{dim}}(U_{i}))-s(\underline{\mathrm{dim}}(\tau (M)))$$
$$=s(\underline{\mathrm{dim}}(M)).$$

In the case $M$ is weak we have:
$$\underline{\mathrm{dim}}(\alpha (M))=\sum _{i=1}^{t}\underline{\mathrm{dim}}(\alpha (M_{i}))+\sum _{i=1}^{z}\underline{\mathrm{dim}}(\alpha (U_{i}))-\underline{\mathrm{dim}}(\alpha (\tau (M)))$$
$$=\sum _{i=1}^{t}s(\underline{\mathrm{dim}}(M_{i}))+\sum _{i=1}^{z}w(\underline{\mathrm{dim}}(U_{i}))-w(\underline{\mathrm{dim}}(\tau (M)))$$
$$=\sum _{i=1}^{t}p\,w(\underline{\mathrm{dim}}(M_{i}))+\sum _{i=1}^{z}w(\underline{\mathrm{dim}}(U_{i}))-w(\underline{\mathrm{dim}}(\tau (M)))$$
$$=w(\underline{\mathrm{dim}}(M)).$$
The item is proved.

\item Take $M\in \mathcal{C}_{\mathcal{U}}^{(r)}$, then if $M$ is strong:
$$\underline{\mathrm{dim}}_{F}(\alpha (M))=w^{-1}(\underline{\mathrm{dim}}(\alpha (M)))=w^{-1}(s(\underline{\mathrm{dim}}(M)))=w^{-1}(\underline{\mathrm{dim}}_{F}(M)).$$

If $M$ is weak:
$$\underline{\mathrm{dim}}_{F}(\alpha (M))=w^{-1}(\underline{\mathrm{dim}}(\alpha (M)))=\underline{\mathrm{dim}}(M)=s^{-1}(\underline{\mathrm{dim}}_{F}(M)).$$
\end{enumerate}
The proof is complete.
\end{proof}

For $\mathcal{C}_{\mathcal{V}}^{(r)}$ we have a similar Theorem.

\begin{theorem}
The function 
$$\beta :\mathcal{C}_{\mathcal{V}}^{(r)}\rightarrow \mathcal{C}_{\mathcal{V}}^{(c)};$$
is an isomorphism of the underlying graphs with the following properties:

\begin{enumerate}
\item For $i\in \mathscr{P}$, $i\neq 0$, we have $F\left(e_{i}\Lambda ^{(r)}\right)\in \mathcal{C}_{\mathcal{V}}^{(r)}$ if and only if
$F\left(e_{i}\Lambda ^{(c)}\right)\in \mathcal{C}_{\mathcal{V}}^{(c)}$. In this case
$\beta \left(F\left(e_{i}\Lambda ^{(r)}\right)\right)=F\left(e_{i}\Lambda ^{(c)}\right)$.

\item Suppose $N\in \mathcal{C}_{\mathcal{V}}^{(r)}$, then if $N$ is not projective, $\beta  (N)$ is not projective and $\tau (\beta  (N))=\beta (\tau (N))$.

\item Let $N$ be in $\mathcal{C}_{\mathcal{V}}^{(r)}$, then
$$\mathrm{cd}(\beta (N))=\phi _{N}(\mathrm{cd}(N));$$
$$\underline{\mathrm{dim}}(\beta (N))=\phi _{N}(\underline{\mathrm{dim}}(N));$$
where $\phi _{N}=s$ for $N$ strong and $\phi _{N}=w$ for $N$ weak.

\item Suppose $N\in \mathcal{C}_{\mathcal{V}}^{(r)}$, then if $N$ is strong
$$\underline{\mathrm{dim}}_{F}(\beta  (N))=w^{-1}(\underline{\mathrm{dim}}_{F}(N));$$
if $N$ is weak:
$$\underline{\mathrm{dim}}_{F}(\beta (N))=s^{-1}(\underline{\mathrm{dim}}_{F}(N)).$$
\end{enumerate}
\end{theorem}

\begin{proof}
Here if $N\in \mathcal{C}_{\mathcal{V}}^{(r)}$, then $N=F(M)$ for some $M\in \mathcal{C}_{\mathcal{U}}^{(r)}$ and $F(\alpha (M))=\beta (F(M))$. From here we conclude that $\beta $ is an isomorphism between the corresponding underlying graphs, with the properties 1., 2., and the first part of 3.

To prove the second part of $(3)$, take $M\in \mathcal{C}_{\mathcal{U}}^{(r)}$ such that $N=F(M)$. We
have exact sequences:
$$0\rightarrow M\rightarrow \left(e_{0}\Lambda ^{(r)}\right)^{\mathrm{cd}(M)(0)}\rightarrow F(M)\rightarrow 0;$$
$$0\rightarrow \alpha (M)\rightarrow \left(e_{0}\Lambda ^{(c)}\right)^{\mathrm{cd}(\alpha (M))(0)}\rightarrow F(\alpha (M))\rightarrow 0.$$
Then
$$\underline{\mathrm{dim}}\beta (F(M))=\underline{\mathrm{dim}}(F(\alpha (M)))=$$
$$\mathrm{cd}(\alpha (M))(0)\underline{\mathrm{dim}}\left(e_{0}\Lambda ^{(c)}\right)-\underline{\mathrm{dim}}(\alpha (M)).$$

If $N=F(M)$ is strong, $M$ is strong and:
$$\underline{\mathrm{dim}}(F(\alpha (M)))=\mathrm{cd}(\alpha (M))(0)\underline{\mathrm{dim}}\left(e_{0}\Lambda ^{(c)}\right)-\underline{\mathrm{dim}}(\alpha (M))$$
$$=\mathrm{cd}(M)(0)s\left(\underline{\mathrm{dim}}\left(e_{0}\Lambda ^{(r)}\right)\right)-s(\underline{\mathrm{dim}}(M))=s(\underline{\mathrm{dim}}(F(M))).$$

If $N=F(M)$ is weak:
$$\underline{\mathrm{dim}}(F(\alpha (M)))=\mathrm{cd}(\alpha (M))(0)\underline{\mathrm{dim}}\left(e_{0}\Lambda ^{(c)}\right)-\underline{\mathrm{dim}}(\alpha (M))$$
$$=\mathrm{cd}(M)(0)(1/p)s\left(\underline{\mathrm{dim}}\left(e_{0}\Lambda ^{(r)}\right)\right)-w(\underline{\mathrm{dim}}(M))=w(\underline{\mathrm{dim}}(F(M))).$$

The proof of $4.$ is the same as the proof of $4.$ of Theorem \ref{BijectionU}.

The proof is complete.
\end{proof}





\end{document}